\documentclass{article}

\usepackage{titletoc}
\usepackage{hyperref}       

\usepackage[numbers]{natbib}
\usepackage[nonatbib,preprint]{neurips_2019}
\usepackage{geometry}

\usepackage[utf8]{inputenc} 
\usepackage[T1]{fontenc}    
\usepackage{url}            
\usepackage{booktabs}       
\usepackage{amsfonts}       
\usepackage{nicefrac}       
\usepackage{microtype}      
\usepackage{algorithmic}
\usepackage{xcolor}
\usepackage[]{algorithm2e}
\usepackage{capt-of}

\usepackage{amsmath,amsfonts,amssymb,amsthm}
\usepackage{graphicx}
\usepackage{microtype}

\usepackage{subfigure}
\usepackage{booktabs} 
\usepackage{bbold}

\usepackage{color}

\newcommand{\pnm}{\mathcal{P}_{[N-1]}^{(1)}}
\newcommand{\pnp}{\mathcal{P}_{[N]}^{(1)}}
\newcommand{\pim}{\mathcal{P}_{[i-1]}^{(1)}}
\newcommand{\pip}{\mathcal{P}_{[i]}^{(1)}}
\newcommand{\yex}{y^{\text{extr}}}
\newcommand{\cl}{c^{\lambda}}
\newcommand{\cheb}{\mathcal{C}^{\tau,\kappa}}

\newcommand{\chebpoly}{\mathcal{T}^{\tau,\kappa}}

\title{Online Regularized Nonlinear Acceleration}

\author{%
  Damien Scieur \\
  Department of Computer Science\\
  Princeton University, Princeton, USA\\
  \texttt{dscieur@cs.princeton.edu} \\
   \And
     Edouard Oyallon\\
  CVN \& INRIA\\
  CentraleSupélec, Gif-sur-Yvette, France\\
  \texttt{edouard.oyallon@centralesupelec.fr} \\
   \And
    Alexandre d'Aspremont \\
    CNRS \& D.I., UMR 8548,\\
    École Normale Supérieure, Paris, France. \\
  \texttt{aspremon@di.ens.fr} \\
   \And
     Francis Bach \\
    INRIA \& D.I., UMR 8548,\\
    École Normale Supérieure, Paris, France. \\
  \texttt{francis.bach@inria.fr} \\
}
\newtheorem{theorem}{Theorem}[section]
\newtheorem{proposition}[theorem]{Proposition}
\newtheorem{definition}[theorem]{Definition}

\renewenvironment{proof}{\textbf{Proof.}}{\QED\bigskip}

 \numberwithin{dummy}{section}

\definecolor{ddarkbrown}{rgb}{0.5,0.2,0.05} \definecolor{bbluegray}{rgb}{0.05,0,0.5}

\newcommand{\BEAS}{\begin{eqnarray*}}
\newcommand{\EEAS}{\end{eqnarray*}}
\newcommand{\BEA}{\begin{eqnarray}}
\newcommand{\EEA}{\end{eqnarray}}
\newcommand{\BEQ}{\begin{equation}}
\newcommand{\EEQ}{\end{equation}}
\newcommand{\BIT}{\begin{itemize}}
\newcommand{\EIT}{\end{itemize}}
\newcommand{\BNUM}{\begin{enumerate}}
\newcommand{\ENUM}{\end{enumerate}}

\newcommand{\BA}{\begin{array}}
\newcommand{\EA}{\end{array}}

\newcommand{\reals}{{\mathbb R}}

\newcommand{\symm}{{\mbox{\bf S}}}  

\newcommand{\QED}{~~\rule[-1pt]{6pt}{6pt}}

\newcommand{\argmin}{\mathop{\rm argmin}}

\newif\ifisarxiv
\begin{document}
\isarxivtrue

\maketitle

\begin{abstract}
Regularized nonlinear acceleration (RNA) estimates the minimum of a function by post-processing iterates from an algorithm such as the gradient method. It can be seen as a regularized version of Anderson acceleration, a classical acceleration scheme from numerical analysis. The new scheme provably improves the rate of convergence of fixed step gradient descent, and its empirical performance is comparable to that of quasi-Newton methods. However, RNA cannot accelerate faster multistep algorithms like Nesterov's method and often diverges in this context. Here, we adapt RNA to overcome these issues, so that our scheme can be used on fast algorithms such as gradient methods with momentum. We show optimal complexity bounds for quadratics and asymptotically optimal rates on general convex minimization problems. Moreover, this new scheme works online, i.e.~extrapolated solution estimates can be reinjected at each iteration, significantly improving numerical performance over classical accelerated methods. 
\end{abstract}

\section{Introduction}

First-order methods are popular tools to minimize high dimensional functions, as their memory requirements and cost per iteration are relatively low. The convergence rate of these methods, however, can be relatively slow in some cases. Consider for example, a $L-$smooth, $\mu-$strongly convex function $f$. The gradient descent scheme with fixed step size $h$ reads
\BEQ
    x_{i+1} = x_i - h\nabla f(x_i), \label{eq:gradient_step}
\EEQ
where $x_i\in\reals^d$. If $h = {1}/{L}$, we have the following bound
\BEQ
    \textstyle f(x_i)-f(x^*) \leq O\big(\left(1-\frac{\mu}{L}\right)^i (f(x_0)-f(x^*))\big), \label{eq:conv_gradient_method}
\EEQ
where $\kappa ={\mu}/{L}$ is usually referred to as the inverse condition number \citep{nesterov2013introductory}. The number of iteration to reach a precision $\epsilon$ is thus, in theory, proportional to $\kappa^{-1}$ and bad conditioning has indeed a significant impact on empirical performance. This can be partly fixed using the fast gradient method, an optimal accelerated first-order method derived by \citet{nesterov1983method}, which matches the lower complexity bound~\cite{nesterov2013introductory}. This algorithm maintains two sequences, and reads
\BEQ \label{eq:nesterov_algo}
        \textstyle x_{i+1} = y_i - h\nabla f(y_i), \qquad  y_{i+1}  \textstyle= (1+\beta) x_{i+1} - \beta x_{i}, \quad \beta = \frac{1-\sqrt{\kappa}}{1+\sqrt{\kappa}}.
\EEQ
It has an improved convergence rate compared to gradient methods. If $h = 1/L$, its convergence is bounded by
\BEQ
    \textstyle f(x_i)-f(x^*) \leq O\big(\left(1-\sqrt{\kappa}\right)^i (f(x_0)-f(x^*))\big),
    \label{eq:rate_nesterov}
\EEQ
hence the fast gradient method is roughly $\sqrt{\kappa^{-1}}$ faster than the gradient method. Unfortunately, the algorithm requires knowledge of the parameter $\kappa$, hence that of $\mu$. This strong convexity constant is typically hard to estimate, and obtaining an accurate online estimation is still an open problem (see~\cite{fercoq2016restarting} for more details).

Quasi-Newton methods form another class of accelerated first-order methods. These algorithms approximate the Newton step by estimating $H\approx \nabla^{-2} f(x_i)$, then perform the step
\[
    x_{i+1} = x_{i} - h H \nabla f(x_i),
\]
where $h$ is the step-size.  The most heavily used scheme of this kind is L-BFGS, which maintains a symmetric positive definite $H$ using efficient rank-two updates. \cite{nocedal2006nonlinear} presents some result about its rate of convergence, but those are very conservative compared to their gradient descent counterpart. In practice, L-BFGS is particularly efficient even in non-convex and non-smooth settings \cite{lewis2009nonsmooth}.

Recently, \citet{scieur2016regularized} proposed a generic way to accelerate an iterative algorithm derived from Anderson acceleration \citep{anderson1965iterative}, which is a classical acceleration scheme from numerical analysis. Contrary to standard accelerated optimization schemes, this algorithm does not require information on problem-dependent parameters. It mixes the approaches of Nesterov and quasi-Newton, as it linearly combines the points $x_j$ using coefficients computed from previous iterates (hence the nonlinear name). \citet{scieur2016regularized,scieur2017nonlinear} show that RNA produces solutions asymptotically reaching the optimal convergence rate of \citep{nesterov2013introductory}, using only iterates from the fixed-step gradient method in~\eqref{eq:gradient_step}. 

Unfortunately there are big limitations for the RNA algorithms. First, their theoretical analysis was limited to basically gradient descent. Practically, it was unable to accelerate other algorithms. For instance, accelerating Nesterov's method usually leads to a method with a slower emprirical rate of convergence. The second limitation was the way RNA accelerates algorithms. In \cite{scieur2016regularized}, the RNA algorithm was used in a restart fashion, waiting $N$ iterates from gradient descent, then restarting the gradient algorithm from an extrapolated point. This raises the two following questions:
\begin{enumerate}
    \item Can we derive a generic acceleration algorithm that works when accelerating methods with momentum (such as Nesterov Method)?
    \item Is it possible to design an acceleration algorithm whose acceleration step is done \textit{online}, i.e., after each iteration?
\end{enumerate}

\paragraph{Contributions.} 
This work makes the following contributions.

\begin{itemize}
    \item We adapt the RNA scheme from \citep{scieur2016regularized} to handle multi-step algorithms. The previous version was limited to basically gradient descent which precluded its application to faster momentum methods for example.
    \item We analyse the convergence and stability of our algorithm. For quadratic functions its rate is similar to conjugate gradient, and we show asymptotically optimal convergence rates for optimizing non-quadratic or stochastic problems.
    \item We also combine RNA and Nesterov's method in a way that ensures the optimal convergence rate on smooth and (strongly)-convex functions, while being adaptive to the smooth and strong convexity parameters.
    \item We perform extensive numerical experiments on convex and nonconvex problems, showing that our approach in competitive with state-of-the-art methods such as Nesterov's fast gradient or BFGS. We also show the scalability of Online RNA with some experiments on convolutional neural networks (CNNs) with millions of parameters.
\end{itemize}

\subsection{Related work}

Our result improves the RNA algorithm, originally derived from the minimal polynomial extrapolation method \cite{cabay1976polynomial} and directly related to Anderson acceleration. In short, both methods focus on accelerating the following class of algorithm, written
\BEQ\label{eq:rna-up}
    x_{i+1} = g(y_i), \quad y_{i+1} = x_{i+1},
\EEQ
where $g(x)$ is typically the gradient step $\eqref{eq:gradient_step}$. The augmented notation with $y$ will be useful later for the generalization of this class.

After $N$ iterations, RNA or Anderson acceleration computes coefficients $c$ to combine the iterates following $\textstyle \yex = \sum_{i=1}^N c_i x_i$, and the extrapolated solution $\yex$ usually converges faster to $x^*$ than $x_N$ as $N$ grows. This uses the information of the past $N$ iterates, contained in the matrices $X$ and $Y$ defined below,
\[
    X = [x_1\ldots x_N], \quad Y = [y_0\ldots y_{N-1}].
\]
The biggest difference between the two algorithms comes from the formula for the $c_i's$. Consider the residual $r(x)$ and the residual matrix $R$
\BEQ
    r(x) \triangleq x-g(x), \qquad R \triangleq [r(y_0),\, r(y_1),\, \ldots,\, r(y_N)] \;=\; Y-X. \label{eq:residual}
\EEQ
Using those notations, RNA solves for $\lambda > 0$
\[
    \textstyle c := \argmin_c \| Rc \|^2 + \lambda \|c\|^2, \quad \sum_{i=1}^N c_i = 1 \qquad (\text{where } R \triangleq X-Y),
\]
while Anderson use the same formula in the particular case where $\lambda = 0$. There is a closed-form formula for the minimization problem (cf. Algorithm \ref{algo:rna}) with complexity bounded by $O(N^2d)$ where $N$ is usually around $10$. In practice, the additionnal computation time required by RNA is negligible compared to gradient steps for medium or large-scale problems.

Recent results \cite{scieur2016regularized,scieur2017nonlinear} produce theoretical convergence guarantees for linear, nonlinear or stochastic functions $g$ using a perturbation argument on the quadratic case where $g$ is affine. For instance, RNA on fixed step gradient descent is asymptotically optimal, matching the complexity bound of \cite{nesterov2013introductory}. There is no such proof for Anderson acceleration, as the algorithm is shown to be unstable in some cases \cite{scieur2016regularized}. For the special case where we apply RNA to a nonsymmetric algorithm (i.e., the Jacobian of $g(x)$ is not symmetric), \citet{bollapragada2019nonlinear} derived a bound that shows acceleration in the convergence. However, their results do not show an optimal rate for the class of algorithm in~\eqref{eq:general_iteration}.

\section{Nonlinear Acceleration of Linear Algorithms}

We consider the following class of algorithms, written
\BEA \label{eq:general_iteration}
    x_{i} = g(y_{i-1})\;; \qquad
    y_i = \textstyle \sum_{j=1}^i \alpha_j^{(i)} x_j + \beta_j^{(i)} y_{j-1},
\EEA
where $x_i,y_i\in\reals^d$ and $g: \reals^d \to \reals^d$ is an iterative update, potentially stochastic. For example, $g(x)$ can be a gradient step with fixed stepsize as in~\eqref{eq:gradient_step}, with $g(x)=x - h\nabla f(x)$. We assume the following condition on the coefficients $\alpha$ and $\beta$, that ensure the consistency of the algorithm \citep{scieur2017integration},
\[
    \textbf{1}^T(\alpha + \beta) = 1, \quad \forall k, \ \alpha_j \neq 0.
\]
Nesterov's fast gradient method for instance belongs to this class of algorithms. We can write these updates in matrix format, with
\BEQ\label{eq:def_xy}
    X_i = [x_1,x_2,\ldots, x_i], \quad  Y_i = [y_0,y_1,\ldots, y_{i-1}]. 
\EEQ
Using this notation, \eqref{eq:general_iteration} reads (assuming $x_0=y_0$)
\BEA\label{eq:general_iteration_matrix}
    X_i = g(Y_{i-1})\,, \qquad Y_i = [x_0, X_{i}]L_i,
\EEA
where $g(Y)$ stands for $[g(y_0),g(y_1),\ldots,g(y_{i-1})]$ and the matrix $L_i$ is upper-triangular of size $i\times i$ with nonzero diagonal coefficients, whose columns sum to one. The matrix $L_i$ is constructed iteratively, following the recurrence
\BEQ
    L_i = \begin{bmatrix} L_{i-1} & \alpha_{[1:i-1]} + L_{i-1}\beta \\
    0_{1\times i-1} & \alpha_i
    \end{bmatrix}, \quad L_0 = 1.\label{eq:recurence_L}
\EEQ
In short, $L_i$ gathers coefficients from the linear combination in~\eqref{eq:general_iteration}. This matrix, together with $g$, characterizes the algorithm.

\subsection{Linear Algorithms}
In this section, we focus on iterative algorithms $g$ that are linear, i.e., where
\BEQ
    g(x) = G(x-x^*) + x^*. \label{eq:linear_g}
\EEQ
Here, $x^*$ is a fixed point of $g$, which corresponds in optimization to the minimum of an objective function. We first treat the case where $g(x)$ is linear, as the nonlinear case will then be extended as a perturbation of the linear one. 

Its worth mentionning that \eqref{eq:linear_g} is equivalent to $Ax+b$. We suppose $G$ symmetric. The case where $G$ is nonsymmetric is studied in \citep{bollapragada2019nonlinear}, but it leads to a more complicated analysis. We write $\pnp$, the set of all polynomials $p$ whose degree is \textit{exactly} $N$ (i.e., the leading coefficient is nonzero), and whose coefficients sum to one. More formally,
\BEQ
    \pnp = \{ p \in \reals[x]: \deg(p) = N,\,  p(1) = 1  \}.
\EEQ
The following proposition extends a result by \citet{scieur2016regularized} showing that \eqref{eq:general_iteration} can be written using polynomials in $\pnp$. \ifisarxiv\else Its proof can be found in Appendix \ref{prop:poly_iter_proof}.\fi
\begin{proposition} \label{prop:poly_iter}
Let $g$ be the linear function \eqref{eq:linear_g}. Then, the $N$-th iteration of \eqref{eq:general_iteration} is equivalent to
\BEAS
    x_N = x^* + G(y_{N-1}-x^*), \qquad y_N = x^* + p_N(G)(x_0-x^*),\quad \mbox{for some $p_N\in \pnp$.}
\EEAS
\end{proposition}
\ifisarxiv\begin{proof}
We prove the result iteratively. Of course, at iteration 0,
    \[
        y_0 = x^* + 1 \cdot (x_0 - x^*),
    \]
    and $1$ is indeed  polynomial of degree $0$ whose coefficient sum to one. Now, assume
    \[
        y_{i-1} - x^* = p_{i-1}(G)(x_0-x^*), \quad p_{i-1} \in \pim.
    \]
    We now show that
    \[
        y_{i} - x^* = p_{i}(G)(x_0-x^*), \quad p_{i} \in \pip.
    \]
    By definition of $y_{i}$ \eqref{eq:general_iteration},
    \[
        y_{i} - x^* = \textstyle \sum_{j=1}^i \alpha_j^{(i)} x_j + \beta_j^{(i)} y_{i-1} - x^*,
    \]
    where $(\alpha + \beta)^T\textbf{1} = 1$. This also means that
    \[
        y_{i} - x^* = \textstyle \sum_{j=1}^i \alpha_j^{(i)} (x_j-x^*) + \beta_j^{(i)} (y_{j-1}-x^*) .
    \]
    By definition, $x_{j} -x^* = G(y_{j-1} - x^*) $, so
    \[
        y_{i} - x^* = \textstyle \sum_{j=1}^i \left(\alpha_j^{(i)} G  + \beta_j^{(i)} I\right) (y_{j-1}-x^*) .
    \]
    By the assumption of recurrence,
    \[
        y_{i} - x^* = \textstyle \sum_{j=1}^i \left(\alpha_j^{(i)} G  + \beta_j^{(i)} I\right) p_{j-1}(G) (x_0-x^*).
    \]
    This is a linear combination of polynomials, thus $y_{i} - x^* = p(G)(x_0-x^*)$. It remains to show that $p \in \pip$. Indeed,
    \[
        \deg (p) = \max_j   \max \left\{(1+\deg (p_{j-1}(G))) 1_{\alpha_j \neq 0},\;\; \deg (p_{j-1}(G)) 1_{\beta_j \neq 0}\right\},
    \]
    where $1_{\alpha_j \neq 0} = 1$ if $\alpha_j \neq 0$ and $0$ otherwise. By assumption, $\alpha_i \neq 0$ thus
    \[
        \deg (p) \geq 1+\deg (p_{i-1}(G)) = i.
    \]
    Since $p$ is a  linear combination of polynomials of degree at most $i$, \[
        \deg (p) = i.
    \]
    It remains to show that $p(1)=1$. Indeed,
    \[
        p(1) = \sum_{j=1}^i \left(\alpha_j^{(i)} 1  + \beta_j^{(i)} \right) p_{j-1}(1).
    \]
    Since $\left(\alpha_j^{(i)} 1  + \beta_j^{(i)} \right) = 1$ and $p_{j-1}(1) = 1$, $p(1) = 1$ and this prove the proposition.
\end{proof}\fi

\subsection{Offline Acceleration}
We now propose a modification of RNA that can accelerate any algorithm of the form \eqref{eq:general_iteration} by combining the approaches of \cite{anderson1965iterative} and \cite{scieur2016regularized}. We introduce a mixing parameter~$\beta$, as in Anderson acceleration. Throughout this paper, the \textbf{RNA} algorithm now refers to Algorithm \ref{algo:rna} below. To simplify the theoretical analysis, we also introduce \emph{constrained} nonlinear acceleration \textbf{CNA} as Algorithm~\ref{algo:cna}.

Algorithms \ref{algo:rna} and~\ref{algo:cna} both lead to an equivalent solution (Section \ref{sec:equivalence_algo}), using different trade-offs. RNA is computationally more efficient as we have a closed-form formula for its solution, while CNA leads to a simpler convergence analysis. The two algorithms are in fact equivalent, because the solution of one can be deduced from the other. A more detailed discution can be found in \ifisarxiv Section \else Appendix \fi \ref{sec:equivalence_algo}.

\begin{algorithm}[htb]
\caption{Regularized Nonlinear Acceleration (\textbf{RNA})}
\label{algo:rna}
\fbox{\parbox{0.95\linewidth}{
\begin{algorithmic}[1]
   \STATE {\bfseries Data:} Matrices $X$ and $Y$ of size $d\times N$ constructed from the iterates as in~\eqref{eq:general_iteration} and~\eqref{eq:def_xy}.
   \STATE {\bfseries Parameters:} Mixing $\beta\neq 0$, regularization $\lambda \geq 0$.\\
   \hrulefill
   \STATE \textbf{1.} Compute matrix of residuals $R = X-Y$.
   \STATE \textbf{2.} Solve
   \BEQ
        \textstyle \cl = \frac{(R^TR+(\lambda\|R\|^2_2) I)^{-1} \textbf{1}_N}{\textbf{1}_N^T(R^TR+(\lambda\|R\|^2_2) I)^{-1}\textbf{1}_N}. \label{eq:cl}
    \EEQ
    \STATE \textbf{3.} Compute extrapolated solution $\yex = (Y-\beta R)\cl$.
\end{algorithmic}
}}
\end{algorithm}

\begin{algorithm}[htb]
   \caption{Constrained Nonlinear Acceleration (\textbf{CNA})}
    \label{algo:cna}
\fbox{\parbox{0.95\linewidth}{
\begin{algorithmic}
   \STATE {\bfseries Data:} Matrices $X$ and $Y$ of size $d\times N$ constructed from the iterates as in~\eqref{eq:general_iteration} and~\eqref{eq:def_xy}.
   \STATE {\bfseries Parameters:} Mixing $\beta\neq 0$, constraint $\tau \geq 0$.\\
   \hrulefill
   \STATE \textbf{1.} Compute matrix of residuals $R = X-Y$.
   \STATE \textbf{2.} Solve
   \BEQ
        \textstyle  c^{(\tau)} = \argmin_{c:c^T\textbf{1} = 1} \|Rc\|_2 \quad \text{s.t. } \; \|c\|_2\leq {\textstyle \frac{1+\tau}{\sqrt{N}}} \label{eq:ctau}
    \EEQ
    \STATE \textbf{3.} Compute extrapolated solution $\yex = (Y-\beta R)c^{(\tau)}$.
\end{algorithmic}
}}
\end{algorithm}

\textbf{Computational Complexity of RNA.} \citet{scieur2016regularized} discuss the complexity of Algorithm \ref{algo:rna} in the case where $N$ is small (compared to $d$). When the algorithm is used once on $X$ and $Y$ (batch acceleration), the computational complexity is $O(N^2 d)$, because we have to multiply $R^T$ and $R$. However, when Algorithm \eqref{algo:rna} accelerates the iterates on-the-fly, the matrix $R^TR$ can be updated using only $O(Nd)$ operations. The complexity of solving the linear system is negligible as it takes only $O(N^3)$ operation. Even if the cubic dependence is bad for large $N$, in our experiment $N$ is typically equal to 10, thus adding a negligible computational overhead compared to the computation of a gradient in large dimension which is higher by orders. This is confirmed by our numerical experiments in Section \ref{sec:num_experiment}.

\ifisarxiv
\subsection{Equivalence Between Constrained \& Regularized Nonlinear Acceleration}\label{sec:equivalence_algo}
\fi

The parameters $\lambda$ in Algorithm \ref{algo:rna} and $\tau$ in Algorithm \ref{algo:cna} play similar roles. High values of $\lambda$ give coefficients close to simple averaging, and $\lambda = 0$ retrieves Anderson Acceleration. We have the same behavior when $\tau = 0$ or $\tau = \infty$. We can jump from one algorithm to the other using dual variables, since~\eqref{eq:cl} is the Lagrangian relaxation of the convex problem \eqref{eq:ctau}. This means that, for all values of $\tau$ there exists $\lambda = \lambda(\tau)$ that achieves $c^{\lambda} = c^{(\tau)}$. In fact, we can retrieve $\tau$ from the solution $c^{\lambda}$ by solving
\[
    \textstyle \frac{1+\tau}{\sqrt{N}} = \|c^{\lambda}\|_2.
\]
Conversely, to retrieve $\lambda$ from $c^{(\tau)}$, it suffices to solve
\BEQ
    \left\| \frac{(R^TR+(\lambda\|R\|^2_2) I)^{-1} \textbf{1}_N}{\textbf{1}_N^T(R^TR+(\lambda\|R\|^2_2) I)^{-1}\textbf{1}_N} \right\|^2 = \frac{(1+\tau)^2}{N}, \label{eq:nonlinear_equation}
\EEQ
assuming the constraint in \eqref{eq:ctau} tight, otherwise $\lambda = 0$. Because the norm in \eqref{eq:nonlinear_equation} is increasing with $\lambda$, a binary search or one-dimensional Newton methods gives the solution in a few iterations.

The next proposition bounds the norm of the coefficients of Algorithm~\ref{algo:rna} with an expression similar to~\eqref{eq:ctau}.
\begin{proposition}
    The norm of $c^{\lambda}$ from \eqref{eq:cl} is bounded by
    \BEQ
        \textstyle \|c^{\lambda}\|_2 \leq \frac{1}{\sqrt{N}}\sqrt{1+\frac{1}{\lambda}} \label{eq:bound_norm_lambda}
    \EEQ
\end{proposition}
\begin{proof}
    See \citet{scieur2016regularized}, (Proposition 3.2).
\end{proof}

Having established the equivalence between constrained and regularized nonlinear acceleration, the next section discusses the computational complexity of RNA.

\subsection{Optimal Convergence Rate}
We now analyze the convergence rate of Algorithm \ref{algo:rna} with $\lambda = 0$, or equivalently Algorithm \ref{algo:cna} with $\tau = \infty$, which corresponds to Anderson acceleration \cite{anderson1965iterative}. In particular, we show its optimal rate of convergence when $g$ is a linear function. In the context of optimization, this is equivalent to the application of gradient descent for minimizing quadratics. Using this special structure, the iterations produce a sequence of polynomials (see Proposition \ref{prop:poly_iter}) and the next theorem uses this special property to prove the optimal convergence rate. Compared to previous work in this vein, \cite{scieur2016regularized,scieur2017nonlinear} where the results only apply to algorithm of the form $x_{i+1} = g(x_i)$, this theorem applies to \textit{any} algorithm of the class \eqref{eq:general_iteration}.

\begin{theorem} \label{thm:optimal_rate}

Let $X$, $Y$ in~\eqref{eq:def_xy} be formed using iterates from \eqref{eq:general_iteration}. Let $g$ be the linear function in~\eqref{eq:linear_g} where
    \BEAS
        G\in \symm_d, \quad 0 \preceq G \preceq (1-\kappa)I, \quad 0<\kappa<1,
    \EEAS
    where $\symm_d$ is the set of symmetric matrices of size $d\times d$. The norm of the residual~\eqref{eq:residual} of the extrapolated solution $\yex$ produced by Algorithm \ref{algo:rna} with $\lambda = 0$, or equivalently by Algorithm \ref{algo:cna} with $\tau = \infty$,  is bounded by
    \[
        \|r(\yex)\|_2 \leq \| I-\beta (G-I)  \|_2 ~\| p^*_{N-1}(G)r(x_0)  \|_2,
    \]
    where $p^*_{N-1}$ solves
    \[
        \textstyle p^*_{N-1} = \argmin_{p\in \pnm} \| p(G)r(x_0)  \|_2.
    \]
    We have the following bound on $p^*_{N-1}$,
    \[
    \| p^*_{N-1}(G)r(x_0)  \|_2 \leq
    \begin{cases}
        \left(\frac{1-\sqrt{\kappa}}{1+\sqrt{\kappa}}\right)^{N-1} \|r(x_0)\|_2,& N\leq d, \\
        0,&  N > d.
    \end{cases}
    \]
    which control the convergence rate.
\end{theorem}
\begin{proof}
First, we write the definition of $\yex$ from Algorithm \ref{algo:rna} when $\lambda = 0$,
    \[
        \yex -x^* = (Y-\beta R) c - x^*.
    \]
    Since $c^T\textbf{1} = 1$, we have $X^*c = x^*$, where $X^* = [x^*,x^*,\ldots, x^*,]$. Thus,
    \[
        \yex -x^* = (Y-X^*-\beta R) c.
    \]
    Since $R = G(Y-X^*)$,
    \[
        \yex - x^* = (I-\beta (G-I))(Y-X^*) c .
    \]
    We have seen that the columns of $Y-X^*$ are polynomials of different degrees, whose coefficients sums to one (Proposition \ref{prop:poly_iter}). This means
    \[
        \yex - x^* = (I-\beta (G-I))\sum_{i=0}^{N-1} c_i p_i(G)(x_0-x^*). 
    \]
    In addition, its residual reads
    \[
        r(\yex) = (G-I)(\yex - x^*) = (G-I)(I-\beta (G-I))\sum_{i=0}^{N-1} p_i(G)(x_0-x^*) = (G-\beta I)\sum_{i=0}^{N-1} c_i p_i(G)r(x_0).
    \]
    Its norm is thus bounded by
    \[
        \|r(\yex)\| \leq \|I-\beta (G-I)\| \|\underbrace{\sum_{i=0}^{N-1} c_i p_i(G)r(x_0)}_{=Rc}\|.
    \]
    By definition of $c$ from Algorithm \ref{algo:rna},
    \[
        \|r(\yex)\| \leq \|I-\beta (G-I)\| \min_{c:c^T\textbf{1} = 1} \|\sum_{i=0}^{N-1} c_i p_i(G)r(x_0)\|.
    \]
    Because $p_i$ are all of degree exactly equal to $i$, the $p_i$ are a basis of the set of all polynomial of degree at most $N-1$. In addition, because $p_i(1) = 1$, restricting the sum of coefficients $c_i$ to $1$ generates the set $\pnm$. We have thus
    \[
        \|r(\yex)\| \leq \|I-\beta (G-I)\| \min_{p\in \pnm } \|p(G)r_0\|.
    \]
    This prove the first statement of the Theorem. The second statement is proved in \cite{golub1961chebyshev}, which shows that
    \[
        \min_{p\in \pnm } \max_{G: 0\preceq G \preceq 1-\kappa} \|p(G)\| \leq \left(\frac{1-\sqrt{\kappa}}{1+\sqrt{\kappa}}\right)^{N-1} .
    \]
    when $N\leq d$. Finally, when $N>d$, it suffice to take the minimal polynomial of the matrix $G$ named $p_{\min,G}$, whose coefficient are normalized by $p_{\min,G}(1)$. Since the eigenvalues of $G$ are strictly inferior to $1$, $p_{\min,G}(1)$ cannot be zero. 
\end{proof}

This rate is optimal according to \cite{nesterov2013introductory}, and is actually similar to bounds on Krylov methods (like GMRES or conjugate gradients \cite{golub1961chebyshev,golub2012matrix}) for quadratic minimization. In optimization, the quantity $\|r(\yex)\|_2$ is propotional to the norm of the gradient of the objective function computed at $\yex$.

\subsection{Online Acceleration}
We now discuss the convergence of online acceleration, coupling iterates in $g$ with the extrapolation Algorithm~\ref{algo:rna} at each iteration when $\lambda = 0$. Equivalently, all the following results also hold for Algorithm~\ref{algo:cna} when $\tau = \infty$. Consider the online acceleration strategy,
\BEA \label{eq:online_rna}
    x_{N} = g(y_{N-1}),\qquad y_N = \textbf{RNA}(X,Y,\lambda,\beta),
\EEA
where $\textbf{RNA}(X,Y,\lambda,\beta)=\yex$ with $\yex$ the output of Algorithm~\ref{algo:rna}. By construction, $\yex$ is written
\[
    \textstyle \yex = \sum_{i=1}^N \cl_i (y_{i-1} - \beta (x_i-y_{i-1})).
\]
If $\cl_N\neq 0$ then $\yex$ matches \eqref{eq:general_iteration}, thus online acceleration iterates in~\eqref{eq:online_rna} belong to the class of algorithms in~\eqref{eq:general_iteration}. If we can ensure $\cl_N \neq 0$, applying Theorem~\ref{thm:optimal_rate} recursively will then show an optimal rate of convergence for online acceleration iterations in~\eqref{eq:online_rna}. We show that under mild conditions this property is satisfied (see Proposition \ref{prop:online_accel_structure} in Appendix \ref{prop:online_accel_structure_proof}). 

Thus, with our modification of the original RNA method in \citep{scieur2016regularized} and with our theoretical result, we have that \textit{RNA can accelerate iterates coming from RNA}. In numerical experiments, we will see that this new approach significantly improves empirical performance.

\subsection{Constrained Chebyshev Polynomial}
The previous results consider the special cases where $\lambda = 0$ or $\tau = 0$,  which means that $\|c\|$ is unbounded. However, \citet{scieur2016regularized} show instability issues when $\|c\|$ is not controlled. Regularization is thus required in practice to make the method more robust to perturbations, even in the quadratic case (e.g., round-off errors). Unfortunately, this section will show that robustness comes at the cost of a potentially slower rate of convergence.

We first introduce \textit{constrained Chebyshev polynomials}. Earlier work in \citep{scieur2016regularized} considered regularized Chebyshev polynomials, but using a constrained formulation significantly simplifies the convergence analysis here. This polynomial plays an important role in Section~\ref{sec:convergence_rate_cna} in the convergence analysis.

\begin{definition}
The {Constrained Chebyshev Polynomial} $\chebpoly_N(x)$ of degree $N$ solves, for $\tau \geq 0$,
\BEQ
    \mbox{minimize} \max_{x\in [0,\kappa]} p(x)  \quad \text{s.t.}  ~  \|p\|_2 \leq {\textstyle \frac{1+\tau}{\sqrt{1+N}}} \label{eq:constrained_cheby} 
\EEQ
in the variable ${p\in \pnp}$. Its optimum value is written $\cheb_N \triangleq \max_{x\in [0,\kappa]} \chebpoly_N(x).$
\end{definition}

\subsection{Convergence Rate of CNA} \label{sec:convergence_rate_cna}
The previous section introduced constrained Chebyshev polynomials, which play an essential role in our convergence results when $g$ is nonlinear and/or iterates~\eqref{eq:general_iteration} are noisy. Instead of analyzing Algorithm \ref{algo:rna} directly, we focus on its constrained counterpart, Algorithm \ref{algo:cna}. 

\begin{proposition} \label{prop:rate_constrained}
Let $X$, $Y$ \eqref{eq:def_xy} be build using iterates from \eqref{eq:general_iteration} where $g$ is linear \eqref{eq:linear_g} and satisfies the assumptions of Theorem \eqref{thm:optimal_rate}. Then, the norm of the residual \eqref{eq:residual} of the extrapolation produced by Algorithm \ref{algo:cna} is bounded by
\BEQ
    \|r(\yex)\|_2 \leq \| I-\beta (G-I)  \|_2 \|r(x_0)  \|_2\; \cheb_{N-1},
\EEQ
where $\tau \geq 0$ and $\cheb_N$ is defined in \eqref{eq:constrained_cheby}.
\end{proposition}
\begin{proof}
    The proof is similar to the one of Theorem \ref{thm:optimal_rate}\ifisarxiv\else, in Appendix \ref{thm:optimal_rate_proof}\fi. It suffices to use the constrained Chebyshev polynomial rather than the rescaled Chebyshev polynomial from \cite{golub1961chebyshev}.
\end{proof}

Proposition \ref{prop:rate_constrained} with $\tau = \infty$ gives the same result than Theorem \ref{thm:optimal_rate}. However, smaller values of $\tau$ give weaker results as $\cheb_{N-1}$ increases. However, smaller values of $\tau$ also reduce the norm of coefficients $c^{(\tau)}$ \eqref{eq:ctau}, which makes the algorithm more robust to noise. 

Using the constrained algorithm in the context of non-perturbed linear function $g$ yields no theoretical benefit, but the bounds on the extrapolated coefficients simplify the analysis of perturbed non-linear optimization schemes as we will see below.

\section{Nonlinear Acceleration with Perturbations} \label{sec:rna_nonlinear}

In this section, we analyze the convergence rate of Algorithm \ref{algo:cna} for simplicity, but the results also hold for Algorithm \ref{algo:rna} (cf. Section \ref{sec:equivalence_algo}). We first introduce the concept of pertubed linear iteration, then we analyse the convergence rate of RNA in this setting.

\textbf{Perturbed Linear Iterations.}
Consider the following perturbed scheme,
\BEA \label{eq:perturbed_iteration_matrix}
    \tilde X_i = X^* + G(\tilde Y_{i-1}-X^*) + E_i, \qquad \tilde Y_i = [x_0,\tilde X_i] L_i,
\EEA
where $\tilde X_i$ and $\tilde Y_i$ are formed as in~\eqref{eq:def_xy} using the perturbed iterates $\tilde x_i$ and $\tilde y_i$, and $L_i$ is constructed using \eqref{eq:recurence_L}, and we write $E_i = [e_1,e_2,\ldots,e_i]$. For now, we do not assume anything on $e_i$ or $E_i$. This class contains many schemes such as gradient descent on nonlinear functions, stochastic gradient descent or even Nesterov's fast gradient with backtracking line search for example.

The notation \eqref{eq:perturbed_iteration_matrix} makes the analysis simpler than in \cite{scieur2016regularized,scieur2017nonlinear}, as we have the explicit form of the error over time. Consider the perturbation matrix $P_i$,
\BEQ
    P_i \triangleq \tilde R_i - R_i, \label{eq:perturbation_matrix}
\EEQ
Proposition \ref{prop:explicit_formula_perturbation} \ifisarxiv\else(in Appendix \ref{prop:explicit_formula_perturbation_proof})\fi shows that the magnitude of the perturbations $\|P_i\|$ is proportionnal to the noise matrix $\|E_i\|$, i.e., $\|P_i\| = O(\|E_i\|)$.
\ifisarxiv
\begin{proposition} \label{prop:explicit_formula_perturbation}
Let $P_i$ be defined in \eqref{eq:perturbation_matrix}, where $(\tilde X_i, \tilde Y_i)$ and $(\tilde X_i, \tilde Y_i)$ are formed respectively by \eqref{eq:general_iteration_matrix} and \eqref{eq:perturbed_iteration_matrix}. Let $\bar L_{j} = \| L_1\|_2  \| L_{2}\|_2 \ldots  \|L_j \|_2$. Then, we have the following bound
    \[
        \|P_i\| \leq \textstyle\|E_i\| (1+\sum_{j=1}^{i} (1-\kappa)^j \bar L_{j}).
    \]
\end{proposition}
\begin{proof}
    First, we start with the definition of $R$ and $\tilde R$. Indeed,
    \[
        \tilde R_i - R_i = \tilde X_i-X_i - (\tilde Y_{i-1} - Y_{i-1}). 
    \]
    By definition,
    \[
        \tilde X_i - X_i = G(\tilde Y_{i-1}-X^*) + X^* + E_i -  G( Y_{i-1}-X^*) - X^* = G(\tilde Y_{i-1}-Y_{i-1}) + E_i
    \]
    On the other side,
    \[
        \tilde Y_{i-1} - Y_{i-1} = [0;\tilde X_{i-1}-X_{i-1}]L_{i-1}
    \]
    We thus have
    \BEAS
        P_i & =  & \tilde X_i-X_i - (\tilde Y_{i-1} - Y_{i-1}),\\
        & = & G(\tilde Y_{i-1}-Y_{i-1}) + E_i - [0;\tilde X_{i-1}-X_{i-1}]L_{i-1},\\
        & = & G( [0;\tilde X_{i-1}-X_{i-1}]L_{i-1}) + E_i - [0;G(\tilde Y_{i-2}-Y_{i-2}) + E_{i-1}]L_{i-1},\\
        & = & G [0;P_{i-1}]L_{i-1} + E_i - [0;E_{i-1}]L_{i-1}.\\
    \EEAS
    Thus,
    \[
        P_i - E_i = G [0;P_{i-1}]L_{i-1} - [0;E_{i-1}]L_{i-1} = G [0;P_{i-1}-E_{i-1}]L_{i-1} + (G-I)[0;E_{i-1}]L_{i-1}.
    \]
    Finally, it suffice to expand
    \BEAS
        \|P_i - E_i\| = \| G \| \|P_{i-1}-E_{i-1}\|\|L_{i-1}\| + \|G-I\|\|E_{i-1}\|\|L_{i-1}\|,
    \EEAS
    and use $\|G-I\| \leq 1$ to have the desired result.
\end{proof}
\fi

\textbf{Convergence Analysis.} Now, we analyze how close the output of Algorithm \ref{algo:cna} is to $x^*$. To do so, we compare scheme \eqref{eq:perturbed_iteration_matrix} to its perturbation-free counterpart \eqref{eq:general_iteration_matrix}. Both schemes have the same starting point $x_0$ and ``fixed point''~$x^*$. It is important to note that scheme~\eqref{eq:perturbed_iteration_matrix} may not converge due to noise. The next theorem bounds the accuracy of the output of CNA\ifisarxiv\else (the proof can be found in App.~\ref{thm:convergence_perturbation_proof})\fi.
\begin{theorem} \label{thm:convergence_perturbation}
    Let $\yex$ be the output of Algorithm \eqref{algo:cna} applied on \eqref{eq:perturbed_iteration_matrix}. Its accuracy is bounded by
    \BEAS
        \|(G-I) \left(\yex - x^*\right)\|
        \leq \|I-\beta(G-I) \| \Big(\underbrace{ \cheb_{N-1} \|(G-I)(x_0-x^*)\|}_{\textbf{acceleration}}
        + \underbrace{\textstyle \frac{1+\tau}{\sqrt{N}} \big( \|P_N\| + \|E_N\|\big)}_{\textbf{stability}}\Big).
    \EEAS
\end{theorem}
\ifisarxiv
\begin{proof}
We start with the following expression for arbitrary coefficients $c$ that sum to one,
\[
    (G-I) \left((\tilde Y - \beta \tilde R)c - x^*\right).
\]
Since
\[
    \tilde R = \tilde X - \tilde Y = (G-I)(\tilde Y - X^*) + E,
\]
we have
\[
    (G-I)(\tilde Y-X^*)  = (\tilde R-E).
\]
So,
\[
    (G-I) (\tilde Y-X^* - \beta \tilde R)c = (\tilde R-E)c - \beta (G-I)\tilde Rc .
\]
After rearranging the terms we get
\BEQ
(G-I) \left((\tilde Y - \beta \tilde R)c - x^*\right) = (I-\beta(G-I))\tilde Rc - E c.\label{eq:decomposition_error}
\EEQ
We bound \eqref{eq:decomposition_error} as follow, using coefficients from \eqref{eq:ctau},
\[
    \|I-\beta(G-I)\| \|\tilde Rc^{(\tau)}\| + \|E\| \|c^{(\tau)}\|.
\]
Indeed,
\BEAS
    \|\tilde Rc^{(\tau)}\|^2 & = & \min_{c: c^T \textbf{1} = 1,\; \|c\| \leq \frac{1+\tau}{\sqrt{N}}}  \|\tilde Rc\|^2.
\EEAS
We have
\BEAS
    \min_{c:\textbf{1}^Tc=1,\; \|c\| \leq \frac{1+\tau}{\sqrt{N}}} \|\tilde Rc\|_2,  & \leq & \min_{c:\textbf{1}^Tc=1\; \|c\| \leq \frac{1+\tau}{\sqrt{N}}} \|Rc\|_2 + \|P_Rc\|_2,\\
    & \leq & \left(\min_{c:\textbf{1}^Tc=1\; \|c\| \leq \frac{1+\tau}{\sqrt{N}}} \|Rc\|_2\right) + \|P_R\|_2\frac{1+\tau}{\sqrt{N}} ,\\
    & \leq & \cheb_{N-1} \|r(x_0)\| +  \frac{\|P_R\|(1+\tau)}{\sqrt{N}}.
\EEAS
This prove the desired result.
\end{proof} 
\fi

This theorem shows that Algorithm \ref{algo:cna} balances acceleration and robustness. The result bounds the accuracy by the sum of an \textit{acceleration term} bounded using constrained Chebyshev polynomials, and a \textit{stability} term proportional to the norm of perturbations. In the next section, we consider the particular case where $g$ corresponds to a gradient step when the perturbations are Gaussian or due to non-linearities.

\section{Convergence Analysis for First-Order Methods}
\label{sec:convergence_with_gd}

We now apply our results when $g$ in \eqref{eq:general_iteration_matrix} corresponds to the gradient step
\BEQ
    g(x) = x-h\nabla f(x), \label{eq:gradient_step_g}
\EEQ
where $f$ is the objective function and $h$ a step size. We assume the function $f$ twice differentiable, $L$-smooth and $\mu$-strongly convex. This means
\BEQ
    \mu I \leq \nabla^2 f(x) \leq LI. \label{eq:smooth_strong_convex}
\EEQ
Also, we assume $h = \frac{1}{L}$ for simplicity. In the next sections, we study two different cases. First, we assume the objective quadratic, but \eqref{eq:gradient_step_g} is corrupted by Gaussian noise. Then, we consider a general nonlinear function $f$, with the additional assumption that its Hessian is Lipchitz-continuous.

\subsection{Random Perturbations}

We perform a gradient step on the quadratic form
\[
    f(x) = \frac{1}{2}(x-x^*) A (x-x^*), \;\; \mu I \preceq A \preceq L I.
\]
The gradient step \eqref{eq:gradient_step_g} thus reads
\[
    g(x) = x-\frac{1}{L} A(x-x^*) = (I-A/L)(x-x^*) + x^*.
\]
Given \eqref{eq:linear_g}, we identify the matrix $G = I-A/L$, as well as $\kappa = \frac{\mu}{L}$. However, each iteration \eqref{eq:perturbed_iteration_matrix} is corrupted by $e_i$, where $e_i$ is a stochastic noise with bounded variance $\sigma^2$.

The next proposition is the application of Theorem \ref{thm:convergence_perturbation} to our setting. To simplify results, we consider $\beta = 1$.
\begin{proposition} \label{prop:convergence_stoch_gradient}
Assume we use Algorithm \eqref{algo:cna} with \mbox{$\beta = 1$} on $N$ iterates from \eqref{eq:perturbed_iteration_matrix}, where $g$ is the gradient step \eqref{eq:gradient_step_g} and $e_i$ are zero-mean independent random noise with variance bounded by $\sigma^2$. Then,
\BEQ
    \mathbb{E}[\|\nabla f(\yex)\|] \leq (1-\kappa) \;\cheb_{N-1}\|\nabla  f(x_0)\| + \mathcal{E} ,
\EEQ
where
\[
    \mathcal{E} \leq  (1-\kappa)\frac{1+\tau}{\sqrt{N}} L\sigma \sum_{j=1}^{N} (1-\kappa)^j \bar L_{j}.
\]
In the simple case where we accelerate the gradient descent algorithm, all $L_i = I$ and thus
\[
    \textstyle \mathcal{E} \leq  \frac{1+\tau}{\sqrt{N}} \frac{L\sigma}{\kappa}. 
\]
\end{proposition}

\begin{proof}
    Since $\beta = 1$,
    \[
        \|I-\beta(G-I)\| = \|G\| \leq 1-\kappa.
    \]
    Now, consider $\mathbb{E}[\|E\|]$. Because each $e_i$ are independents Gaussian noise with variance bounded by $\sigma$, we have,
    \[
        \mathbb{E}[\|E\|] \leq \sqrt{\mathbb{E}[\|E\|^2]} \leq \sigma.
    \]
    Similarly, for $P$ \eqref{eq:perturbation_matrix}, we use Proposition \eqref{prop:explicit_formula_perturbation} and we have
    \BEAS
        \mathbb{E}[\|P\|] & \leq & \textstyle \mathbb{E}[\|E_i\|] \left(1+\sum_{j=1}^{i} (1-\kappa)^j \bar L_{j}\right)\\
        & \leq & \textstyle  \sigma  \left(1+\sum_{j=1}^{i} (1-\kappa)^j \bar L_{j}\right)
    \EEAS
    Thus, $\mathcal{E}_{N}^{\kappa,\tau}$ in Theorem \ref{thm:convergence_perturbation} becomes
    \[
        \mathcal{E}_{N}^{\kappa,\tau} \leq \textstyle  \frac{\sigma(1+\tau)}{\sqrt{N}}  \left(2+\sum_{j=1}^{N} (1-\kappa)^j \bar L_{j}\right)
    \]
    Finally, it suffice to see that
    \[
        (G-I)(x-x^*)+x^* = (A/L)(x-x^*)+x^* = \frac{1}{L} \nabla f(x),
    \]
    and we get the desired result. In the special case of plain gradient method, $L_i = I$ so $\bar L_i = 1$. Then, we bound
    \[
        \textstyle \sum_{j=1}^{N} (1-\kappa)^j  \leq \sum_{j=1}^{\infty} (1-\kappa)^j \leq \frac{1}{\kappa}.
    \]
\end{proof}

This proposition also applies to gradient descent with momentum or with our online acceleration algorithm \eqref{eq:online_rna}.

With this proposition, we can distinguish two different regimes when accelerating gradient descent with noise. One when $\sigma$ is small compared to $\|f(x_0)\|$, and one when $\sigma$ is large. In the first case, the acceleration term dominates. In this case, Algorithm \ref{algo:cna} with large $\tau$ produces output $\yex$ that converges with a near-optimal rate of convergence. In the second regime where the noise dominates, $\tau$ should be close to zero. In this case, using our extrapolation method when perturbation are high naturally gives the simple averaging scheme. We can thus see Algorithm \eqref{algo:cna} as a way to interpolate optimal acceleration with averaging.

\subsection{Nonlinear Perturbations}
Here, we study the general case where the perturbation $e_i$ are bounded by a function of $D$, where $D$ satisfies
\BEQ
    \| \tilde y_i - x^* \|_2 \leq D \qquad \forall i. \label{eq:def_d}
\EEQ
This assumption is usually met when we accelerate non-divergent algorithms. More precisely, we assume the perturbation are bounded by
\BEQ
    \big(\|I-\beta(G-I) \| \|P_N\| + \|E\|\big) \leq \gamma\sqrt{N} D^\alpha. \label{nonlinear_perturbation}
\EEQ
where $\gamma$ and $\alpha$ are scalar. This is typically the case when
\[
    \|e_i\| \leq O(D^\alpha).
\]
We call these perturbations "nonlinear" because the error term typically corresponds to the difference between $g$ and its linearization around $x^*$. For example, the optimization of smooth non-quadratic functions with gradient descent can be described using \eqref{nonlinear_perturbation} with $\alpha = 1$ or $\alpha = 2$, as shown in Section \ref{sec:smooth_functions}. The next proposition bounds the accuracy of the extrapolation produced by Algorithm \eqref{algo:cna} in the presence of such perturbation.
\begin{proposition} \label{prop:conv_nonlinear}
    
    Consider Algorithm \eqref{algo:cna} with $\beta = 1$ on $N$ iterates from \eqref{eq:perturbed_iteration_matrix}, where perturbations satisfy \eqref{eq:def_d}. Then,
    \BEAS
        \textstyle \left\|(G-I)(\yex-x^*)\right\|
        & \leq &  (1-\kappa)\Big(  \cheb_{N-1}\left\|(G-I)(x_0-x^*)\right\| + \mathcal{E}\Big)
    \EEAS
    where $\mathcal{E} \leq (1+\tau)\gamma D^\alpha$.
\end{proposition}
\begin{proof}
    Combine Theorem \ref{thm:convergence_perturbation} with assumption \eqref{nonlinear_perturbation}.
\end{proof}

Here, $\|x_0-x^*\|$ is of the order of $D$. This bound is generic as does not consider any strong structural assumption on $g$, only that its first-order approximation error is bounded by a power of $D$. We did not even assume that scheme \eqref{eq:perturbed_iteration_matrix} converges. This explains why Proposition \ref{prop:conv_nonlinear} does not necessary give a convergent bound. Nevertheless, in the case of convergent scheme, Algorithm \ref{algo:cna} with $\tau = 0$ output the average of previous iterates, that also converge to $x^*$.

However, Proposition \ref{prop:conv_nonlinear} is interesting when perturbations are small compared to $\|x_0-x^*\|$. In particular, it is possible to link $\tau$ and $D^\alpha$ so that Algorithm \ref{algo:cna} asymptotically reach an optimal rate of convergence, when $D\rightarrow 0$.

\begin{proposition}\label{prop:asymptotic_optimal_rate}
    If $\tau = O(D^{-s})$ with $0<s<\alpha-1$, then, when $D\rightarrow 0$, Proposition \ref{prop:conv_nonlinear} becomes
    \BEAS
        \textstyle \left\|(G-I)(\yex-x^*)\right\|
        & \leq &  (1-\kappa)  \left(\frac{1-\sqrt{\kappa}}{1+\sqrt{\kappa}}\right)^{N-1}\left\|(G-I)(x_0-x^*)\right\|
    \EEAS
    The same result holds with Algorithm \ref{algo:rna} if $\lambda = O(D^r)$ with $0<r<2(\alpha-1)$.
\end{proposition}

\begin{proof}
By assumption, 
\[
    \|x_0 - x^*\| = O(D).
\]
We thus have, by Proposition \ref{prop:conv_nonlinear}
    \BEAS
        \textstyle \left\|(G-I)(\yex-x^*)\right\| 
        & \leq &  (1-\kappa)\Big(  \cheb_{N-1}O(D) + (1+\tau)O(D^{\alpha})\Big).
    \EEAS
    $\tau$ will be a function of $D$, in particular $\tau = D^{-s}$. We want to have the following conditions,
    \[
        \lim\limits_{D\rightarrow 0} (1+\tau(D))D^{\alpha-1} = 0, \qquad \lim\limits_{D\rightarrow 0} \tau = \inf.
    \]
    The first condition ensures that the perturbation converge faster to zero than the acceleration term. The second condition ask $\tau$ to grow as $D$ decreases, so that CNA becomes unconstrained. Since $\tau = D^{-s}$, we have to solve
    \[
        \lim\limits_{D\rightarrow 0} D^{\alpha-1} + D^{\alpha-s-1} = 0, \qquad \lim\limits_{D\rightarrow 0} D^{-s} = \inf.
    \]
    Clearly, $0 < s < \alpha-1$ satisfies the two conditions. For the second result, by using \eqref{eq:bound_norm_lambda},
    \[
        \|c^{\lambda}\|_2 \leq \frac{1}{\sqrt{N}}\sqrt{1+\frac{1}{\lambda}}.
    \]
    Setting 
    \[
        \frac{1+\tau}{\sqrt{N}} = \frac{1}{\sqrt{N}}\sqrt{1+\frac{1}{\lambda}}
    \]
    with $\tau = D^{-s}$ gives the desired result.
\end{proof}

This proposition shows that, when perturbations are of the order of $D^\alpha$ with $\alpha > 1$, then our extrapolation algorithm converges optimally once the $\tilde y_i$ are close to the solution $x^*$. The next section shows this holds, for example, when minimizing function with smooth gradients.

\subsubsection{Optimization of Smooth Functions} \label{sec:smooth_functions}
Let the objective function $f$ be a nonlinear function that follows \eqref{eq:smooth_strong_convex}, which also has a Lipchitz-continuous Hessian with constant $M$,
\BEQ \label{eq:smooth_gradient}
    \|\nabla^2 f(y)-\nabla^2 f(x)\| \leq M\|y-x\|.
\EEQ
This assumption is common in the convergence analysis of second-order methods.

For the convergence analysis, we consider that $g(x)$ perform a gradient step on the quadratic function
\BEQ
    \frac{1}{2}(x-x^*)\nabla^2 f(x^*)(x-x^*). \label{eq:gradient_approx}
\EEQ
This is the quadratic approximation of $f$ around $x^*$. The gradient step thus reads, if we set $h=1/L$,
\BEQ
    g(x) = \left(I-\frac{\nabla^2 f(x^*)}{L}\right)(x-x^*)+x^*. \label{eq:gradient_step_linearized}
\EEQ

The perturbed scheme corresponds to the application of \eqref{eq:gradient_step_linearized} with a specific nonlinear perturbation,
\BEQ
    \textstyle \tilde x_{i+1} = g(\tilde y_i) - \underbrace{ \textstyle \frac{1}{L}(\nabla f(\tilde y_i)-\nabla^2 f(x^*)(\tilde y_i-x^*))}_{=e_i}. \label{eq:nonlinear_perturbation}
\EEQ
This way, we recover the gradient step on the non-quadratic function $f$. The next Proposition shows that schemes \eqref{eq:nonlinear_perturbation} satisfies \eqref{eq:def_d} with $\alpha = 1$ when $D$ is big, or $\alpha=2$ when $D$ is small.
\begin{proposition}\label{eq:bound_function_smooth_gradient}
    Consider the scheme \eqref{eq:nonlinear_perturbation}, where $f$ satisfies \eqref{eq:smooth_gradient}. If $\|y_i-x^*\| \leq D$, then we satisfy $\eqref{eq:def_d}$ with $\alpha = 1$ for large $D$ or $\alpha = 2$ for small $D$.
\end{proposition}
\begin{proof}
\citet{nesterov2006cubic} show that
\[
    \textstyle L e_i \leq\min\{ L\|y_i-x^*\| ,\;\;
    \frac{M}{2} \|y_i-x^*\|^2\}
\]
    Since $\|y-x^*\| \leq O(D)$ we get the desired result.
\end{proof}

The combination of Proposition \ref{prop:asymptotic_optimal_rate} with Proposition \ref{eq:bound_function_smooth_gradient} means that RNA (or CNA) converges asymptotically when $\lambda$ (or $\tau$) are set properly. In other words, if $\lambda$ decreases a little bit faster than the perturbations, the extrapolation on the perturbed iterations behave like it accelerates an perturbation-free scheme. This means its rate of convergence is optimal (Theorem \eqref{thm:optimal_rate}). Our result improves the one in \cite{scieur2016regularized,scieur2017nonlinear}, where $r\in]0,\frac{2(\alpha-1)}{3}[$.

\section{Combining RNA and Nesterov Acceleration: an Optimal and Adaptive Algorithm}
\label{sec:conbination_nesterov}

We now briefly discuss a strategy that combines the Nesterov's acceleration and RNA approaches. This means using RNA instead of the classical momentum term in Nesterov's original algorithm.  Using RNA, we can produce iterates that are adaptive to the problem constants, while ensuring an optimal upper bound if one provides constants $L$ and $\mu$. We show below how to design a condition that decides after each gradient steps if we should combine previous iterates using RNA or Nesterov coefficients.

Nesterov's algorithm \cite{nesterov2013introductory} first performs a gradient step, then combines the two previous iterates following \eqref{eq:nesterov_algo}. However, a more generic version with a basic line search reads 
\BEQ
    \begin{cases}
        \text{Find } x_{i+1} : f(x_{i+1}) \leq f(y_i) - \frac{1}{2L}\|f(y_i)\|_2^2\\
        \textstyle y_{i+1}  = (1+\beta) x_{i+1} - \beta x_{i}, \quad \beta = \frac{1-\sqrt{\kappa}}{1+\sqrt{\kappa}}\,.
    \end{cases} \label{eq:general_nesterov_step}
\EEQ
The first condition is automatically met when we perform the gradient step $x_{i+1} = x_i - \nabla f(x_i)/L$. Based on this, we propose the following algorithm.

\begin{algorithm}[htb]
   \caption{Optimal Adaptive Algorithm}
    \label{algo:optimal_adaptive}
\fbox{\parbox{0.95\linewidth}{
\begin{algorithmic}
    \STATE Compute gradient step $x_{i+1} = y_{i} - \frac{1}{L} \nabla f(y_i)$.
    \STATE Compute $\yex = \textbf{RNA}(X,Y,\lambda)$.
    \STATE Let
    \[
        z = \frac{\yex + \beta x_i}{1+\beta}
    \]
    \STATE Chose the next iterate, such that
    \[
        y_{i+1} =
        \begin{cases}
            \yex \quad \text{If}\;\; f(z) \leq f(x_i) - \frac{1}{2L}\|f(x_i)\|_2^2,\\
            (1+\beta) x_i - \beta x_{i-1}\quad \text{Otherwise}.
        \end{cases}
    \]
\end{algorithmic}
}}
\end{algorithm}

Algorithm \ref{algo:optimal_adaptive} has an optimal rate of convergence, and the proof is straightforward. If we do not satisfy the condition, then we perform a standard Nesterov step. Otherwise, we pick $z$ instead of the gradient step, and we combine
\[
    y_{i+1} = (1+\beta) z - \beta x_{i-1} = \yex.
\]
By construction this satisfies~\eqref{eq:general_nesterov_step}, and inherits its properties, like an optimal rate of convergence.

\section{Numerical Experiments on Logistic Regression and CNN} \label{sec:num_experiment}

\textbf{Logistic regression.} We test Algorithm \ref{algo:rna} on a standard machine learning problem, logistic regression with $\ell_2$ regularization. We use the Sido0 dataset from \cite{guyon2008sido}, binary classification of $12 000$ instances with $5 000$ features each, and the MNIST dataset \cite{lecun2010convolutional} (in this section, for this dataset, we used a one-versus-all approach with the digit $9$). We set the regularization parameter so that $\kappa$ is bounded by $10^{-6}$ in the deterministic case, or $L/(100\times \text{\#samples})$ in the stochastic case. Both situations correspond to ill-conditioned strongly convex functions. We present further numerical experiments on more datasets in Appendix \ref{sec:detem} (deterministic algorithms) and Appendix \ref{sec:stoch} (stochastic algorithms). We also apply our algorithm to the MNIST dataset in Appendix \ref{sec:mnist}.

We first solve the problem using deterministic methods, then stochastic methods for minimizing finite sums. In all experiments, we use RNA with parameters $N=10$, $\beta = 1$ and $\lambda = 10^{-8}$. Figure \ref{fig:comparison_reg} in Appendix \ref{sec:num_experiment} shows that the parameter $\lambda$ controls the stability of the method. Even in this simple example, with the presence of ``nonlinear'' noise, the unregularized Algorithm \ref{algo:rna} (a.~k.~a.~Anderson acceleration) is clearly unstable. Regularization is thus crucial to ensure convergence.

In Figure \ref{fig:experiments} we highlight the benefits of the combination of RNA with another optimization algorithm. On the left, we accelerate deterministic algorithms, gradient descent and Nesterov fast gradient \cite{nesterov2013introductory}. In Appendix \ref{sec:conbination_nesterov} we show how to combine RNA and fast gradient methods to preserve the optimal convergence rate. In the center, we only accelerate gradient descent, as this approach gives better results. Online RNA  significantly improves the performance of the accelerated algorithm. Our acceleration algorithm applied to gradient method is also competitive with BFGS, but does not need a line-search. Furthermore, the extension of quasi-Newton methods to stochastic algorithms is not straightforward as it requires the computation of Hessian-vector products \cite{byrd2016stochastic} while our technique applies in this setting.

On the right of Figure \ref{fig:experiments}, we combine RNA with stochastic algorithms such as SGD or SAGA~\cite{defazio2014saga}. Again, we see that Algorithm \ref{algo:rna} improves significantly the convergence rate of the method, without the requirement of a line-search, and is competitive with, e.g., Katyusha \cite{allen2016katyusha}, an accelerated algorithm for minimizing finite sums of functions. We notice that RNA on SGD does not converge. This is predicted by Proposition \ref{prop:convergence_stoch_gradient} (and this was also an observation from \cite{scieur2017nonlinear}), because contrary to SAGA, in SGD the variance of the noise $\sigma^2$ does not decrease as $x_i$ converges to $x^*$.

\textbf{CNN preliminary experiments.} We conduct neural network experiments on the ImageNet dataset in order to show the scability of our method, using a standard ResNet-50, which  incorporates a momentum at training time. We compare the offline and online acceleration, the latter being less efficient on this non-convex problem.

 Our models are trained via SGD with momentum $0.9$, an $\ell^2$-regularization of $10^{-4}$ and standard data-augmentation. At each epoch, the step-size is linearly decreased from $10^{-1}$ to $10^{-3}$. At test time, we report the accuracy of a single central crop. The model is accelerated via RNA at the end of each epoch, with $N=10,\lambda=10^{-9}$. Figure \ref{fig:resnet18} highlights that combined with offline-RNA, the performance of a ResNet-50 is not altered but this is not the case of online RNA.

We now describe the differences raised by the use of RNA. First, we observe that the computations involved are negligible. Given the fact that, on standard hardware, a GPU is about 2 orders of magnitude faster than CPUs for solving a linear problem, Figure \ref{tab:speed} indicates that  a fordward and backward pass of a single batch of a deep neural network requires more flops than a RNA pass. For this experiment, we used two P100 GPUs (\textit{a single forward+backward for a batch of size $256\times 3 \times 224\times 224$}) and a 24 cores CPU (\textit{a single acceleration via RNA for various values of $N$}); we remark that, in general, a single acceleration step of  RNA at the end of an epoch would also be diluted by order of magnitudes in the mass of computations due to the SGD with mini-batches on a whole epoch. Note also that RNA is more appropriate for CPU computations as $N$ increases, the memory consumption being higher. Secondly, we noticed that the weights $c_\lambda$ obtained via RNA are close to an uniform averaging. Obtaining those weights is not surprising given Theorem \ref{thm:convergence_perturbation}, since the noise in the SGD step is particularly high in this setting.

\begin{figure}[ht]
\centering
  \includegraphics[width=0.32\linewidth]{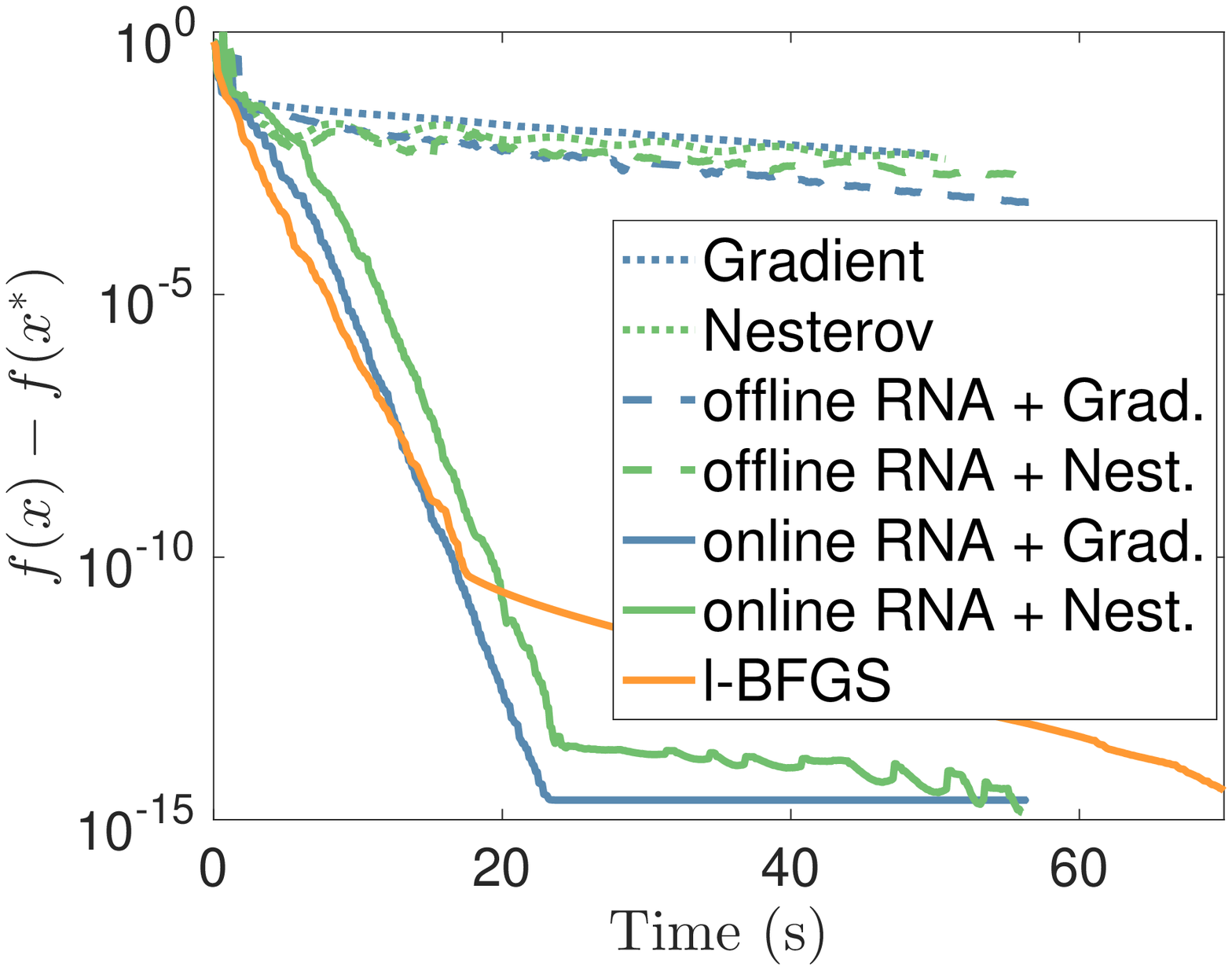}
  \includegraphics[width=0.32\linewidth]{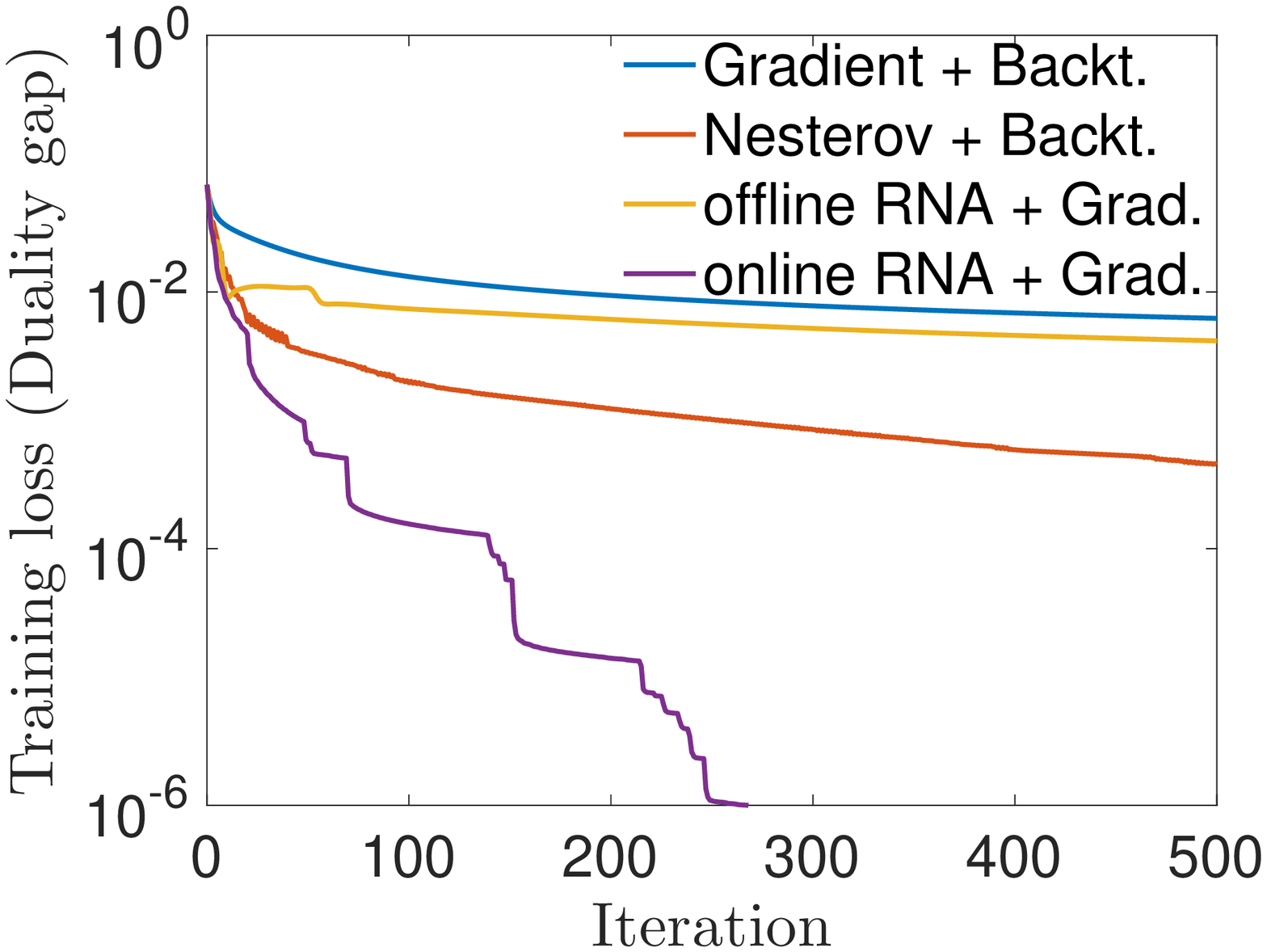}
  \includegraphics[width=0.32\linewidth]{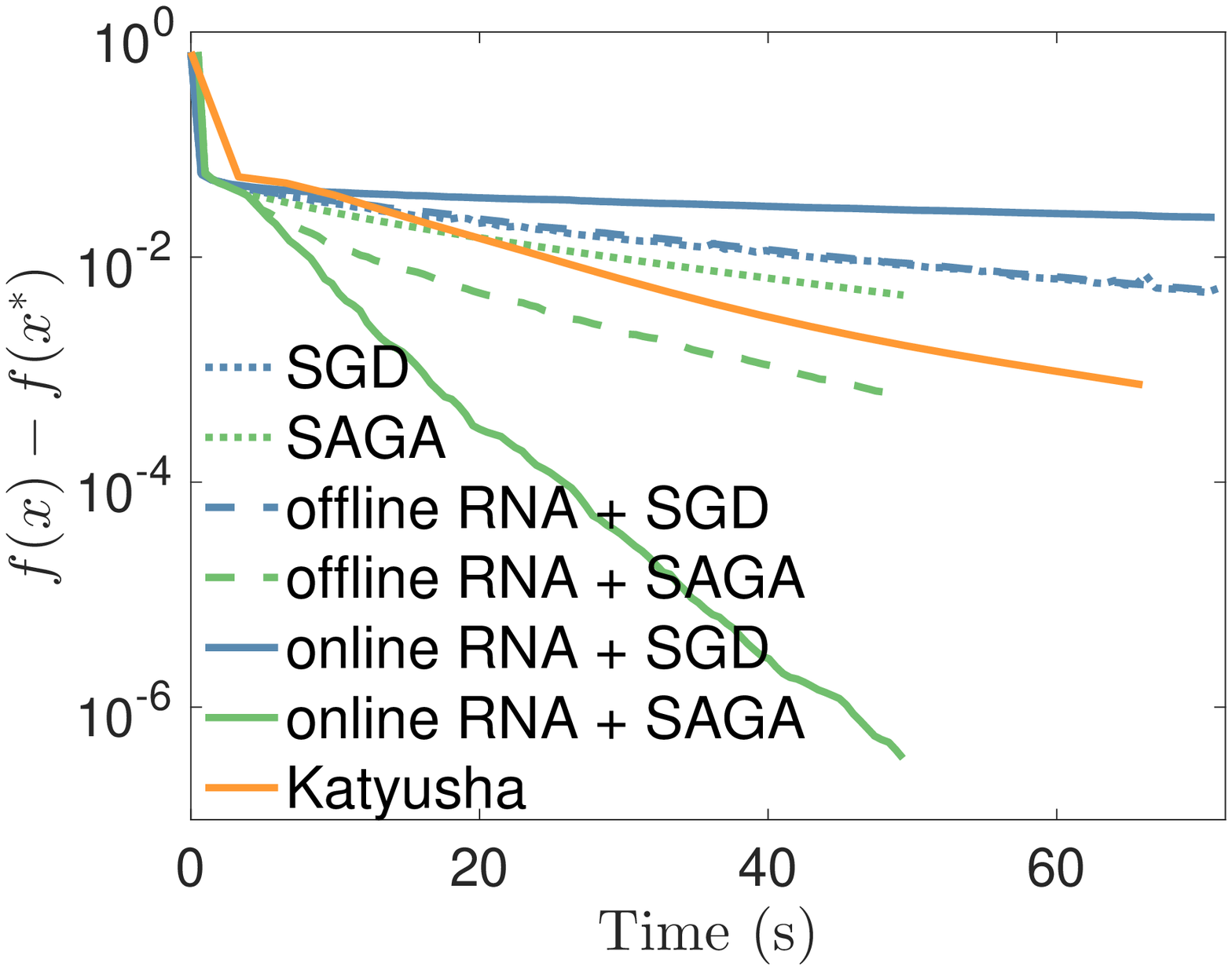}

  \caption{(Left) Logistic regression on sido0 using deterministic algorithms; (middle) logistic regression on MNIST dataset (digit 9 vs all) with deterministic algorithms; (right) logistic regression on Sido0 dataset with stochastic algorithms.}
  \label{fig:experiments}
\end{figure} \begin{figure}[ht]
\centering
  \begin{minipage}{.42\linewidth}
  \includegraphics[width=0.9\textwidth]{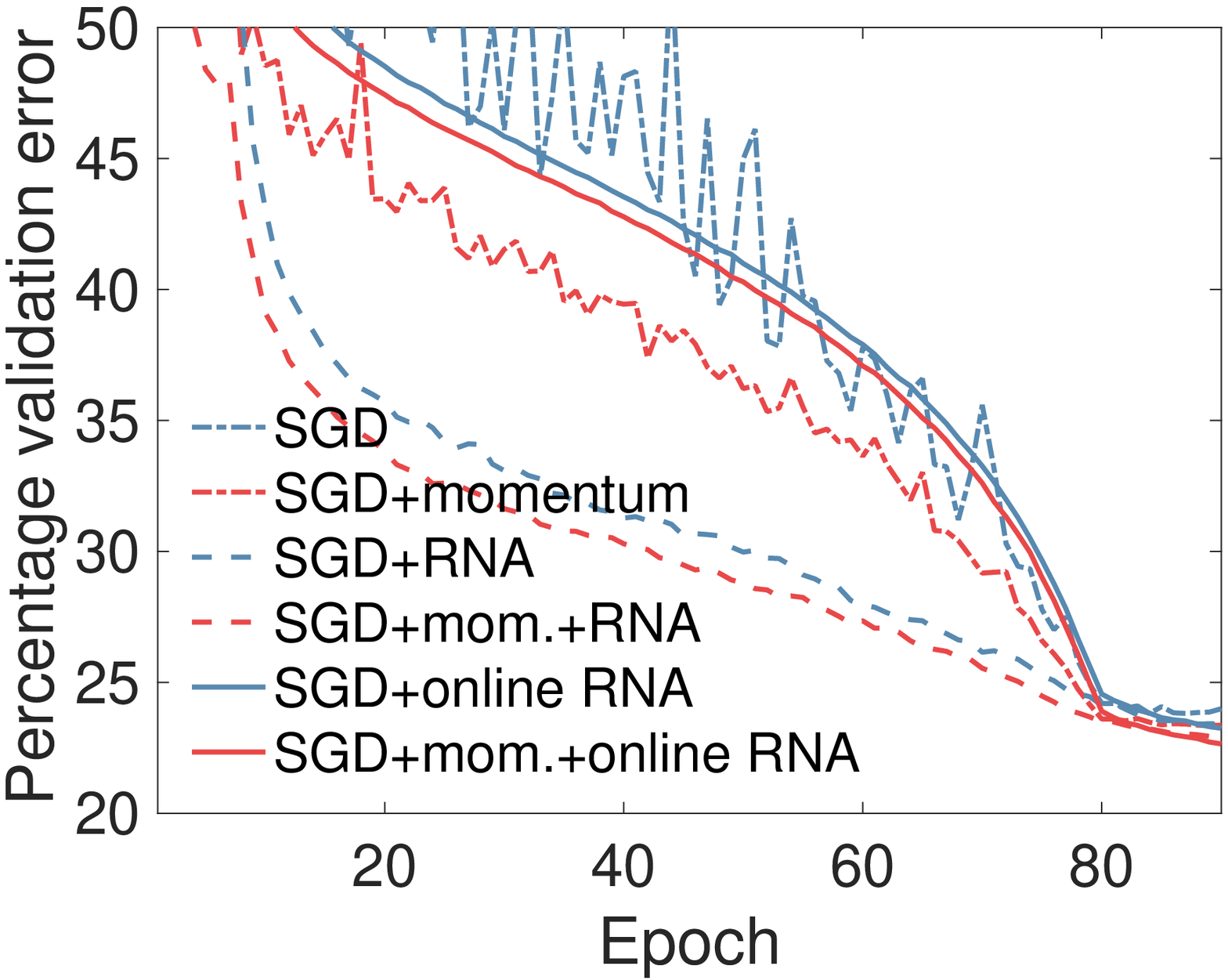}
  \caption{Validation error on ImageNet with offline or online RNA and a standard training on a ResNet-50.}
  \label{fig:resnet18}
  \end{minipage}
  \hspace{6ex}
  \begin{minipage}{.45\linewidth}
  \includegraphics[width=0.9\textwidth]{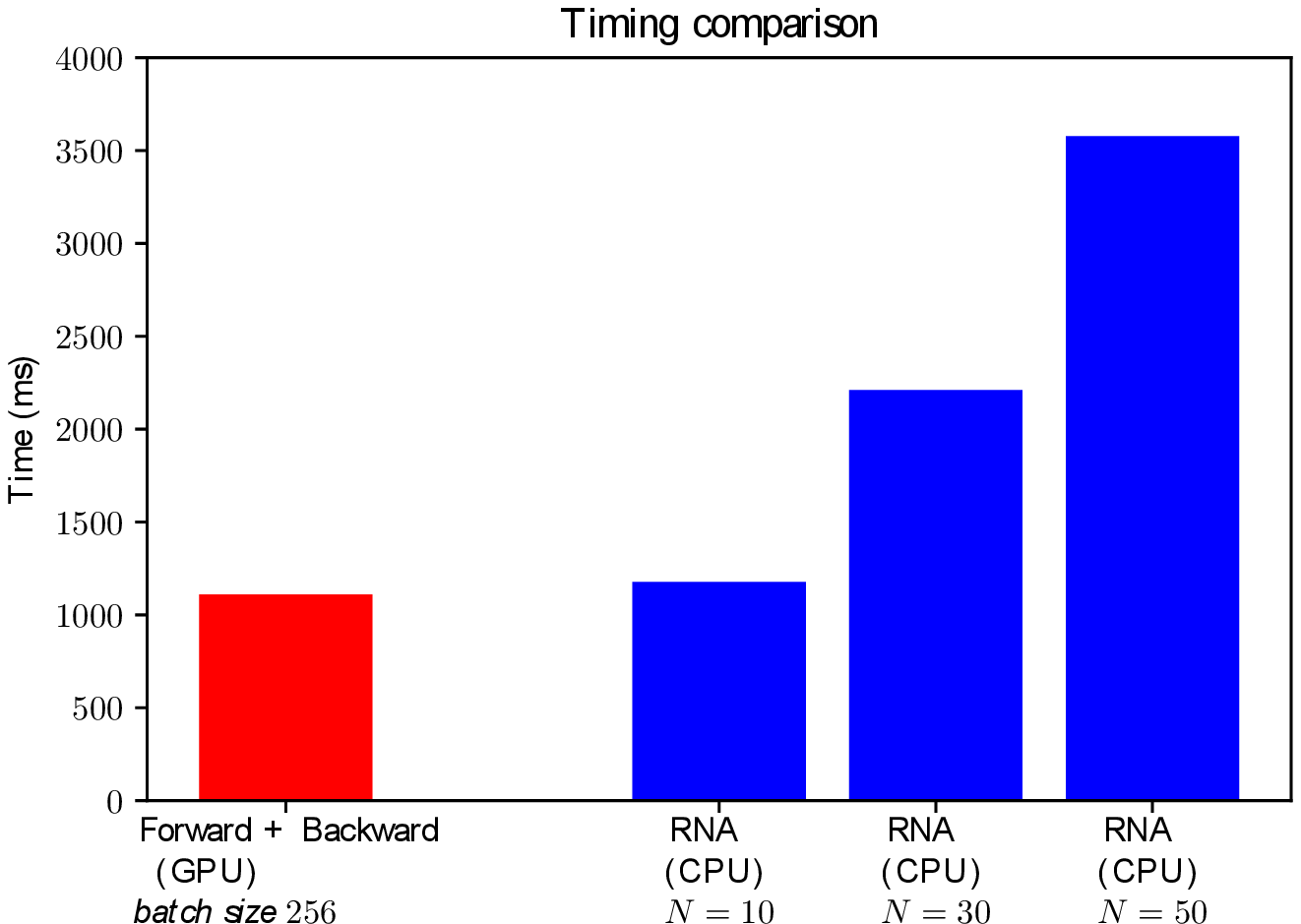}
  \caption{Speed comparison between RNA (CPU) and a forward+backward pass(GPU), for ResNet-50.}
  \label{tab:speed}
  \end{minipage}
\end{figure}

\clearpage
{
    \bibliography{biblio,MainPerso}
    \bibliographystyle{plainnat}
}

\clearpage

\appendix

\section*{Appendices for the paper Online Regularized Nonlinear Acceleration}

This part contains the supplementary materials for the paper \textit{Online Regularized Nonlinear Acceleration}. The appendices are structured as follow. We first present in Appendix \ref{prop:online_accel_structure_proof} a poposition that shows the special structure of iterates of the online RNA algorithm. In Appendices \ref{sec:add_num_exp} and \ref{sec:mnist} we present additionnal experiments on logistic regression on several datasets.

\startcontents[sections]

\section*{Table of Contents (Appendix)}

\printcontents[sections]{}{1}{\setcounter{tocdepth}{2}}

\clearpage

\section{Proposition \ref{prop:online_accel_structure}}
\label{prop:online_accel_structure_proof}

\begin{proposition} \label{prop:online_accel_structure}

Let $X$, $Y$ \eqref{eq:def_xy} be build using iterates from \eqref{eq:general_iteration}. Let $g$ be the linear, satisfying the assumptions of Theorem \eqref{thm:optimal_rate}. Consider $\yex$ the output of Algorithm \ref{algo:rna} with $\lambda = 0$ and $\beta \neq 0$. If $R = X-Y$ is full column rank, then $\cl_N \neq 0$. Otherwise, $\yex = x^*$.
\end{proposition}
\begin{proof}
    Since, by definition, $\textbf{1}^T c^{\lambda}=1$, it suffices to prove that the last coefficient $c^{\lambda}_N \neq 0$. For simplicity, in the scope of this proof we write $c=c^{\lambda}$ We prove it by contradiction. Let $R^{-}$ be the matrix $R$ without its last column, and $c^-$ be the coefficients computed by RNA using $R^-$. Assume $c_N = 0$. In this case,
    \[
        Rc = R^- c^-.
    \]
    In other words, because $R$ (and by consequence $R^-$) are full column rank,
    \[
        c = [c^-;\; 0].
    \]
    This also means that, using the explicit formula for $c$ \eqref{eq:cl},
    \[
        \frac{(R^TR)^{-1}\textbf{1}}{\textbf{1}(R^TR)^{-1}\textbf{1}} = \left[ \frac{((R^-)^TR^-)^{-1}\textbf{1}}{\textbf{1}((R^-)^TR^-)^{-1}\textbf{1}};\; 0\right], \qquad \Leftrightarrow \qquad (R^TR)^{-1}\textbf{1} = \left[ ((R^-)^TR^-)^{-1}\textbf{1};\; 0\right].
    \]
    The equivalence is obtained because
    \[
        \textbf{1}(R^TR)^{-1}\textbf{1} = \textbf{1}^T c = \textbf{1}^T c^- = \textbf{1}((R^-)^TR^-)^{-1}\textbf{1}.
    \]
    We can write $c$ and $c^-$ under the form of a linear system,
    \[
        R^TRc = \alpha \textbf{1}_N, \quad ((R^-)^TR^-)c^- = \alpha \textbf{1}_{N-1},
    \]
    where $\alpha  = \textbf{1}(R^TR)^{-1}\textbf{1} = \textbf{1}((R^-)^TR^-)^{-1}\textbf{1}$, which is nonzero. We augment the system with $c^-$ by concatenating zeros,
    \[
        R^TRc = \alpha \textbf{1}_N, \quad 
        \begin{bmatrix}
            ((R^-)^TR^-) & 0_{N-1 \times 1} \\ 
            0_{1 \times N-1} & 0
        \end{bmatrix}
        \begin{bmatrix}
        c^- \\
        0
        \end{bmatrix}
        = \alpha
        \begin{bmatrix}
            \textbf{1}_{N-1}\\
            0
        \end{bmatrix}
    \]
    Let $r^+$ be the residual at iteration $N$. This means $R = [R^-,r^+]$. We substract the two linear systems,
    \[
        \begin{bmatrix}
            0 & R^Tr^+ \\ 
            (r^+)^TR & (r^+)^Tr^+
        \end{bmatrix}
        \begin{bmatrix}
        c^- \\
        0
        \end{bmatrix}
        =
        \begin{bmatrix}
            0\\
            \alpha \neq 0
        \end{bmatrix}
    \]
    The $N-1$ first equations tells us that either $(R^Tr^+)_i$ or $c^-_i$ are equal to zero. This implies 
    \[
        (R^Tr^+)^Tc = \sum{i=1}^{N-1}(R^Tr^+)^T_ic_i=0.
    \]
    However, the last equation reads
    \[
        (R^Tr^+)^Tc + 0\cdot (r^+)^Tr^+ \neq 0.
    \]
    This is a contraction, since 
    \[
        (R^Tr^+)^Tc + 0\cdot (r^+)^Tr^+ = 0.
    \]
    Now, assume $R$ is not full rank. This means there exist a non-zero linear combination such that
    \[
        Rc = 0.
    \]
    However, due to its structure $R$ is a Krylov basis of the Krylov subspace
    \[
        \mathcal{K}_N = \text{span}[r_0,Gr_0,\ldots, G^{N}]
    \]
    If the rank of $R$ is strictly less $N$ (says $N-1$), this means
    \[
        \mathcal{K}_N = \mathcal{K}_{N-1}.
    \]
    Due to properties of the Krylov subspace, this means that
    \[
        r_0 = \sum_{i=1}^{N-1} \alpha_i \lambda_i v_i
    \]
    where $\lambda_i$ are distinct eigenvalues of $G$, and $v_i$ the associated eigenvector. Thus, it suffices to take the polynomial $p^*$ that interpolates the $N-1$ distinct $\lambda_i$. In this case,
    \[
        p^*(G)r_0 = 0.
    \]
    Since $p(1)\neq 0$ because $\lambda_i \leq 1-\kappa < 1$, we have
    \[
        \min \|Rc\| = \min_{p\in \pnm} \|p(G)r_0\| = \frac{p^*(G)}{p(1)}r_0 = 0. 
    \]
\end{proof}

\clearpage
\section{Additional Figures}
\label{sec:add_num_exp}

\subsection{Comparison between Anderson Acceleration (or RNA without regularization) and RNA}
\label{sec:comparison_anderson}

\begin{figure}[ht]
    \centering
    \includegraphics[width=0.5\linewidth]{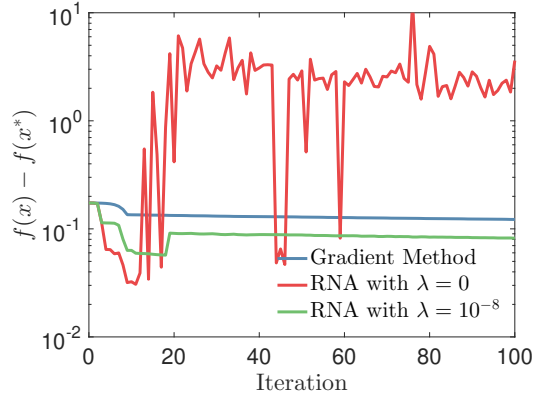}
    \caption{Comparing regularized and unregularized RNA.}
    \label{fig:comparison_reg}
\end{figure}

\subsection{Dataset Description}\label{sec::dataset}

In the following experiments, we use the following datasets.
\begin{itemize}
    \item Sonar UCI \cite{gorman1988analysis}: 208 samples, 60 features.
    \item Madelon UCI \cite{guyon2008feature}: 4400 samples, 500 features.
    \item Sido0 Dataset \cite{guyon2008sido}: 12678 samples, 4932 features.
\end{itemize}

\subsection{Experiments on Logistic regression}

We test Algorithm \ref{algo:rna} on a standard machine learning problem, logistic regression with $\ell_2$ regularization. We use the dataset presented in section \ref{sec::dataset}. We set the regularization parameter so that $\kappa$ is bounded by $10^{-3}$ ("good condition number"), $10^{-6}$ ("regular condition number") and $10^{-9}$ ("bad condition number") in the deterministic case, or $100 L/(\times \text{\#samples})$ ("good condition number"), $L/(1\times \text{\#samples})$ ("regular condition number"), $L/(100\times \text{\#samples})$ ("bad condition number") in the stochastic case.

We first solve the problem using deterministic methods, then stochastic methods for minimizing finite sums. In all experiments, we use RNA with parameters $N=10$, $\beta = 1$ and $\lambda = 10^{-8}$.

\clearpage

\subsection{Legend for Logistic Regression}

\vspace{-2ex}
\begin{figure}[ht]
\centering
\begin{minipage}{.2\linewidth}
  \centering
  \includegraphics[width=1\linewidth]{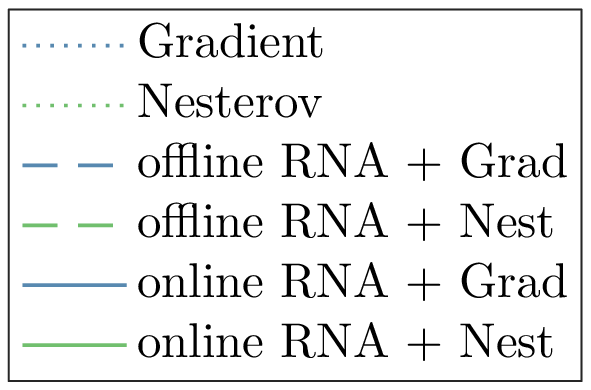}
\end{minipage}
\hspace{0.1\linewidth}
\begin{minipage}{.2\linewidth}
  \centering
  \includegraphics[width=1\linewidth]{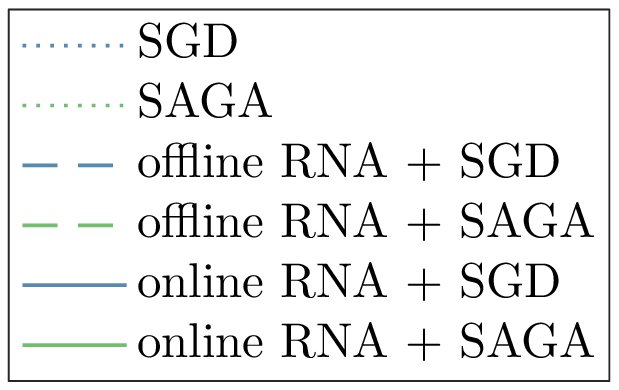}
\end{minipage}
  \caption{Left: legend for deterministic algorithms. Right: legend for stochastic algorithms.}
\end{figure}
\vspace{-3ex}
\subsection{Logistic regression, deterministic}
\label{sec:detem}

\subsubsection{Dataset: sonar (Top to bottom: good, regular and bad condition number)}
\vspace{-3ex}
\begin{figure}[ht]
\centering
\includegraphics[width=0.4\textwidth]{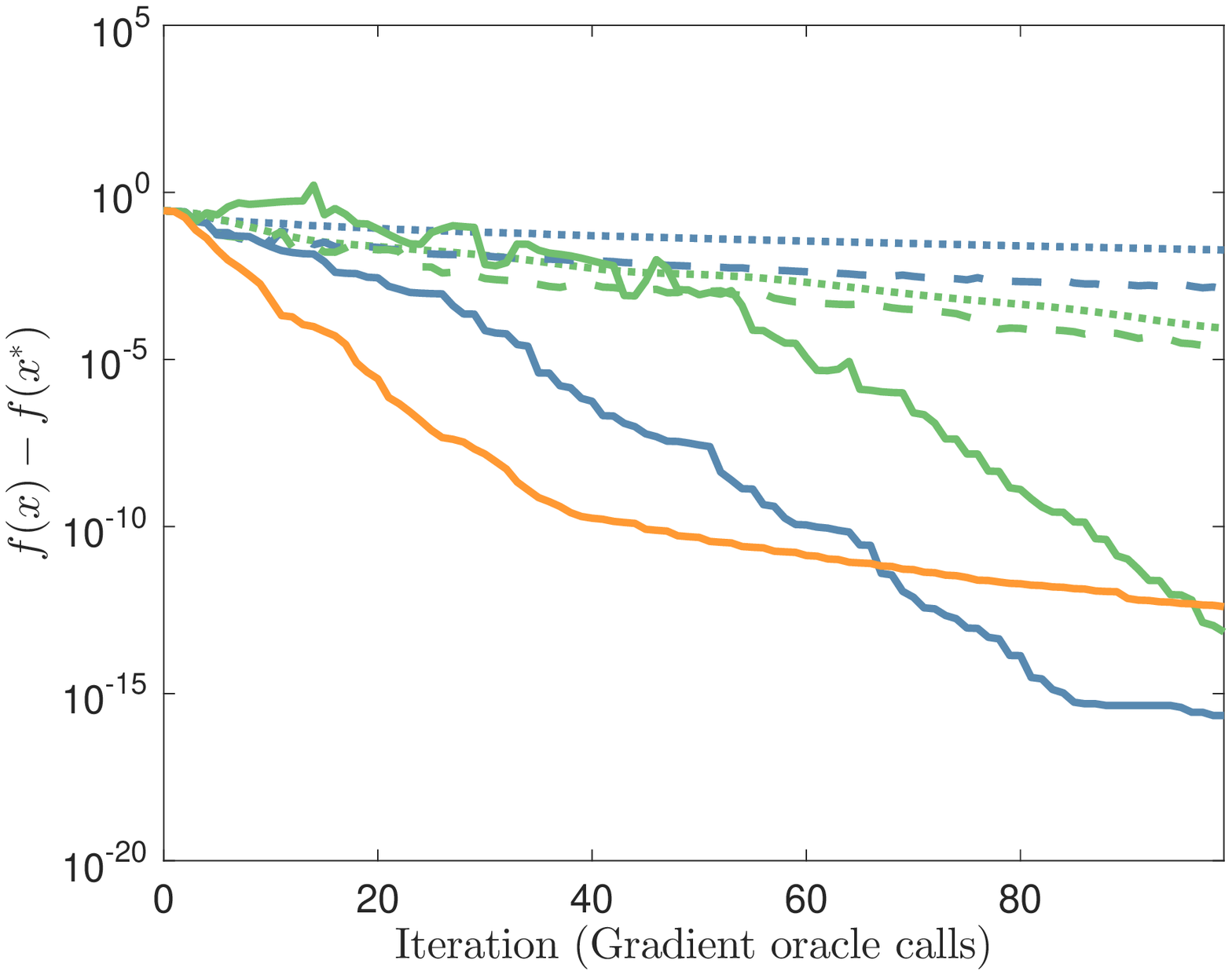}
\includegraphics[width=0.4\textwidth]{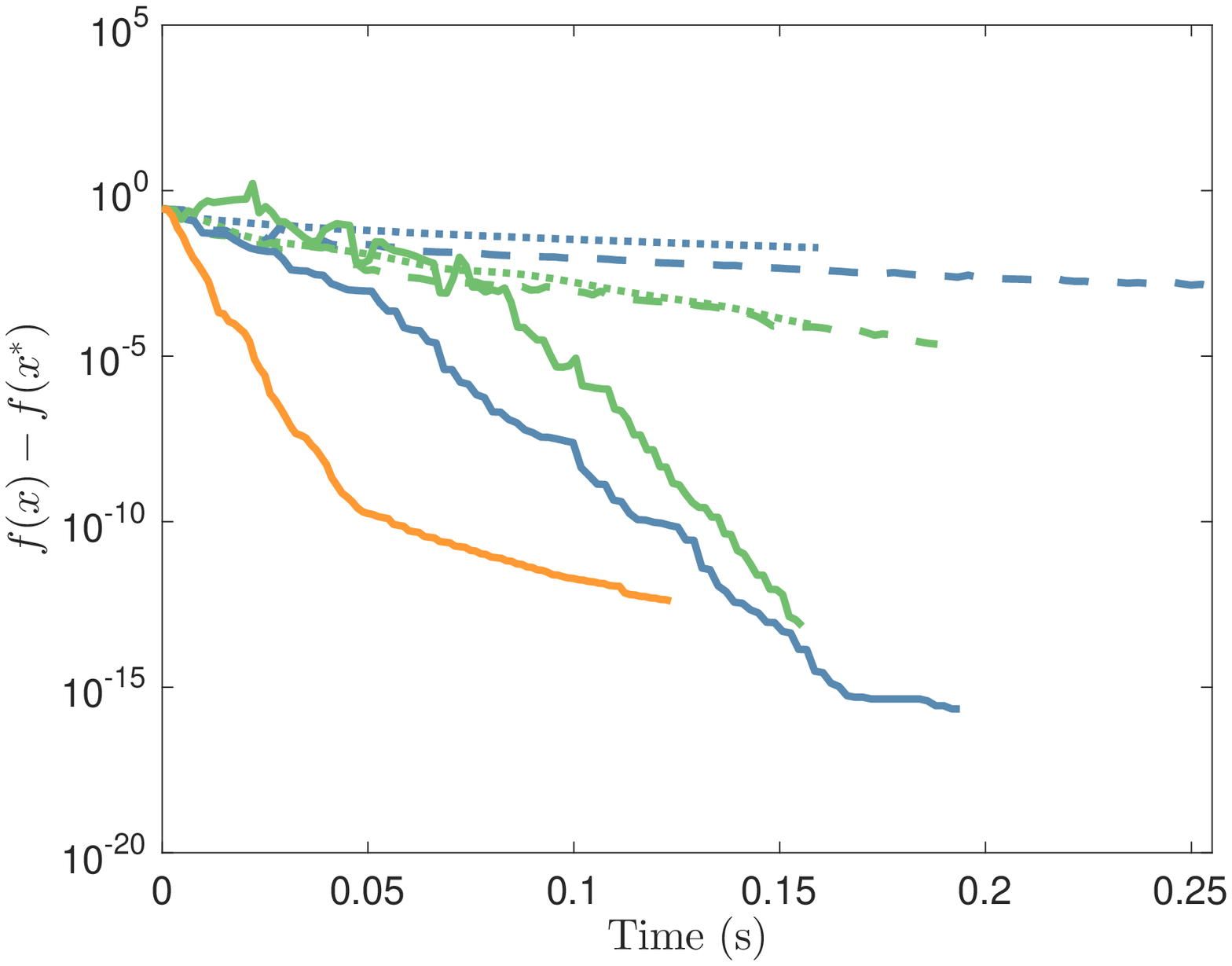}
\end{figure}
\begin{figure}[ht]
\centering
\vspace{-3ex}
\includegraphics[width=0.4\textwidth]{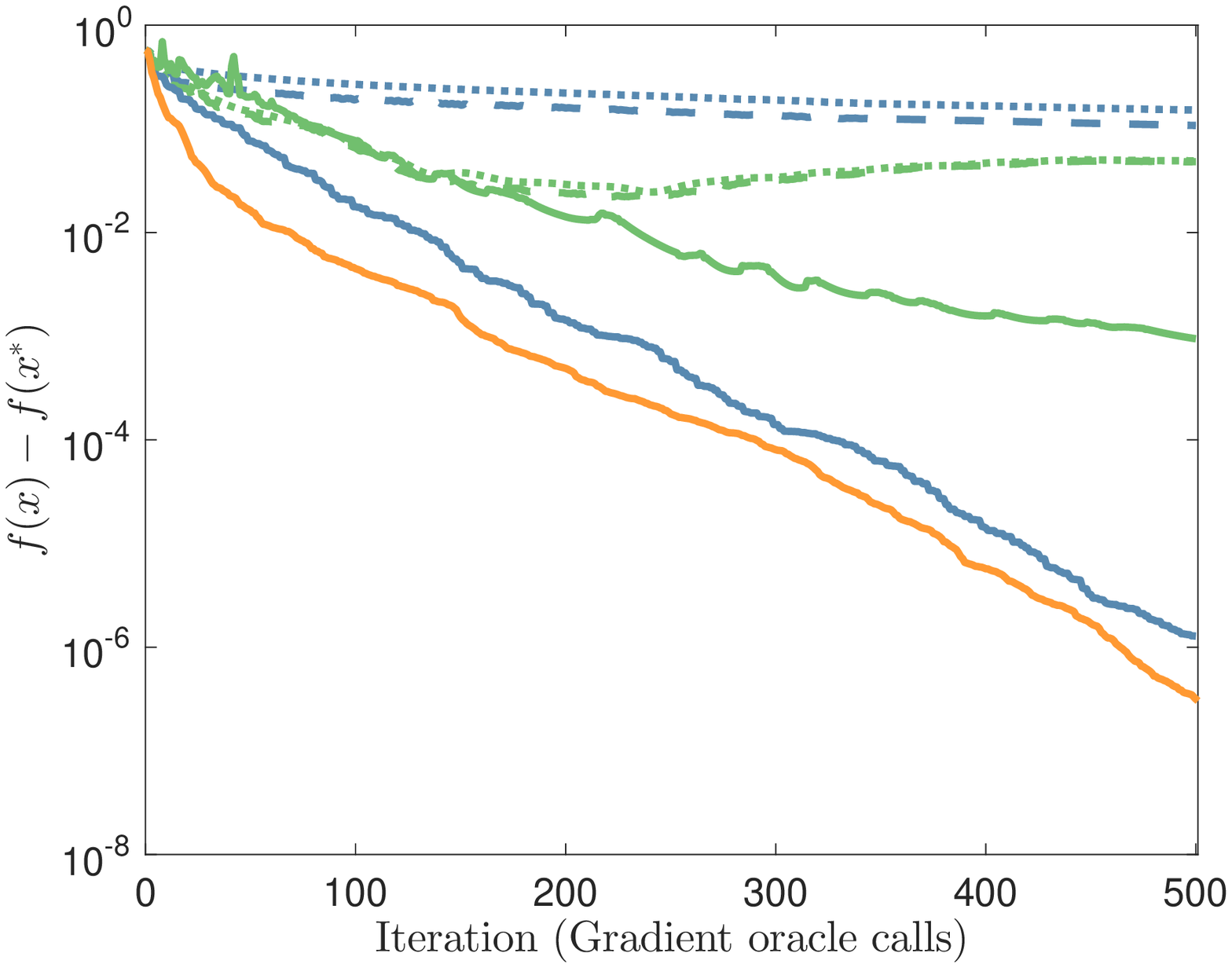}
\includegraphics[width=0.4\textwidth]{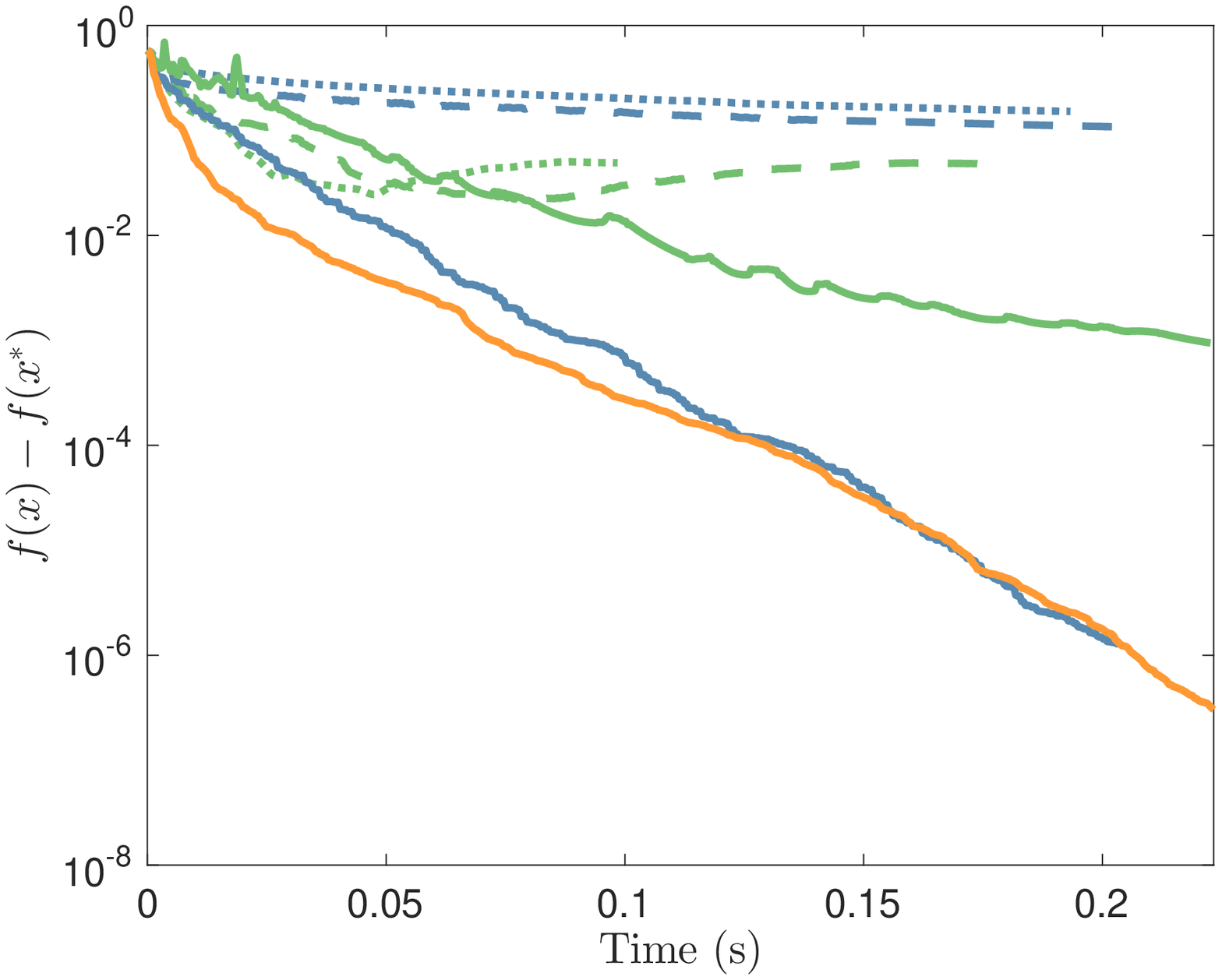}
\end{figure}
\begin{figure}[h!t]
\centering\vspace{-3ex}\includegraphics[width=0.4\textwidth]{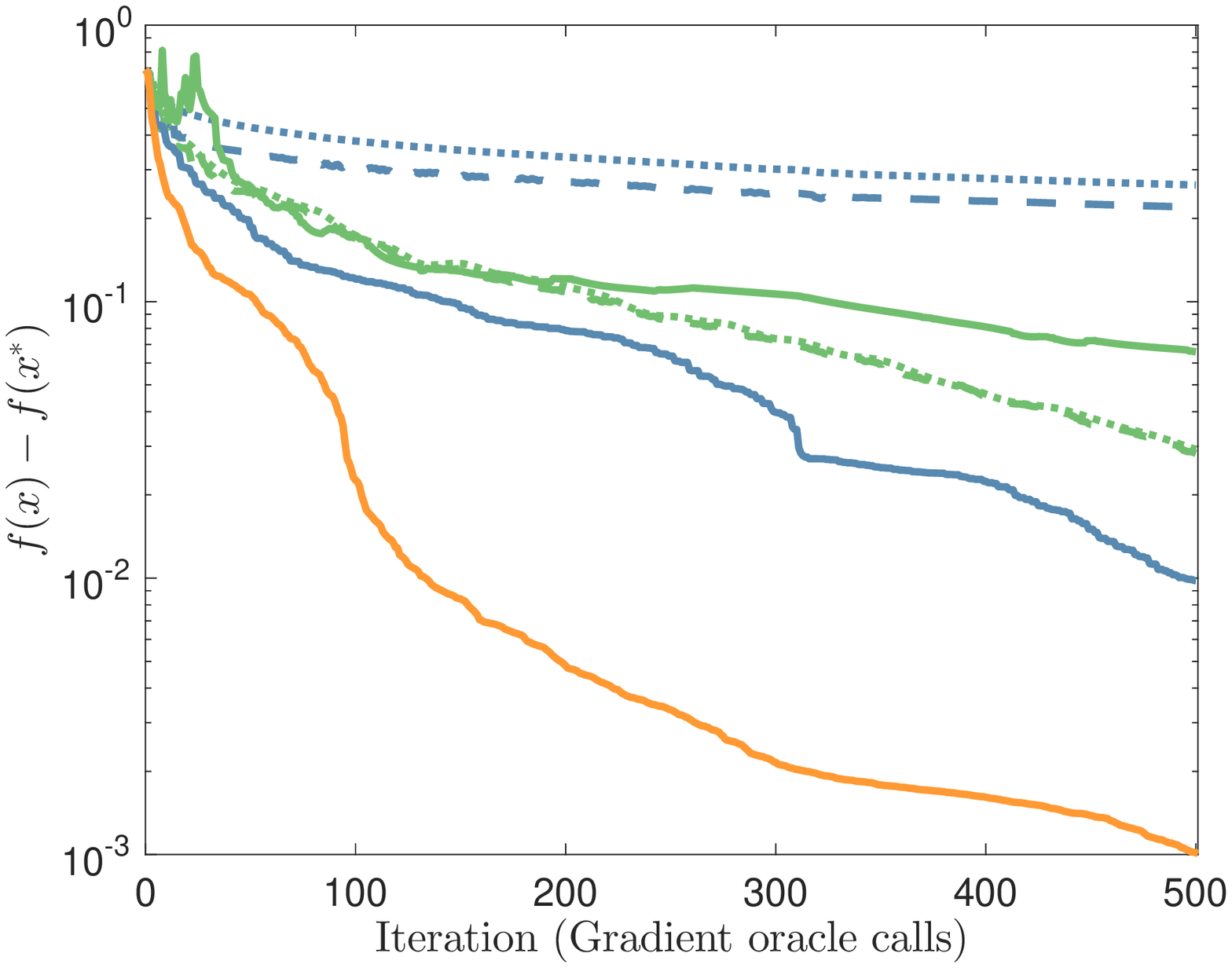}\includegraphics[width=0.4\textwidth]{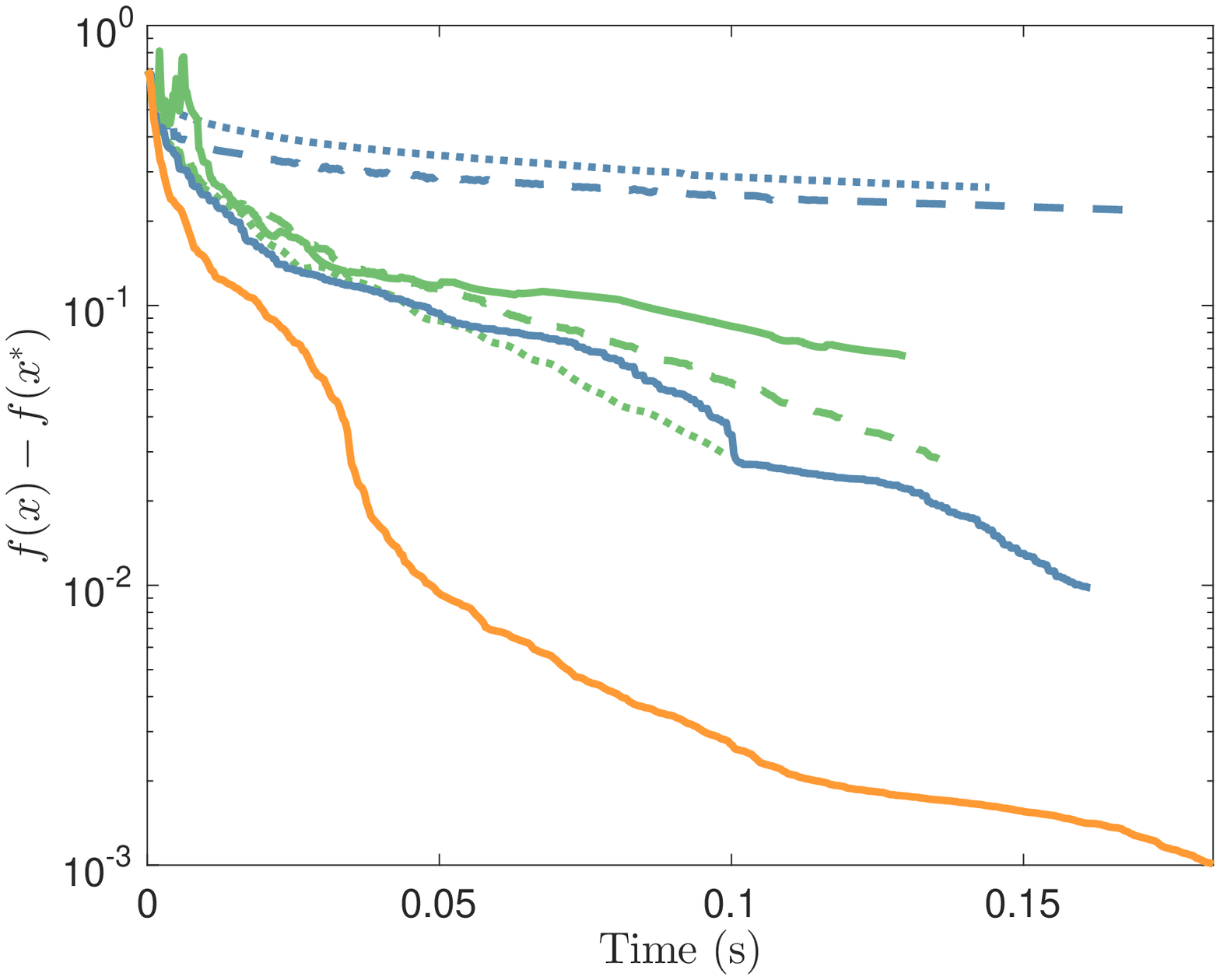}
\end{figure}

\clearpage
 
\subsubsection{Dataset: madelon (Top to bottom: good, regular and bad condition number)}
\begin{figure}[ht]
\centering
\includegraphics[width=0.4\textwidth]{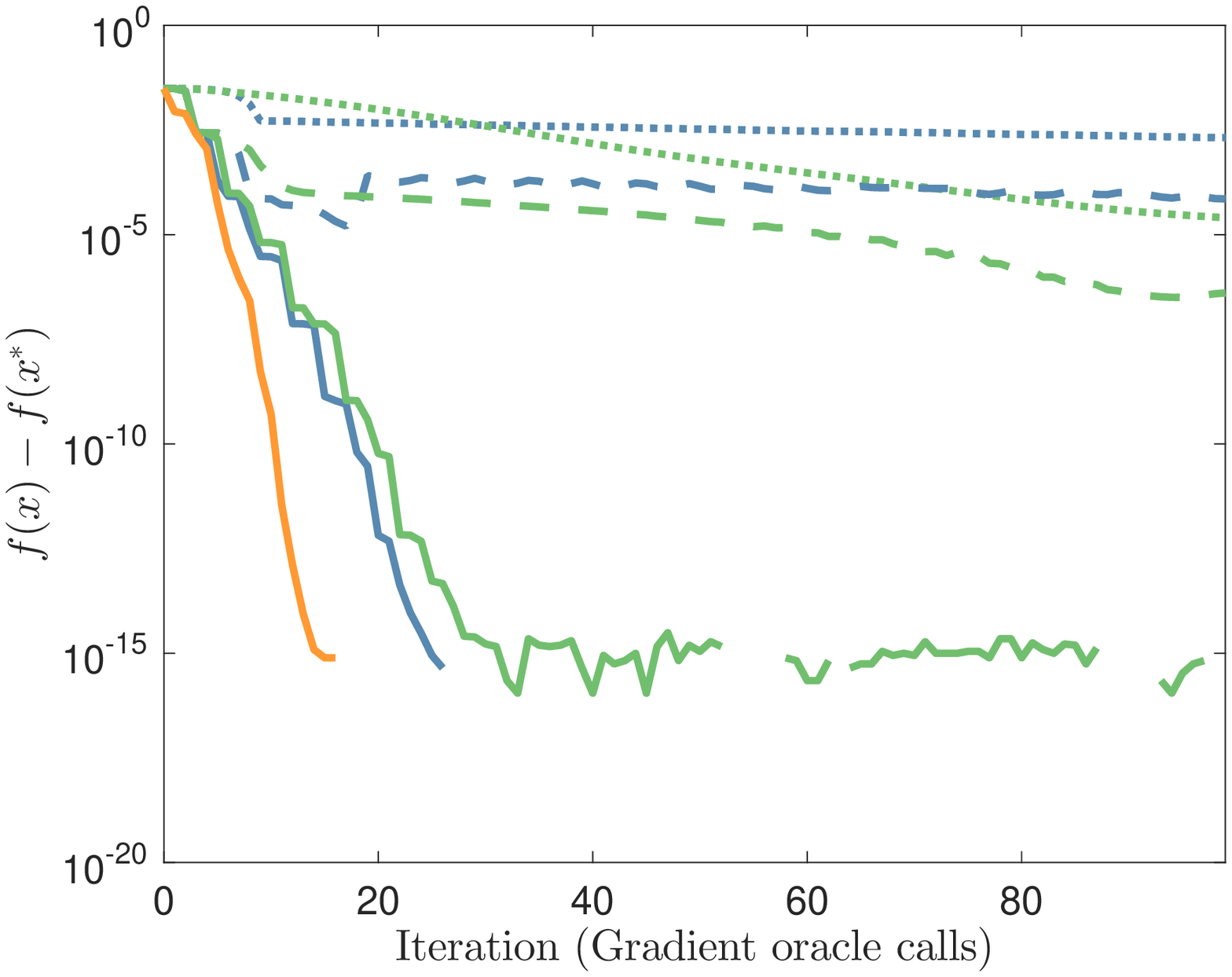}
\includegraphics[width=0.4\textwidth]{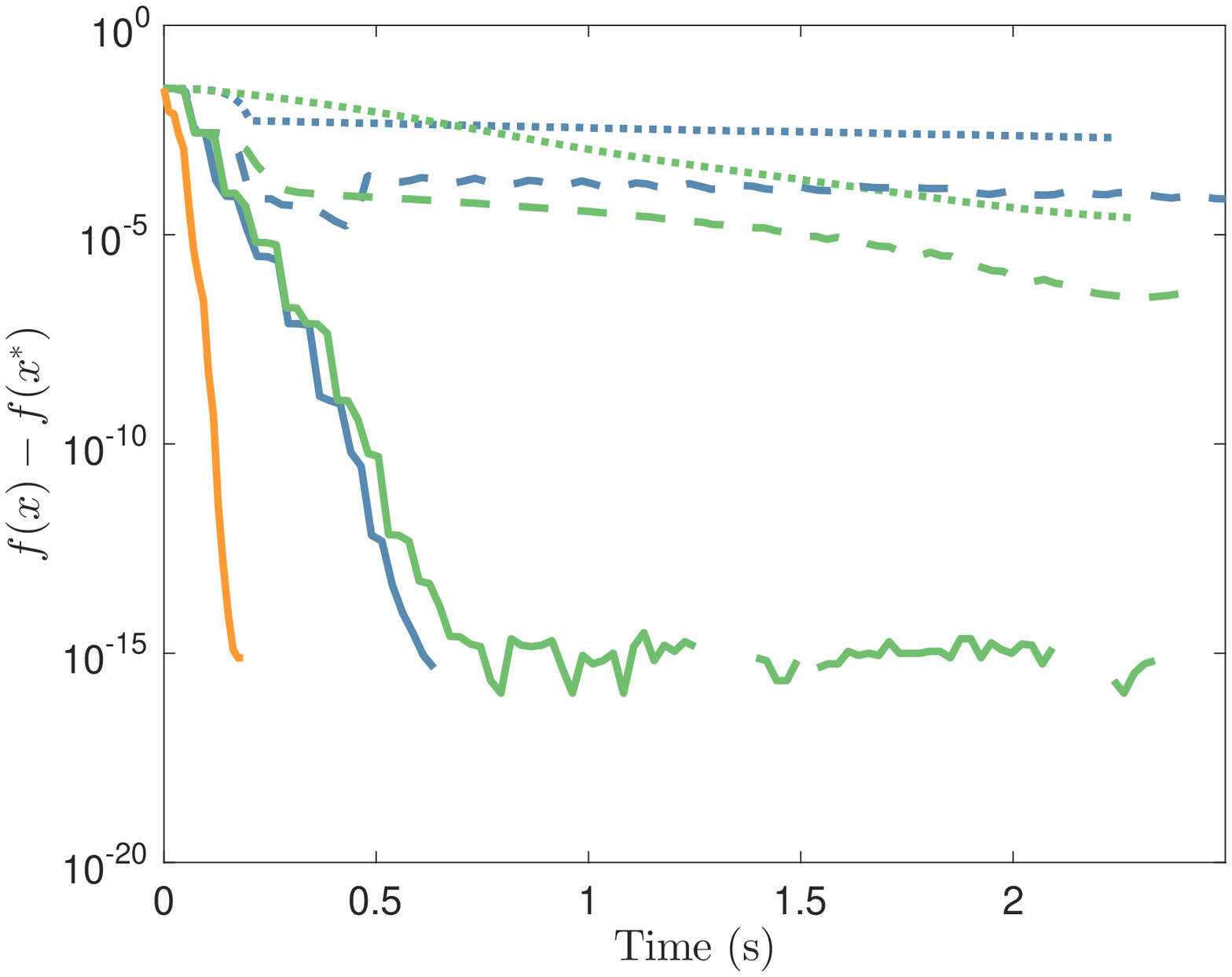}
\end{figure}
\begin{figure}[ht]
\centering
\includegraphics[width=0.4\textwidth]{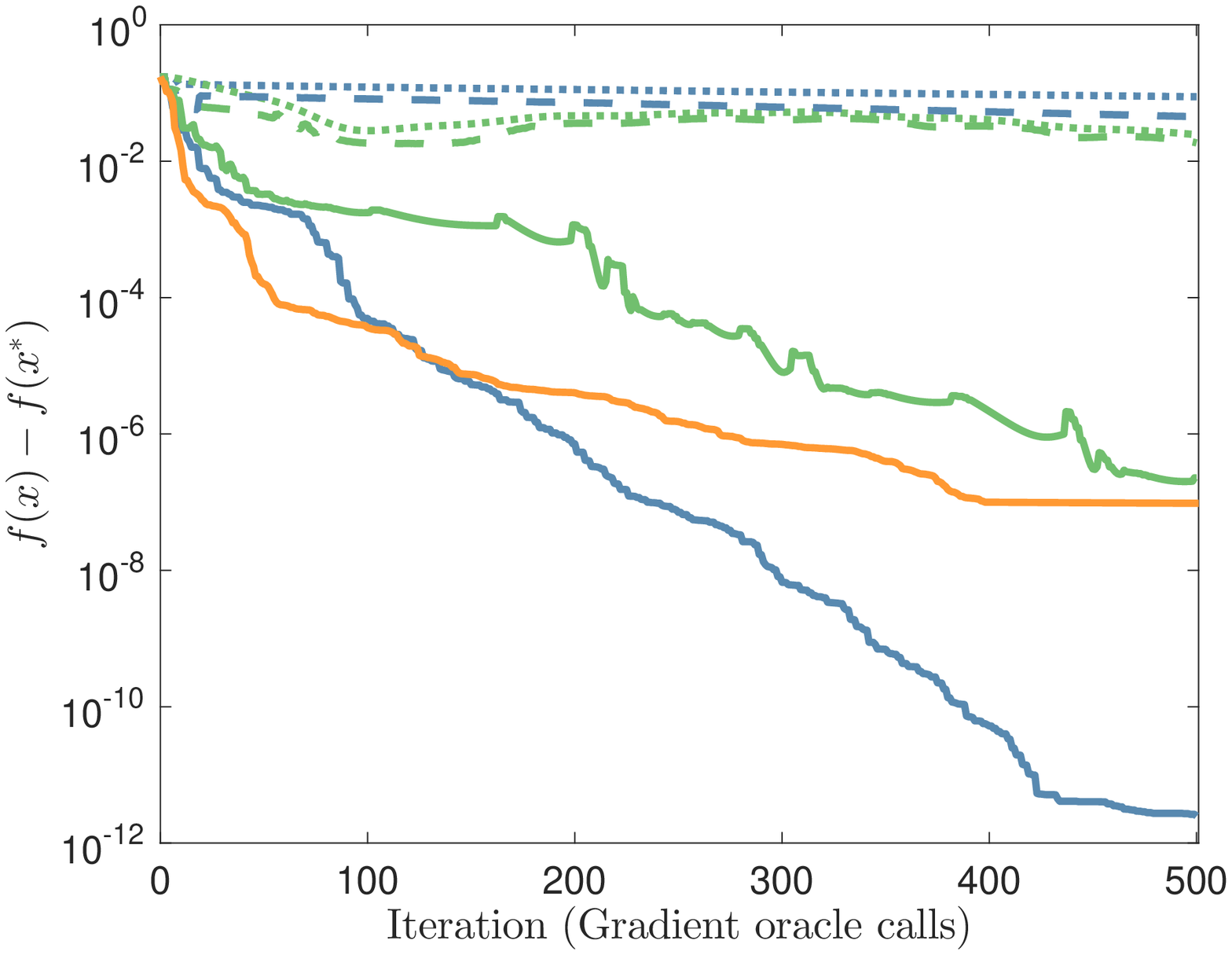}
\includegraphics[width=0.4\textwidth]{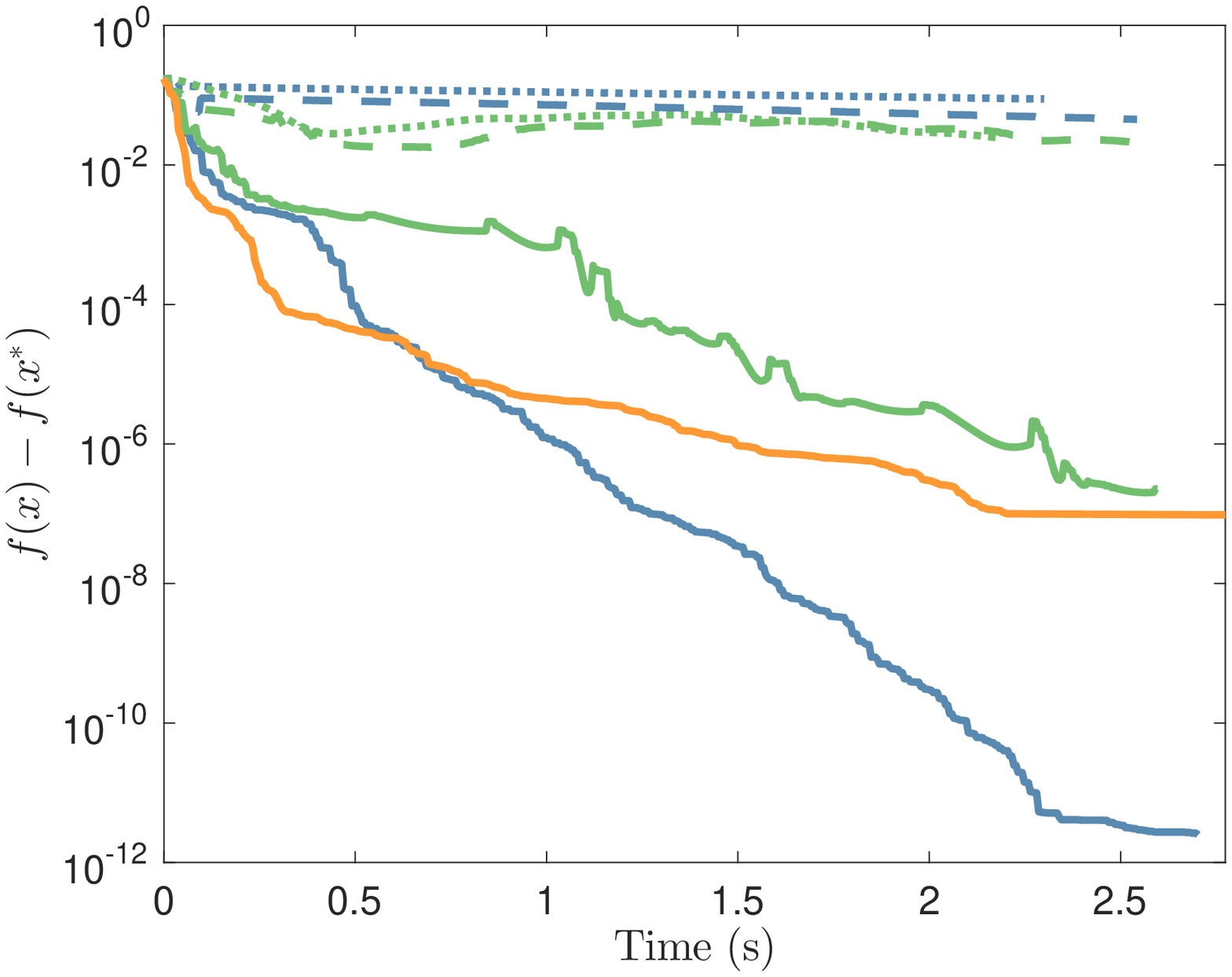}
\end{figure}
\begin{figure}[ht]
\centering
\includegraphics[width=0.4\textwidth]{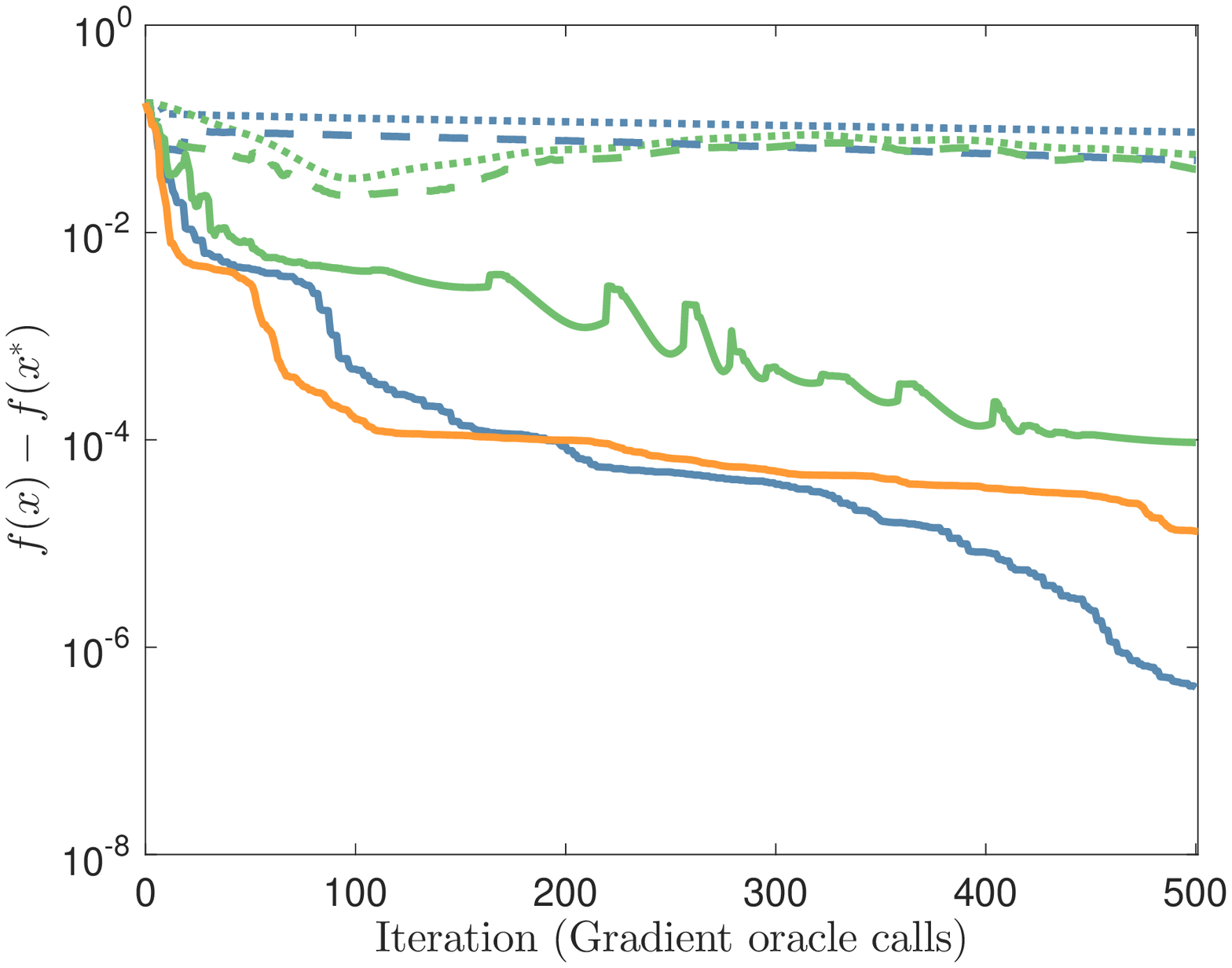}
\includegraphics[width=0.4\textwidth]{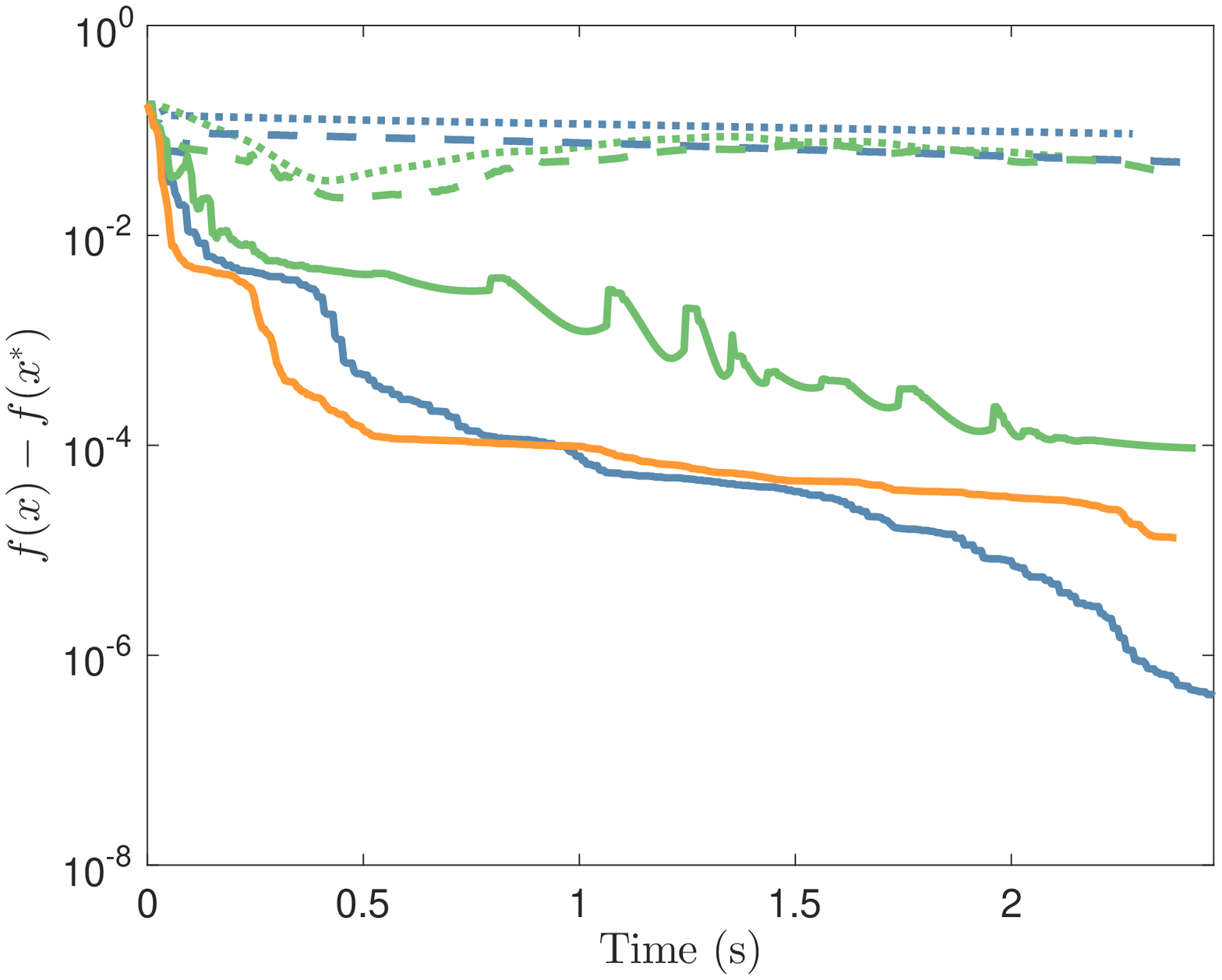}
\end{figure}
\clearpage
 
\subsubsection{Dataset: sido0 (Top to bottom: good, regular and bad condition number)}
\begin{figure}[ht]
\centering
\includegraphics[width=0.4\textwidth]{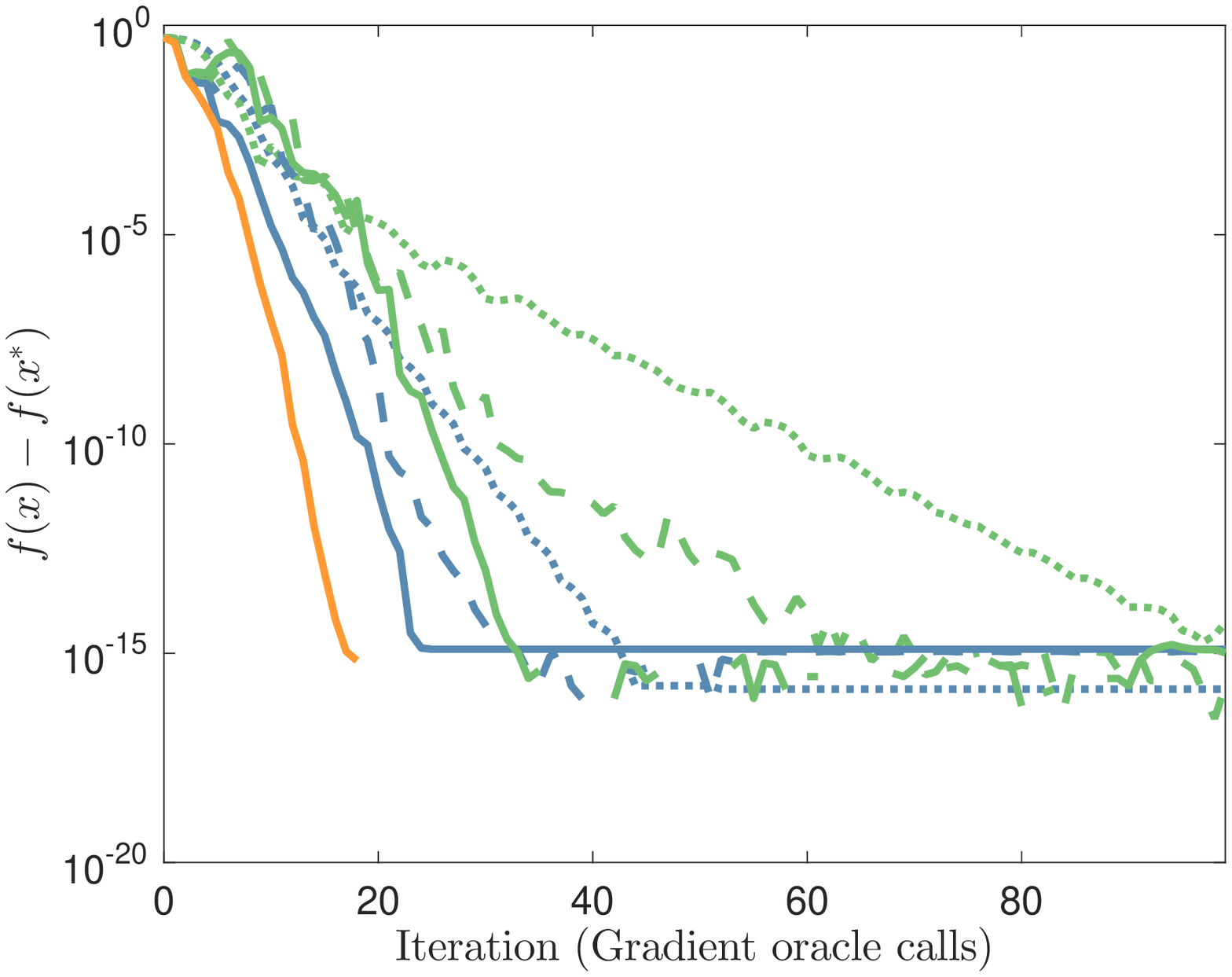}
\includegraphics[width=0.4\textwidth]{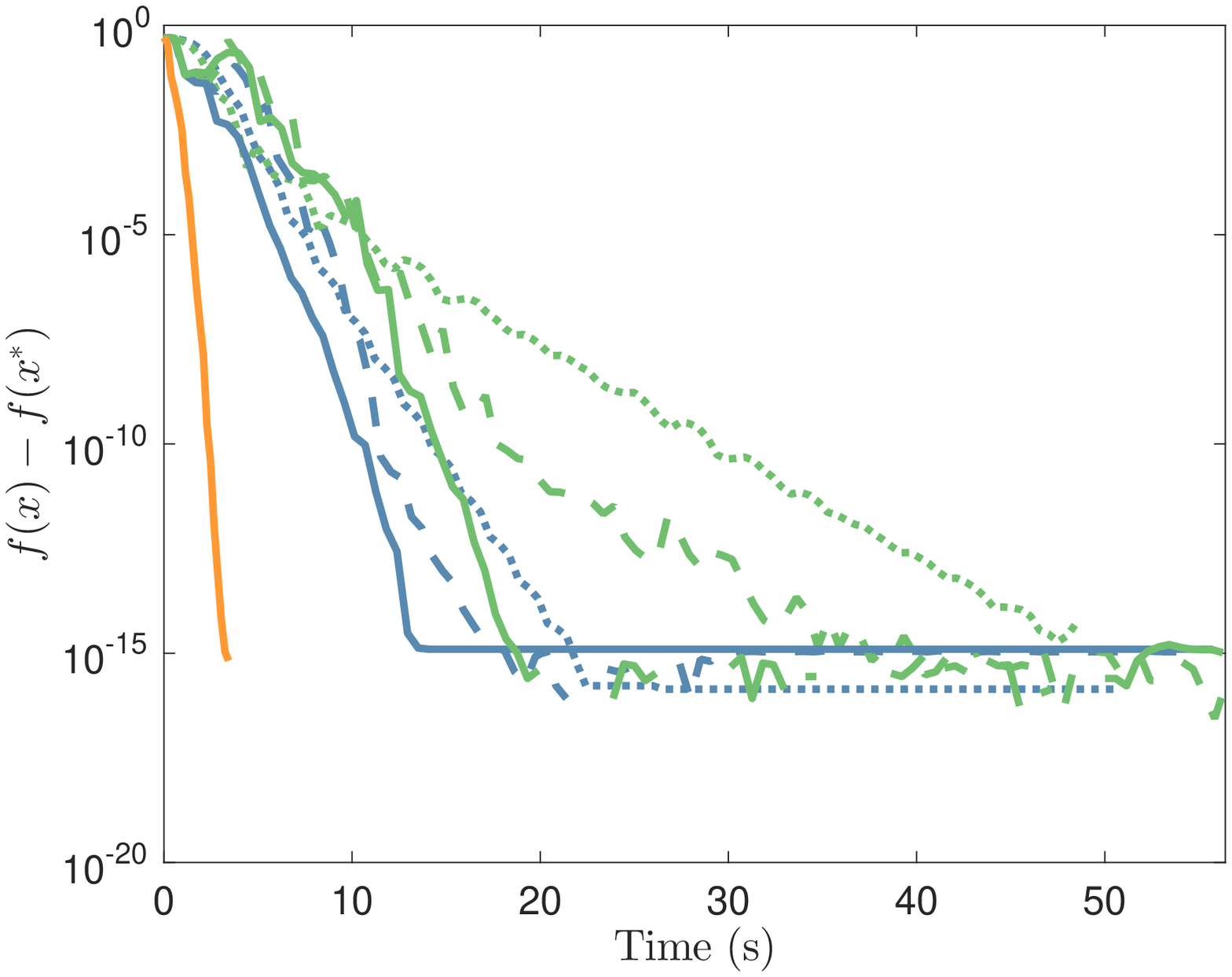}
\end{figure}
\begin{figure}[ht]
\centering
\includegraphics[width=0.4\textwidth]{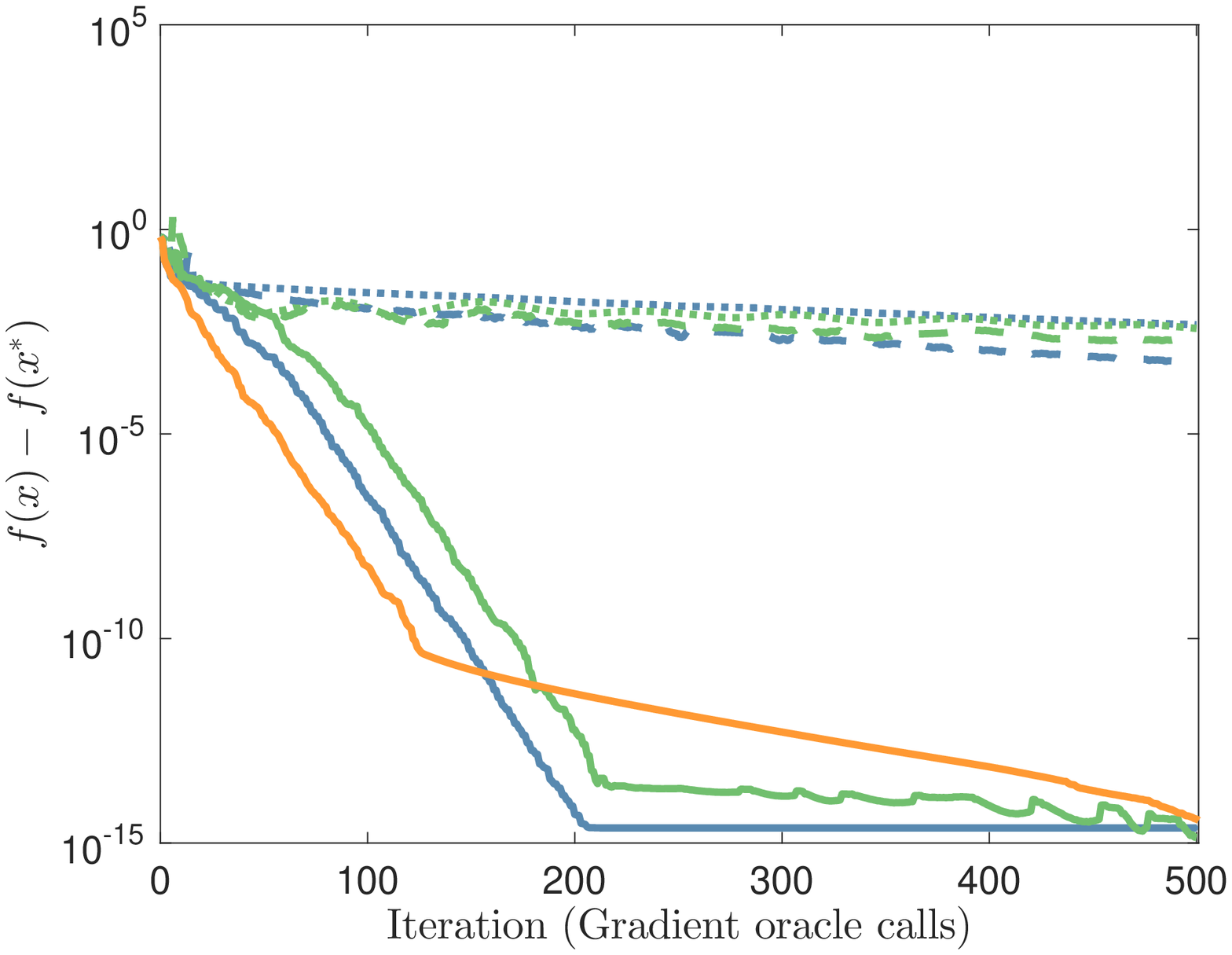}
\includegraphics[width=0.4\textwidth]{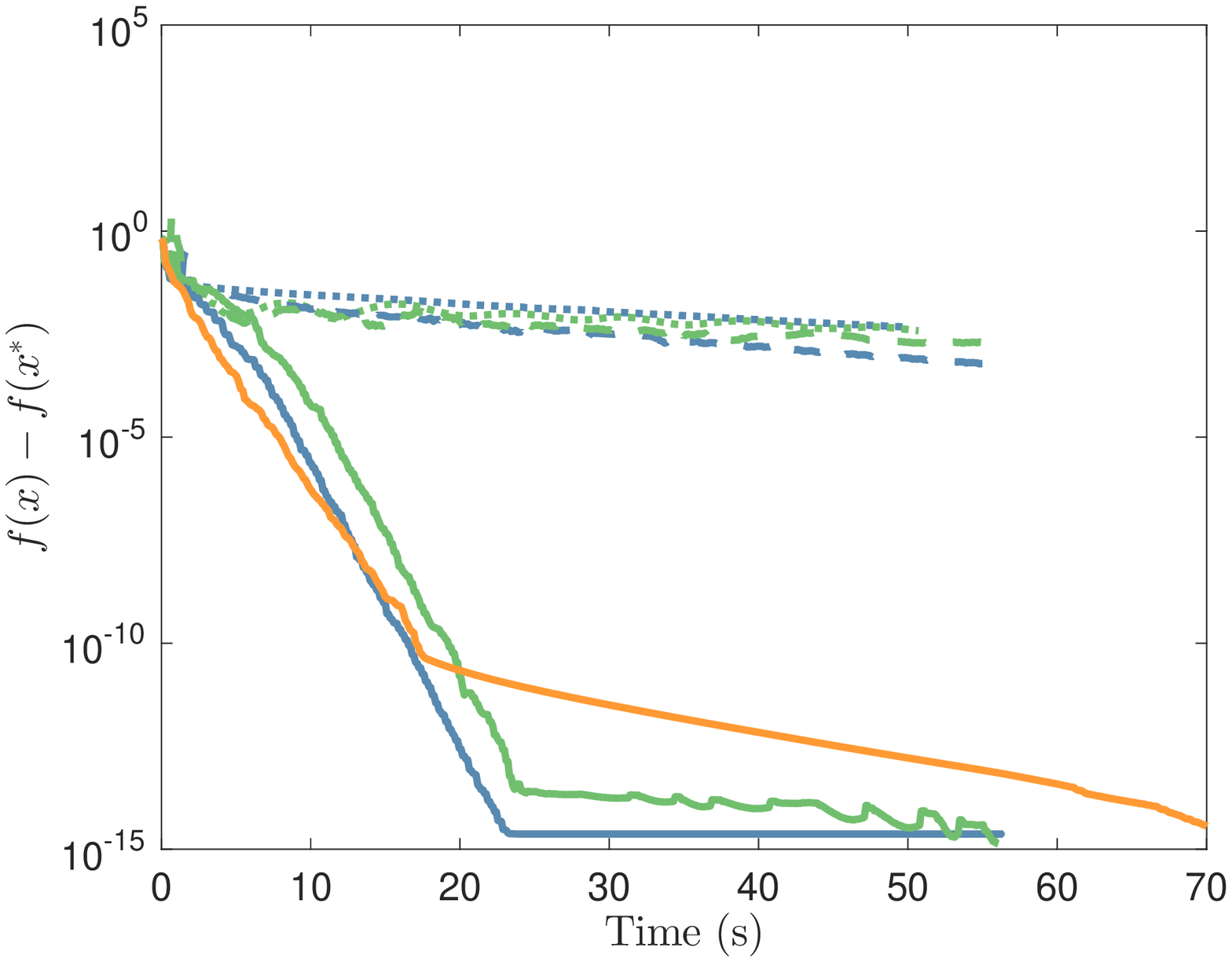}
\end{figure}
\begin{figure}[ht]
\centering
\includegraphics[width=0.4\textwidth]{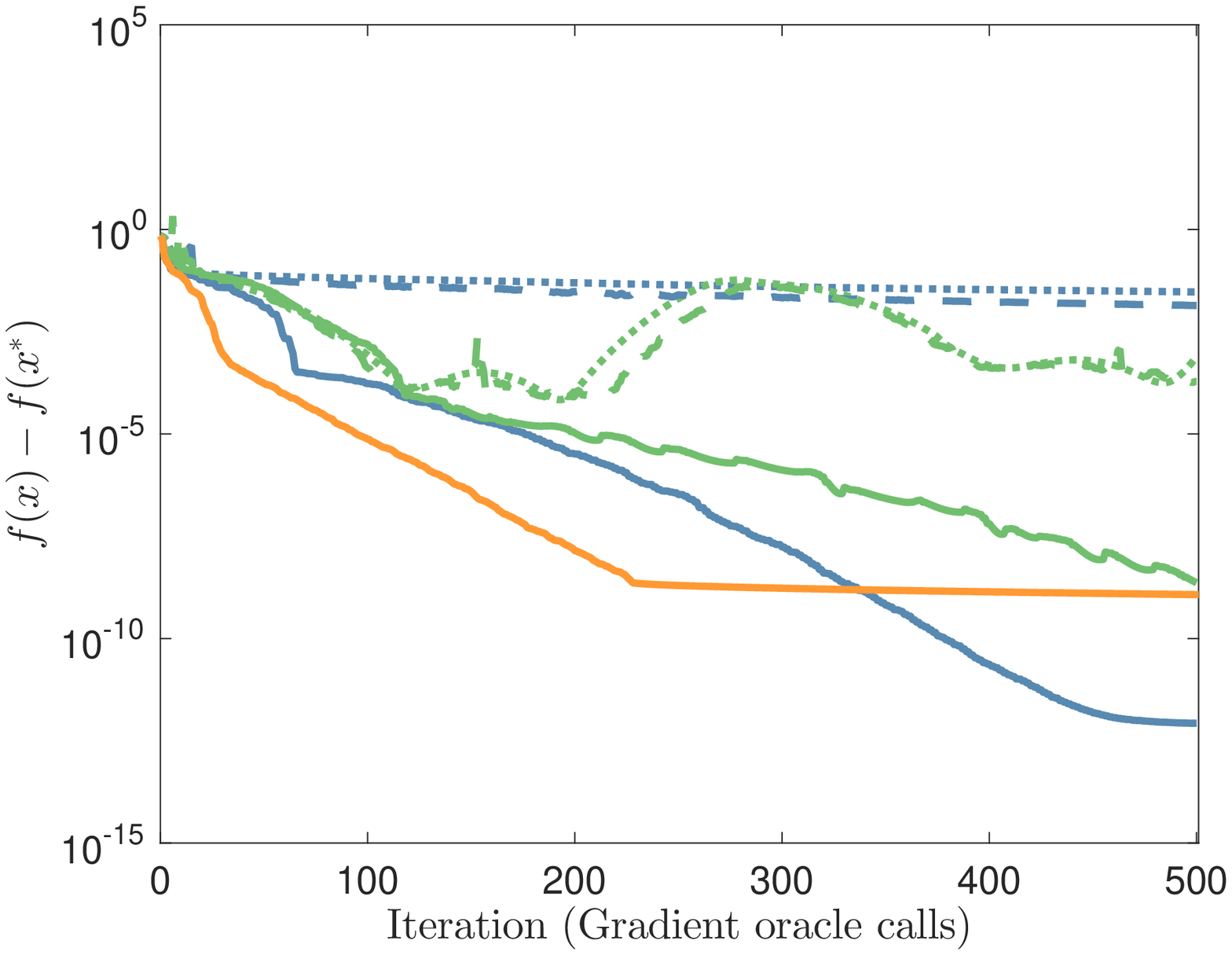}
\includegraphics[width=0.4\textwidth]{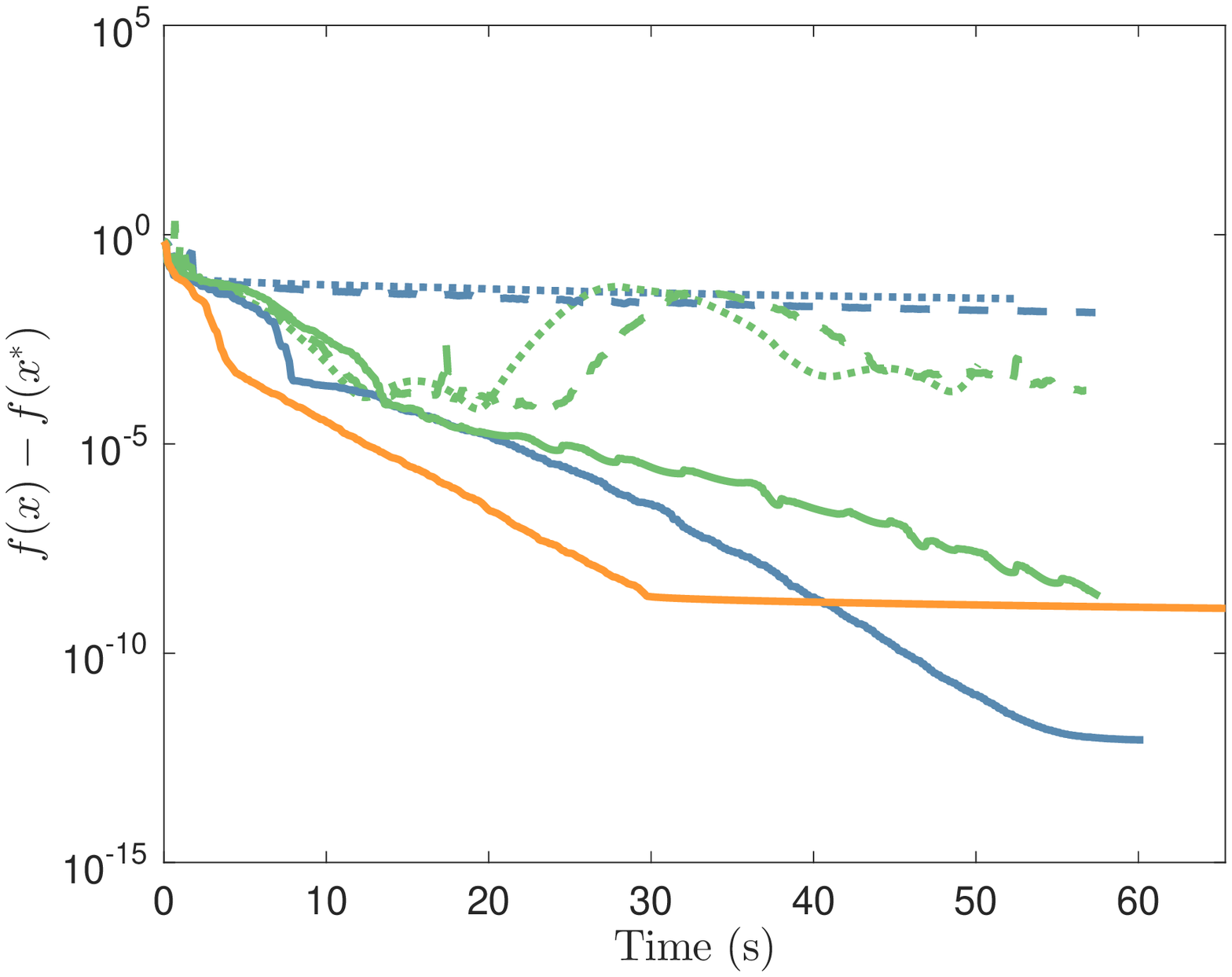}
\end{figure}

\clearpage
\subsection{Logistic regression, stochastic}
\label{sec:stoch}

\subsubsection{Dataset: sonar (Top to bottom: good, regular and bad condition number)}
\begin{figure}[ht]
\centering
\includegraphics[width=0.4\textwidth]{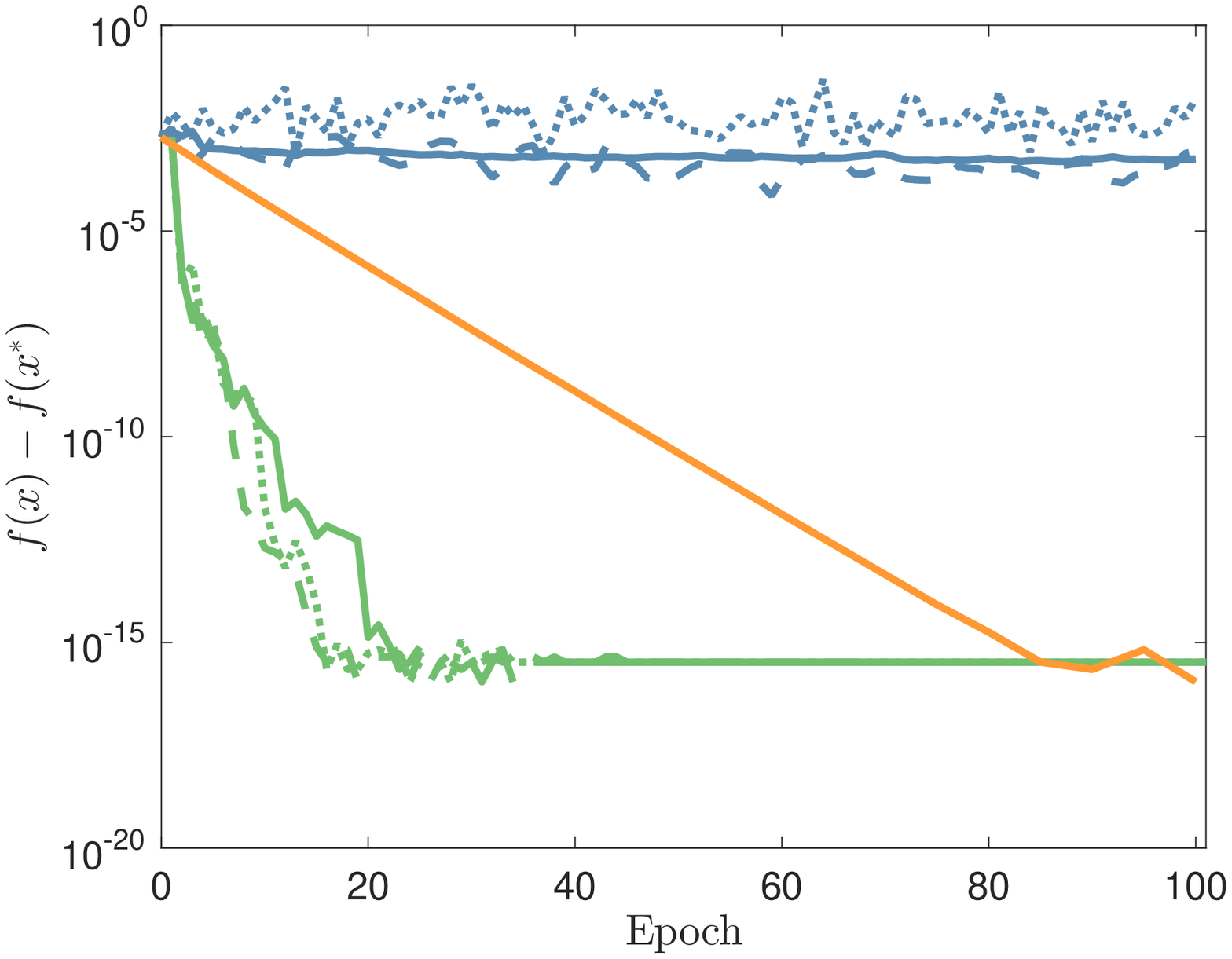}
\includegraphics[width=0.4\textwidth]{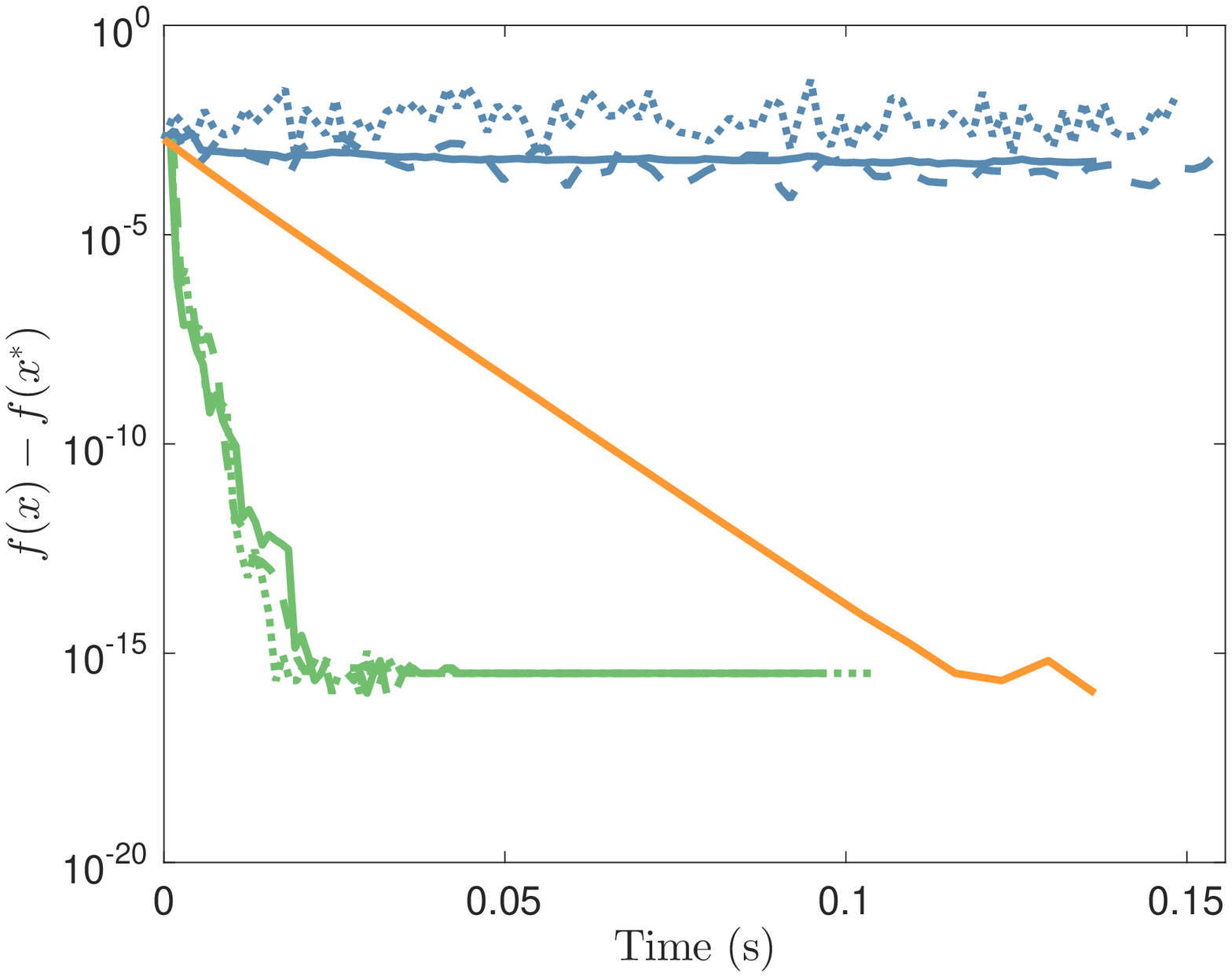}
\end{figure}
\begin{figure}[ht]
\centering
\includegraphics[width=0.4\textwidth]{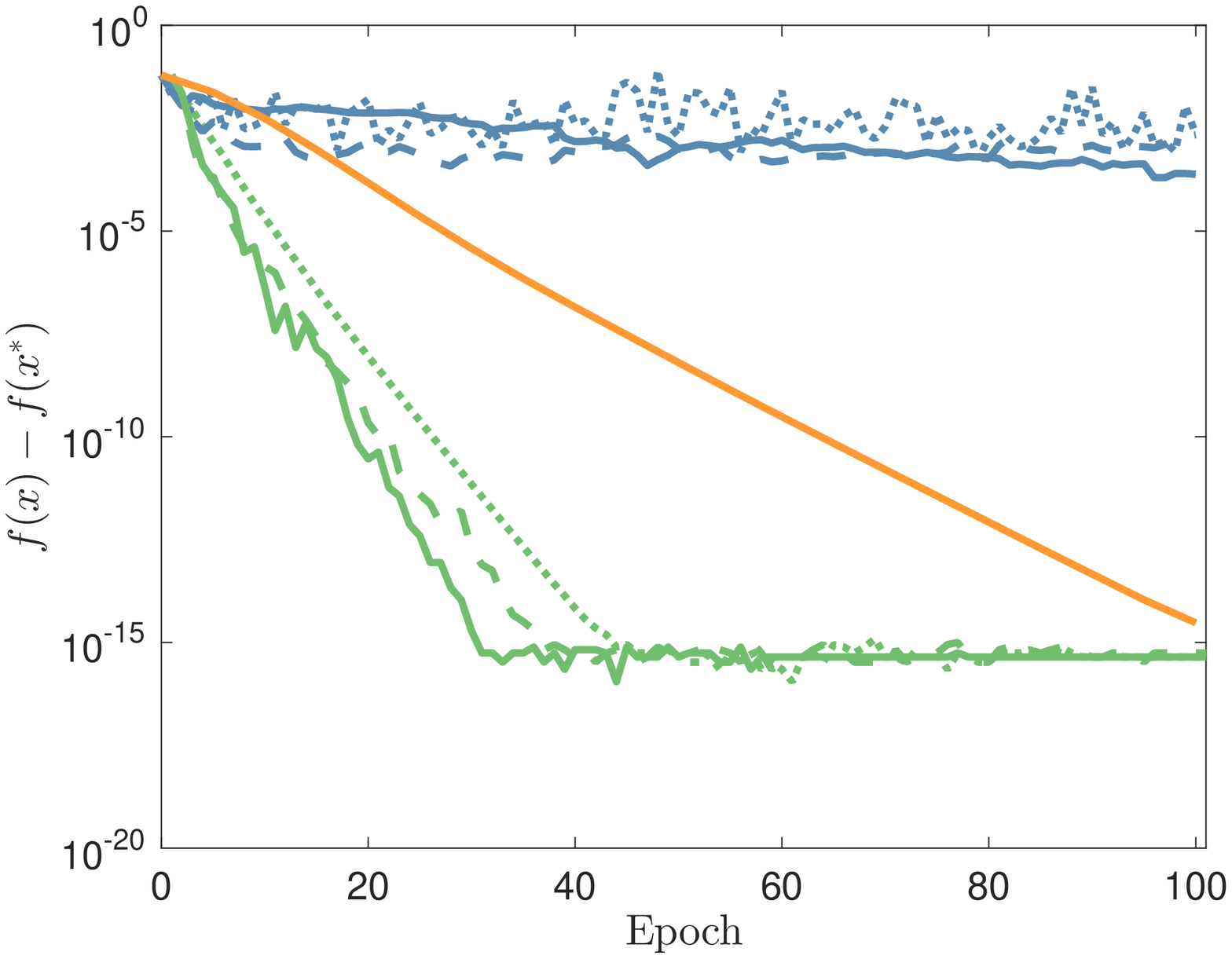}
\includegraphics[width=0.4\textwidth]{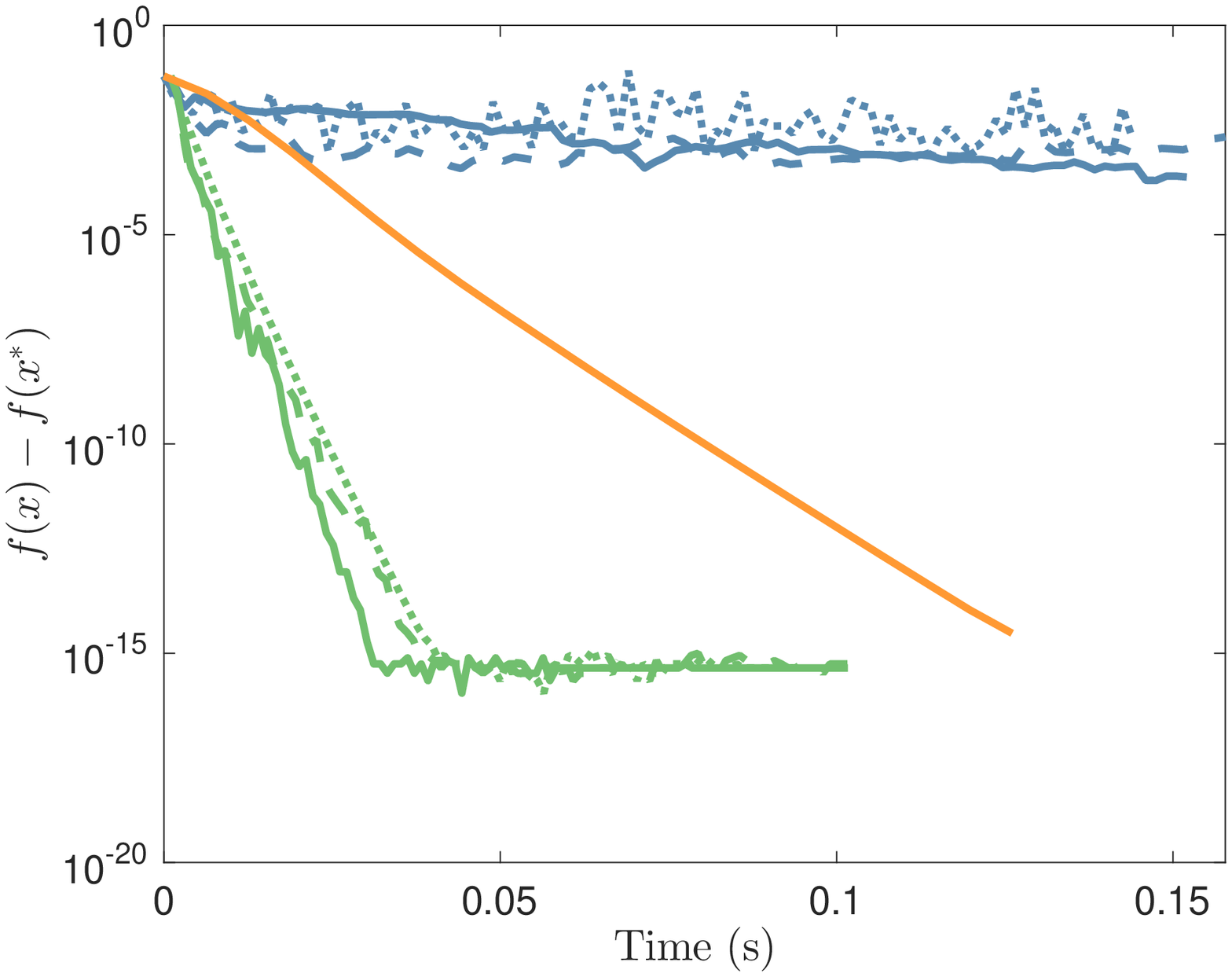}
\end{figure}
\begin{figure}[ht]
\centering
\includegraphics[width=0.4\textwidth]{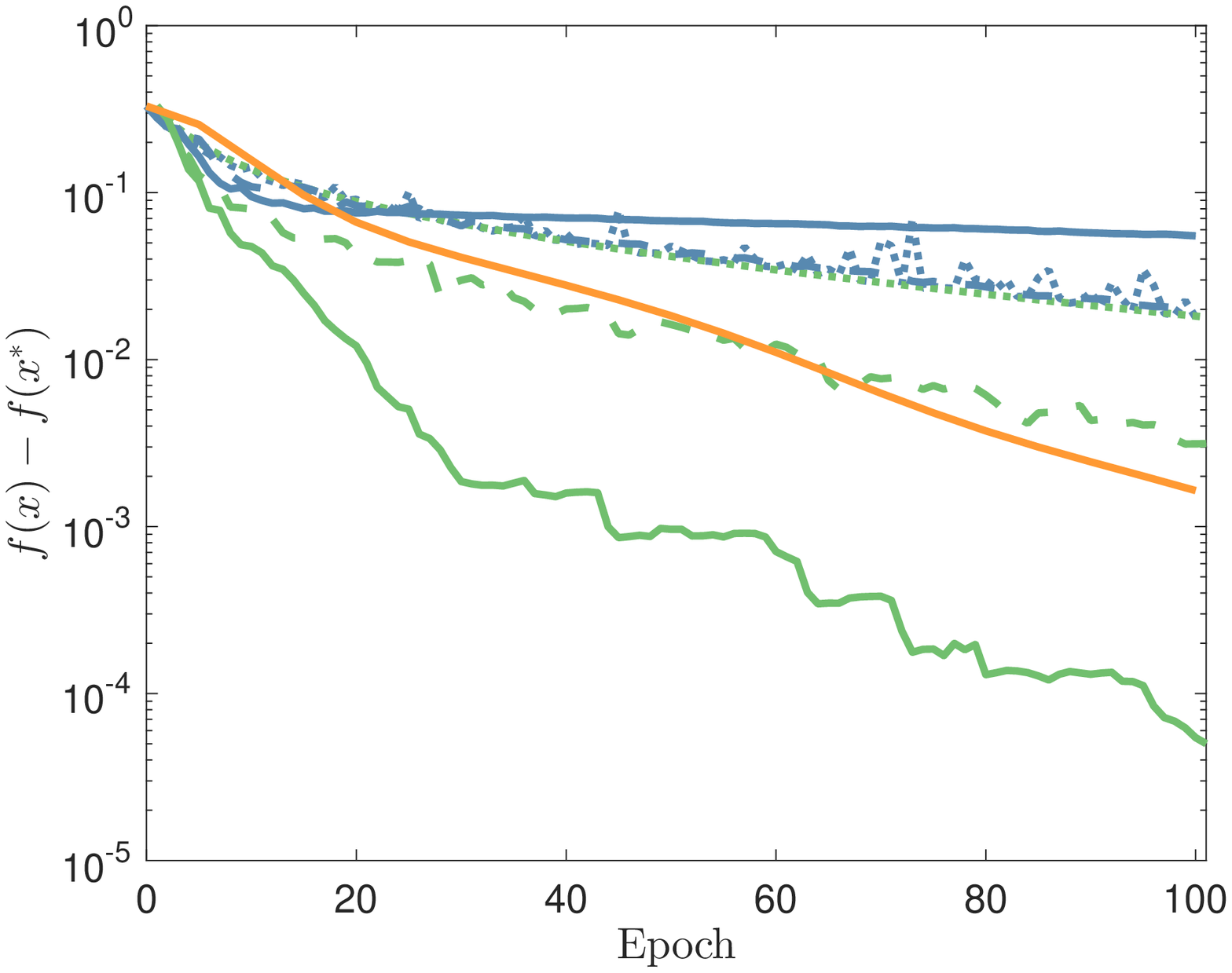}
\includegraphics[width=0.4\textwidth]{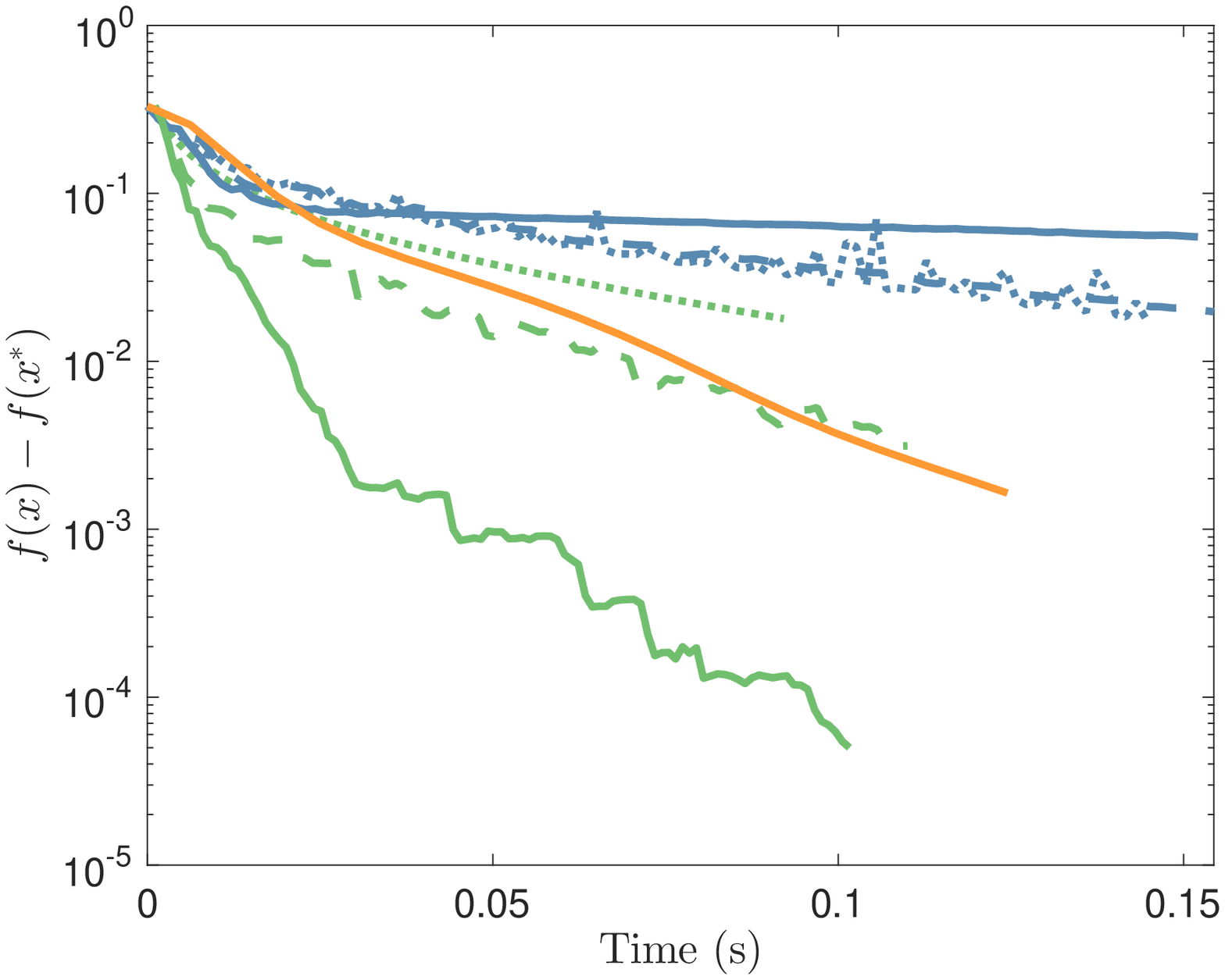}
\end{figure}
\clearpage
 
\subsubsection{Dataset: madelon (Top to bottom: good, regular and bad condition number)}
\begin{figure}[ht]
\centering
\includegraphics[width=0.4\textwidth]{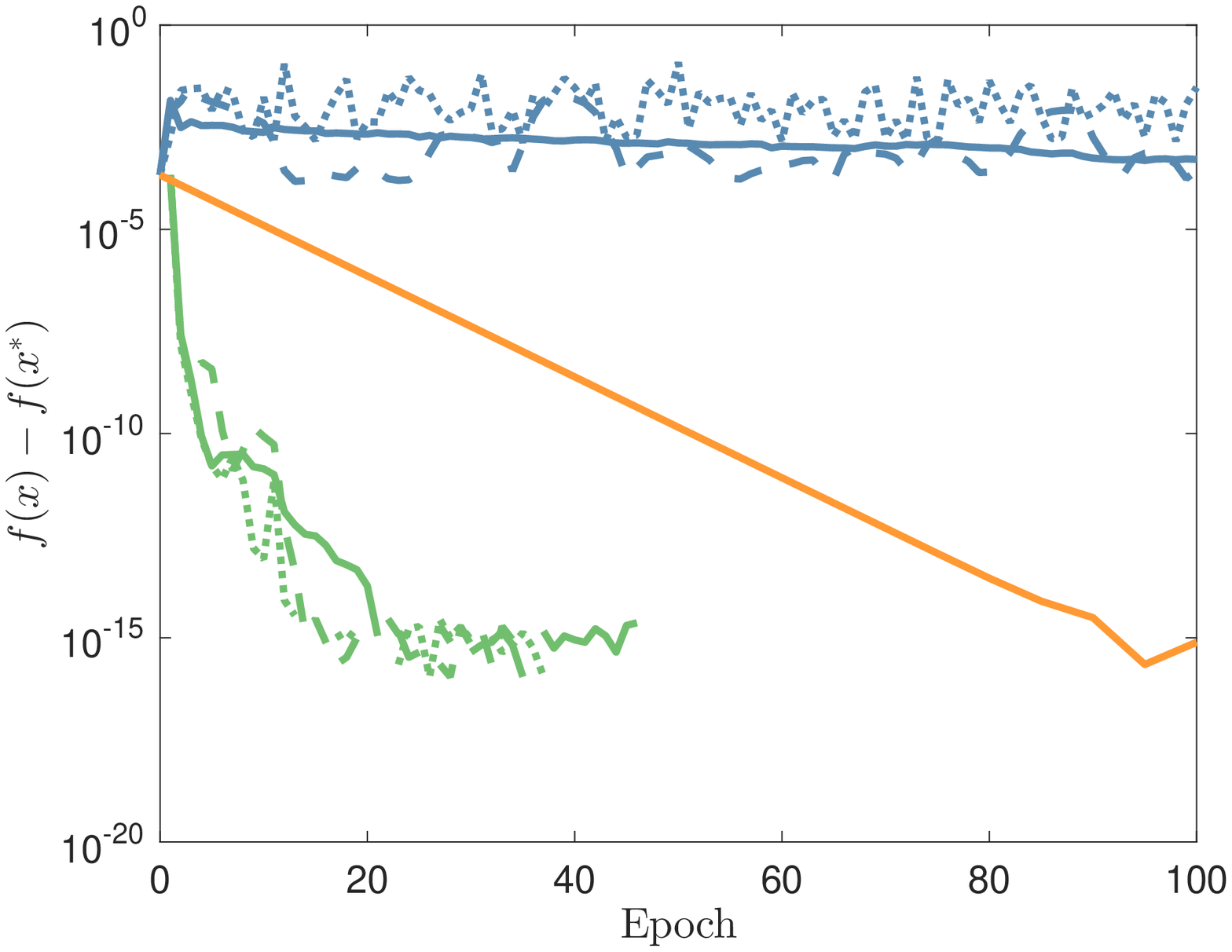}
\includegraphics[width=0.4\textwidth]{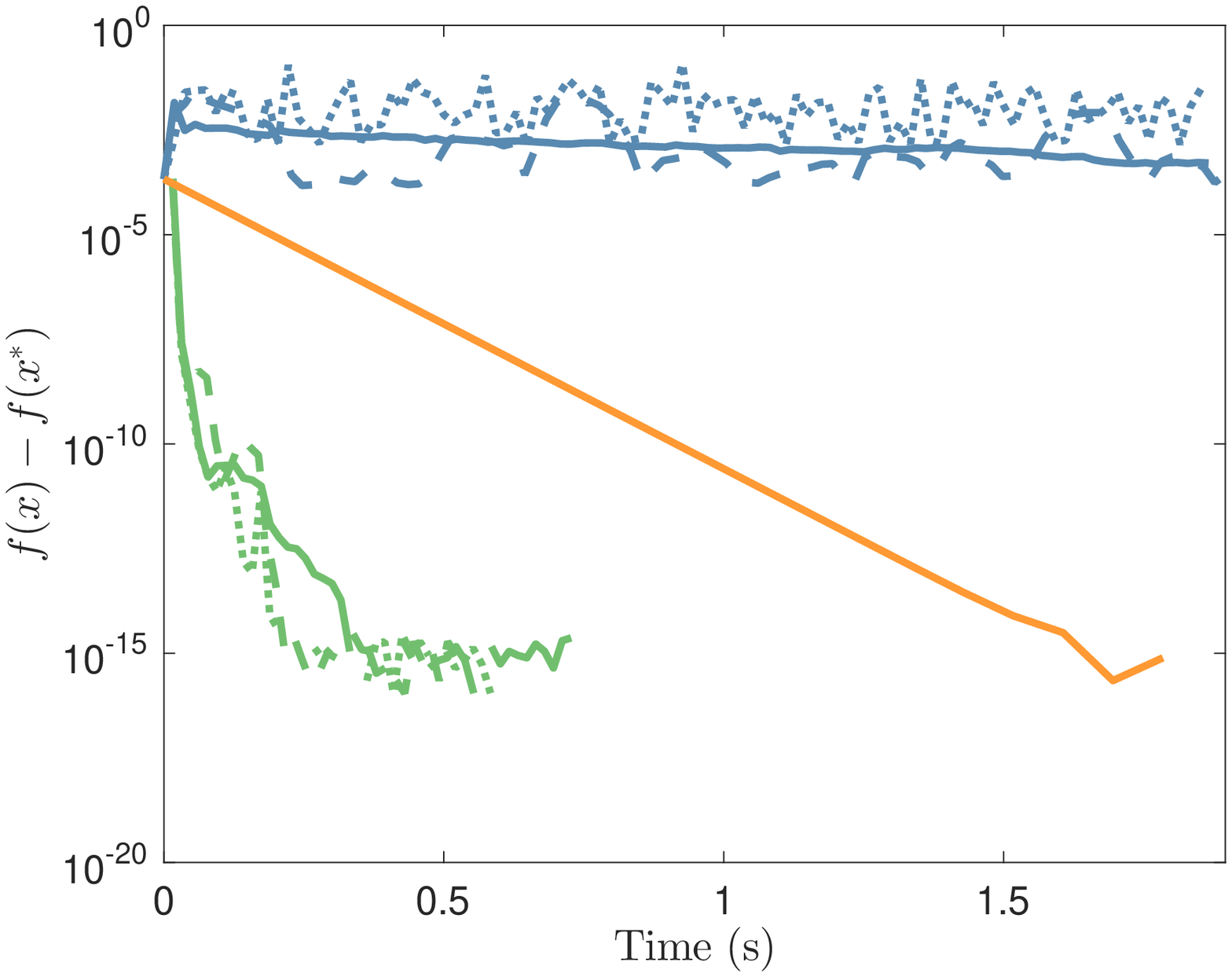}
\end{figure}
\begin{figure}[ht]
\centering
\includegraphics[width=0.4\textwidth]{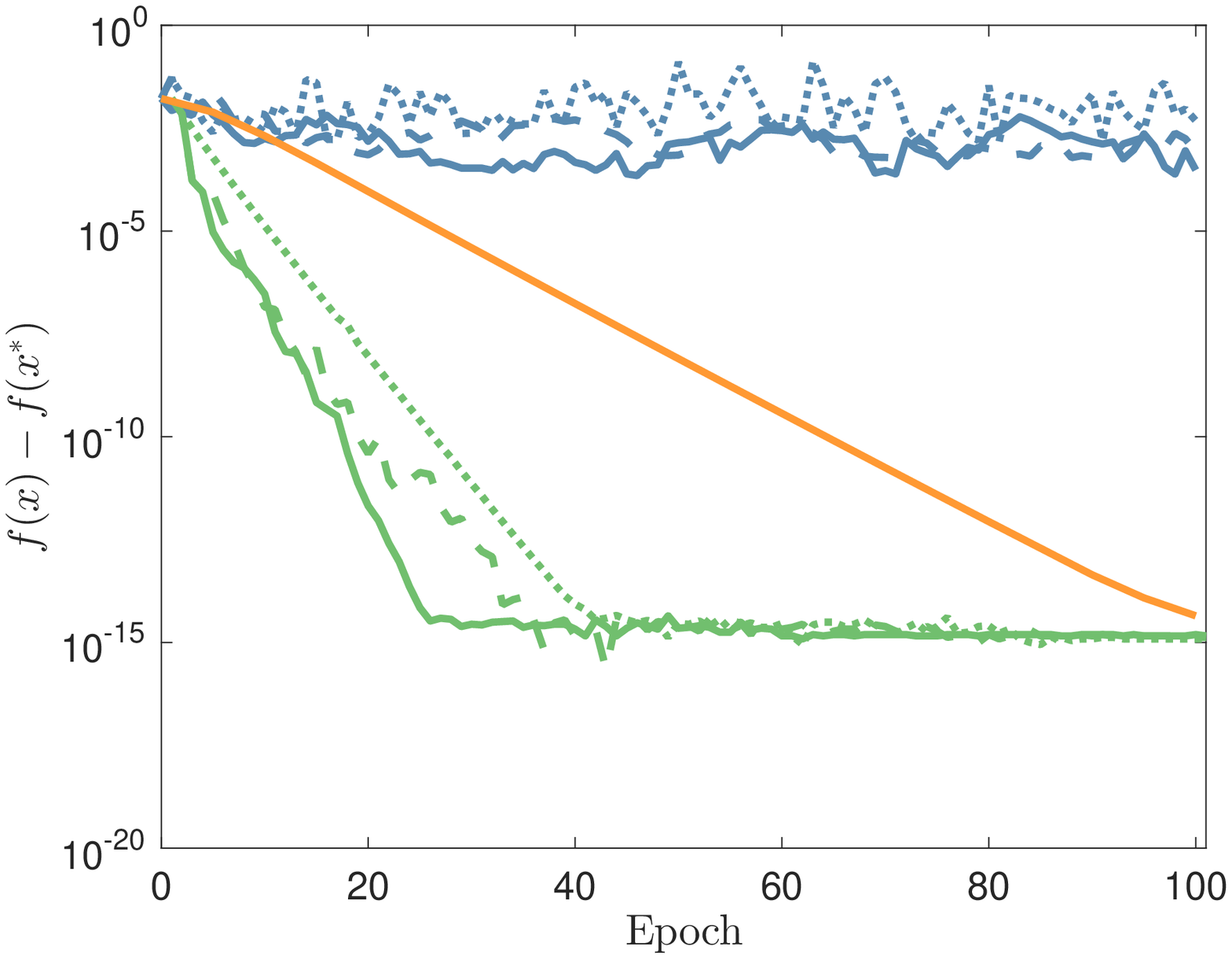}
\includegraphics[width=0.4\textwidth]{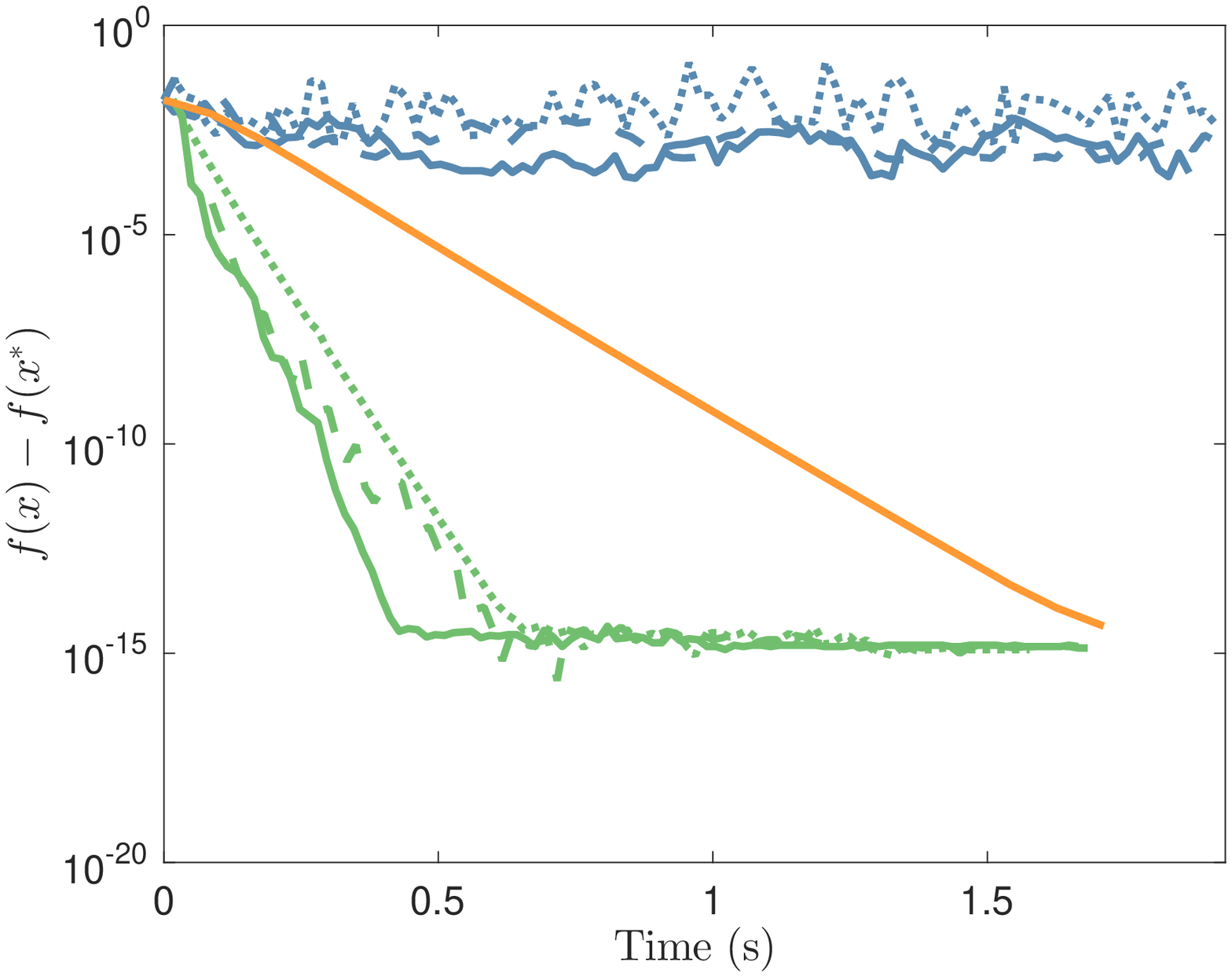}
\end{figure}
\begin{figure}[ht]
\centering
\includegraphics[width=0.4\textwidth]{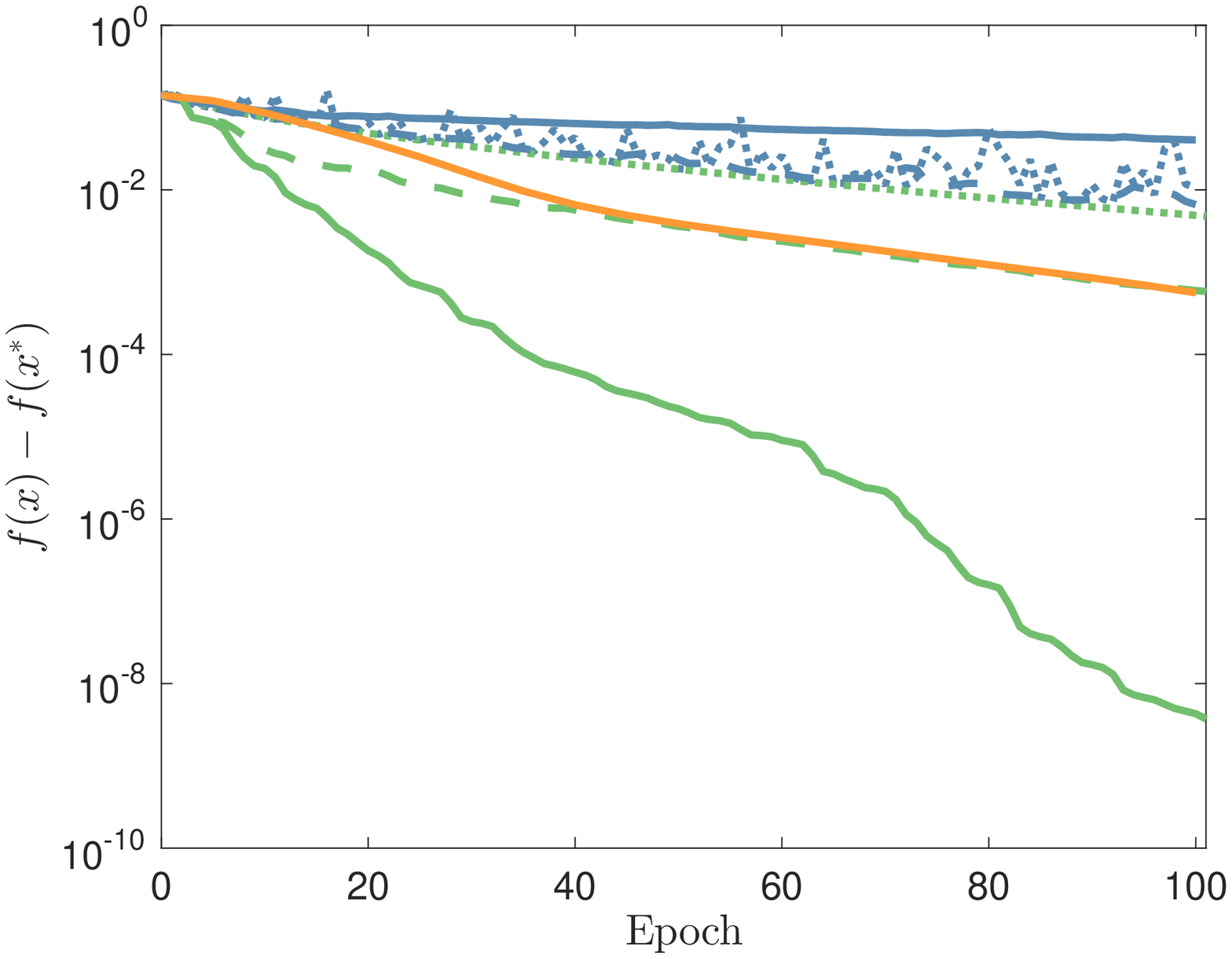}
\includegraphics[width=0.4\textwidth]{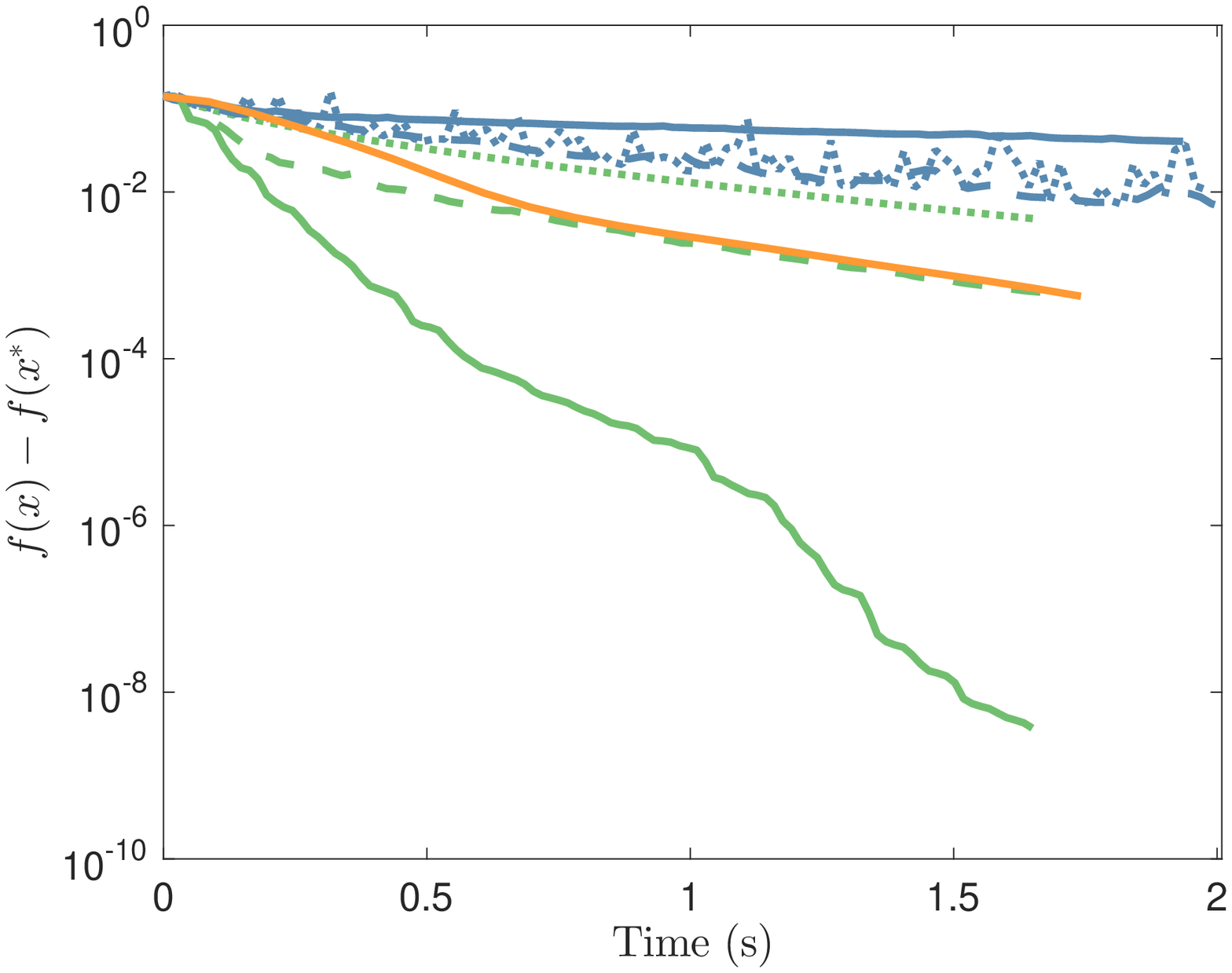}
\end{figure}
\clearpage
 
\subsubsection{Dataset: sido0 (Top to bottom: good, regular and bad condition number)}
\begin{figure}[ht]
\centering
\includegraphics[width=0.4\textwidth]{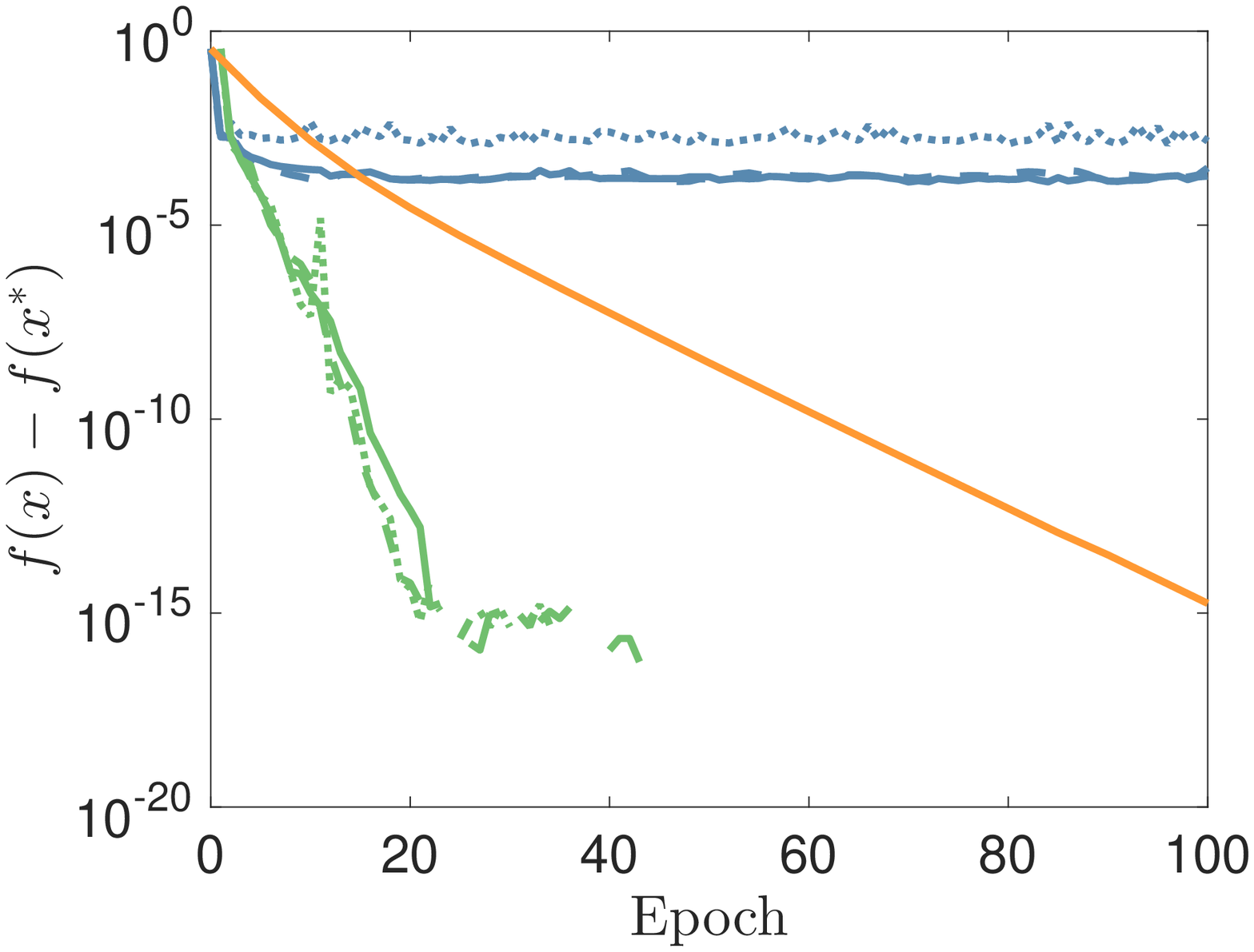}
\includegraphics[width=0.4\textwidth]{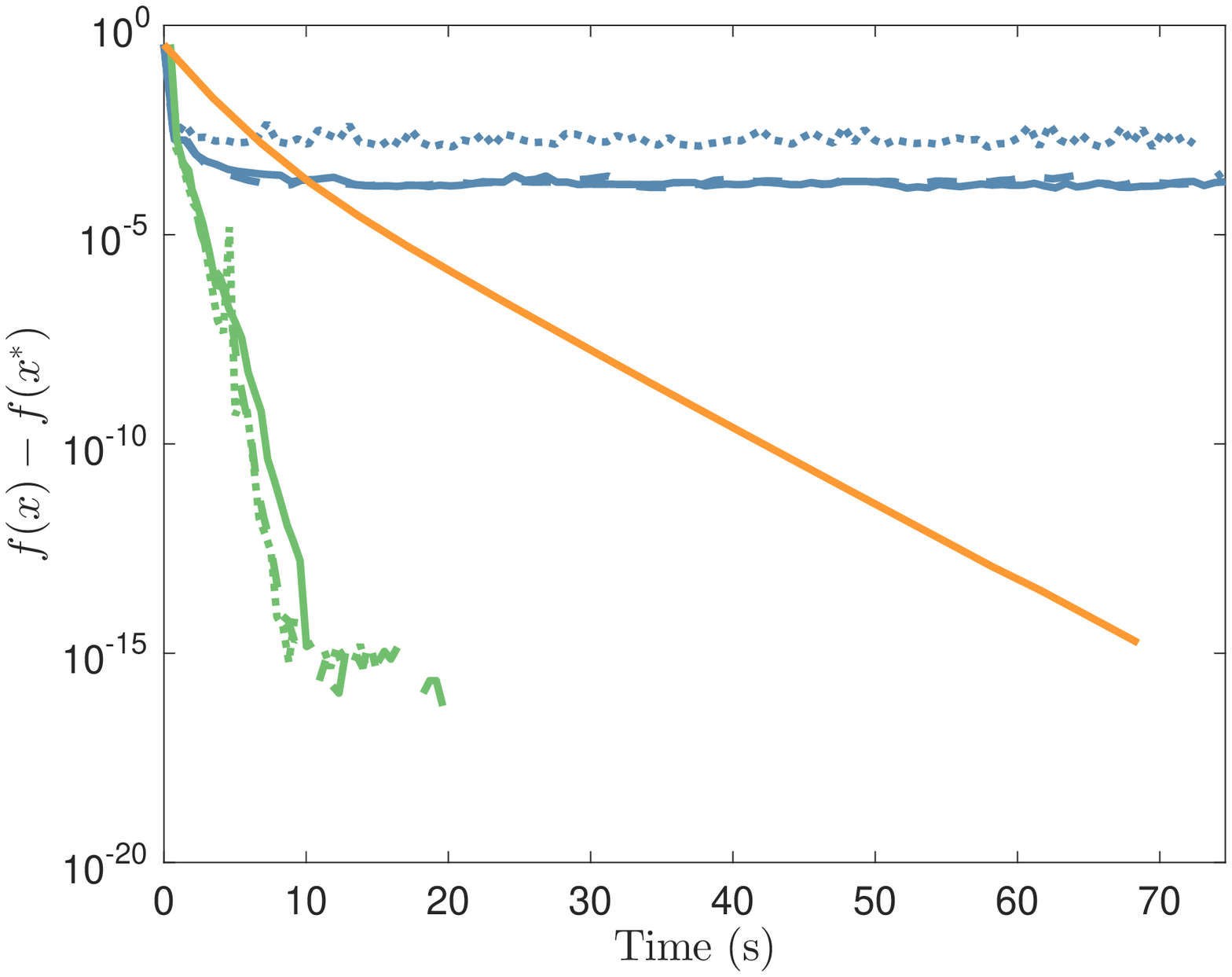}
\end{figure}
\begin{figure}[ht]
\centering
\includegraphics[width=0.4\textwidth]{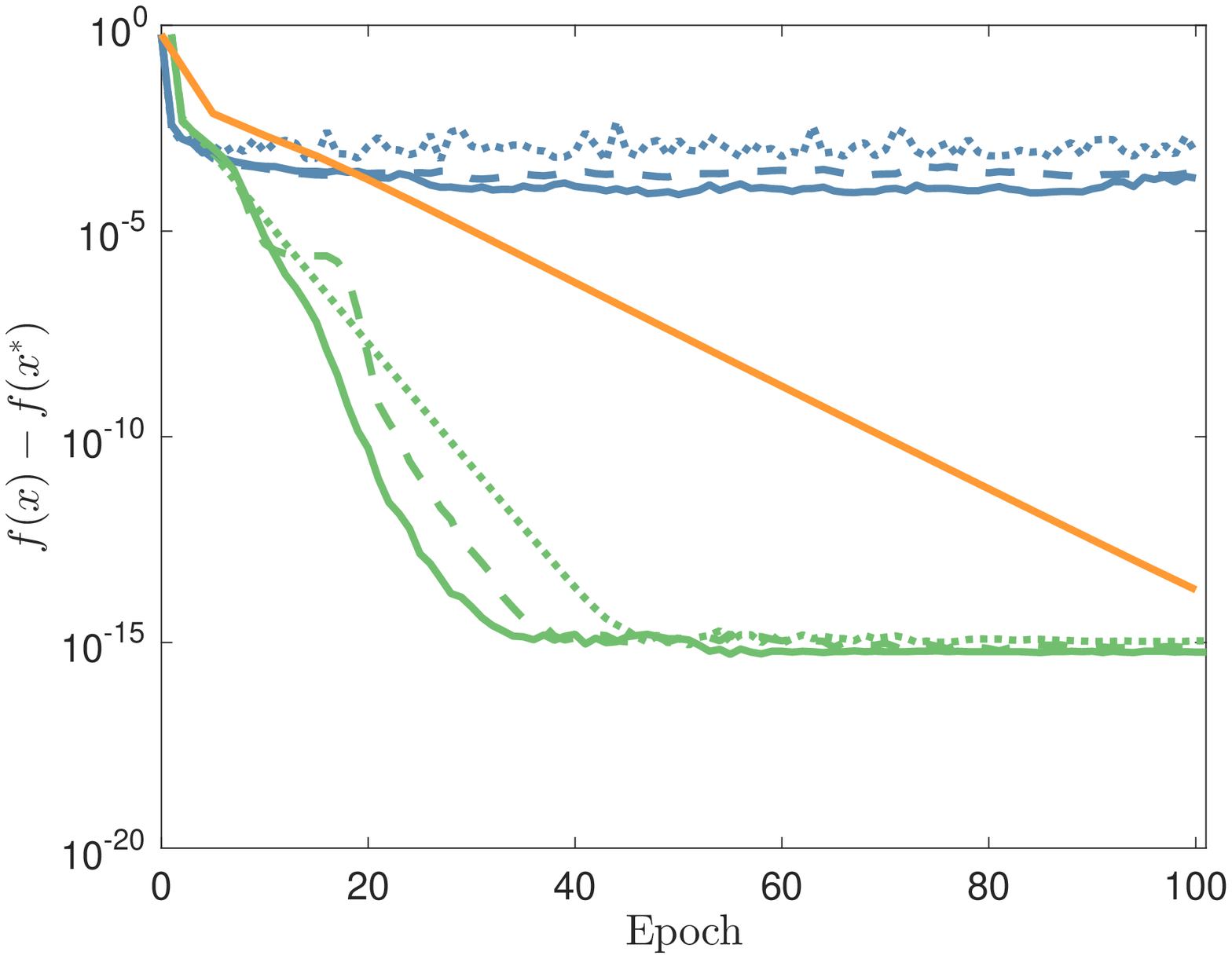}
\includegraphics[width=0.4\textwidth]{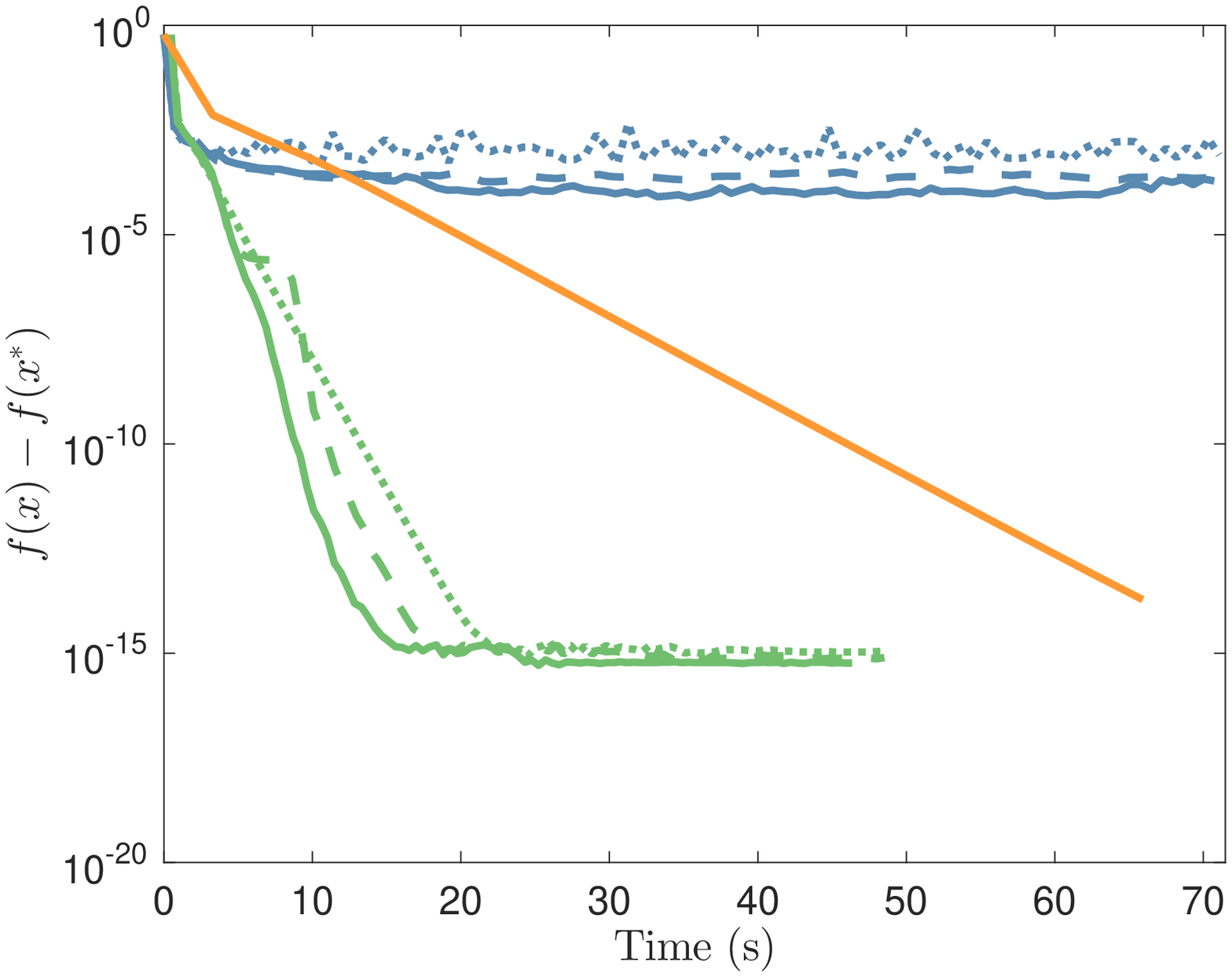}
\end{figure}
\begin{figure}[ht]
\centering
\includegraphics[width=0.4\textwidth]{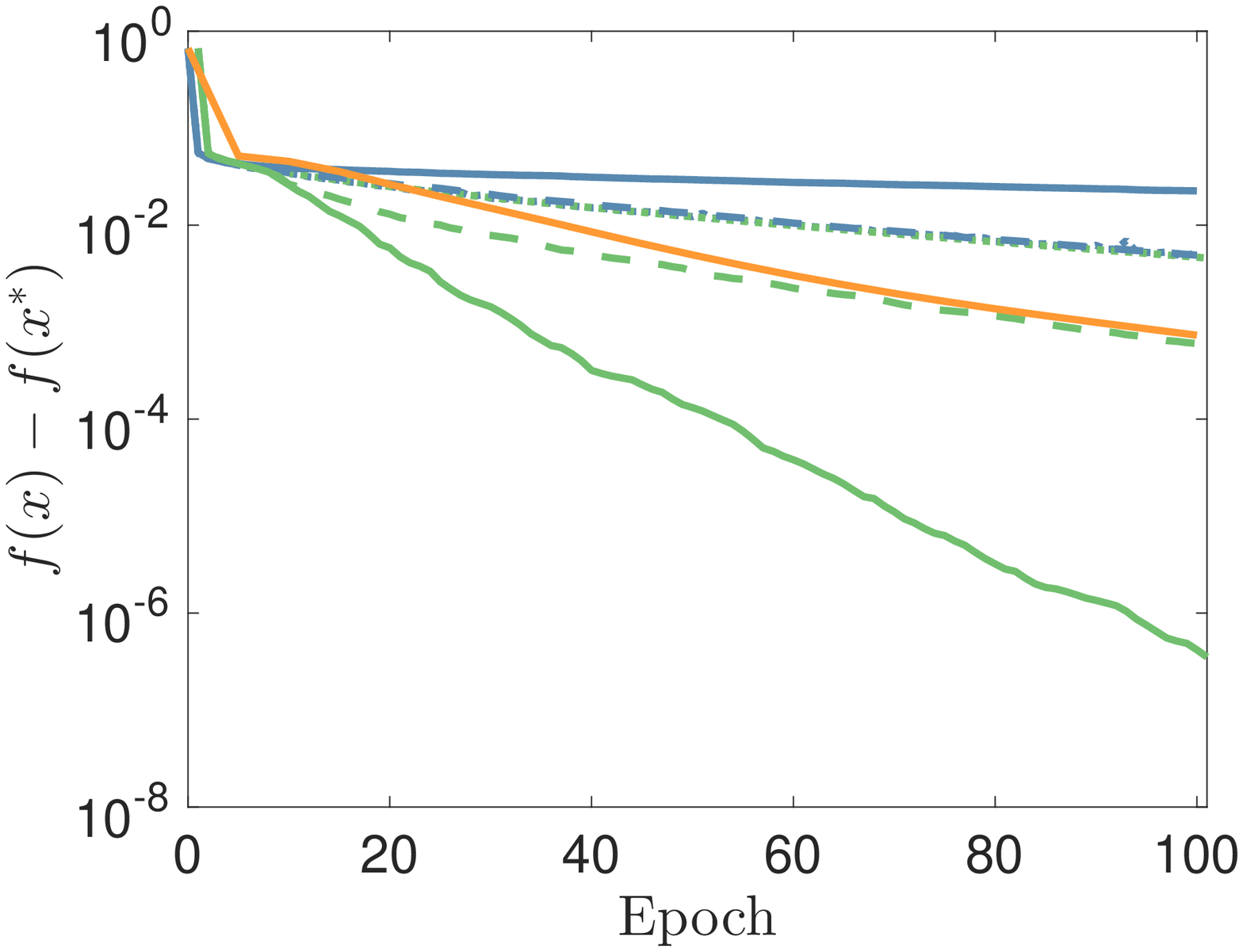}
\includegraphics[width=0.4\textwidth]{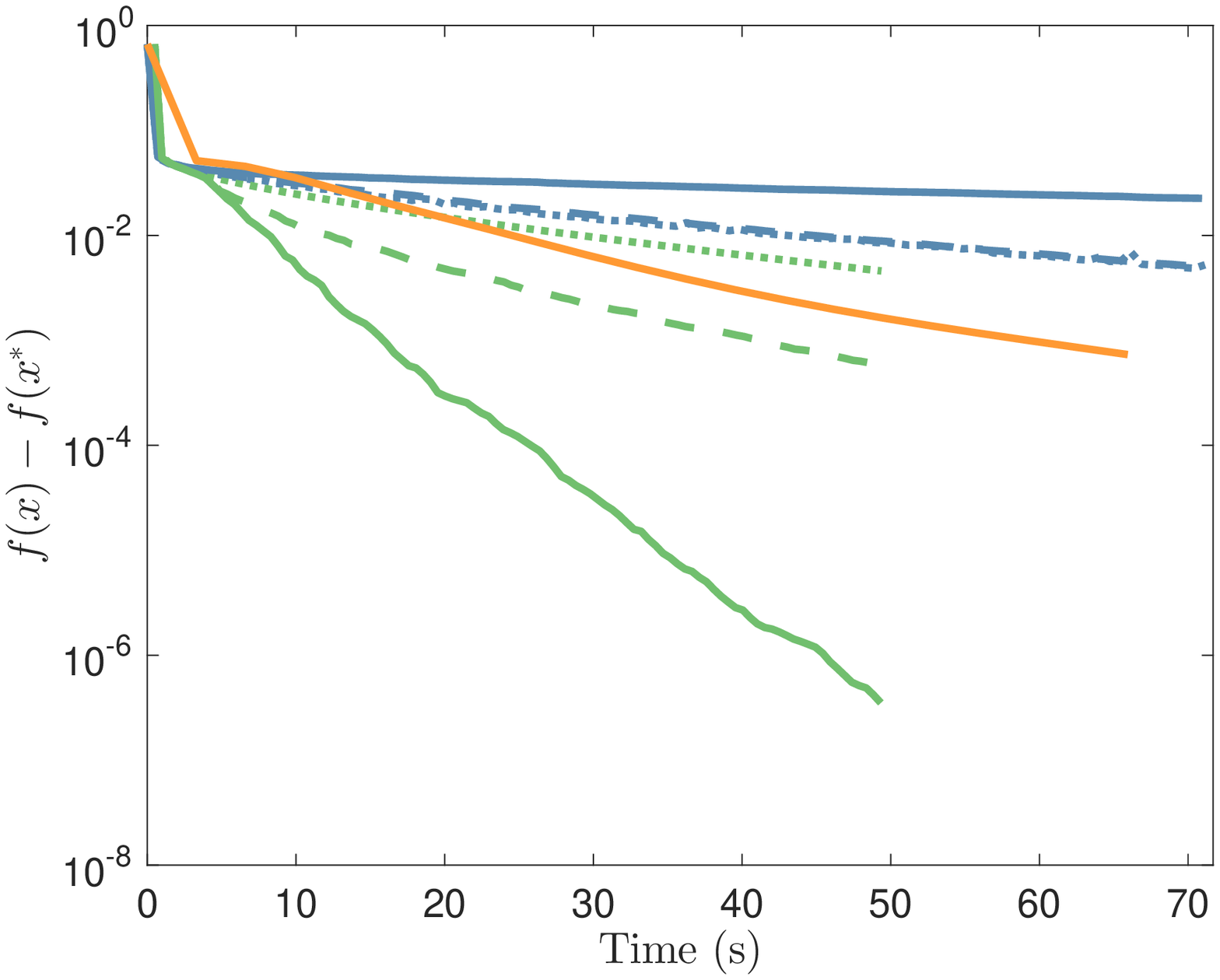}
\end{figure}

\clearpage
\section{Logistic Regression on MNIST Dataset}
\label{sec:mnist}

We now present our result for optimizing a logistic regression using the MNIST dataset. MNIST contains handwritten digits, and consists in 10 different classes of points, each containing 6000 samples with 784 features. We train each class separately, using a one-versus-all approach. Use compare plain gradient method with backtracking line-search with its accelerated version using online RNA (with $N=10$ and $\lambda = 10^{-8}$). In all cases, RNA outperform gradient descent and Nesterov's fast gradient.

\begin{figure}[ht]
\centering
\includegraphics[width=0.3\textwidth]{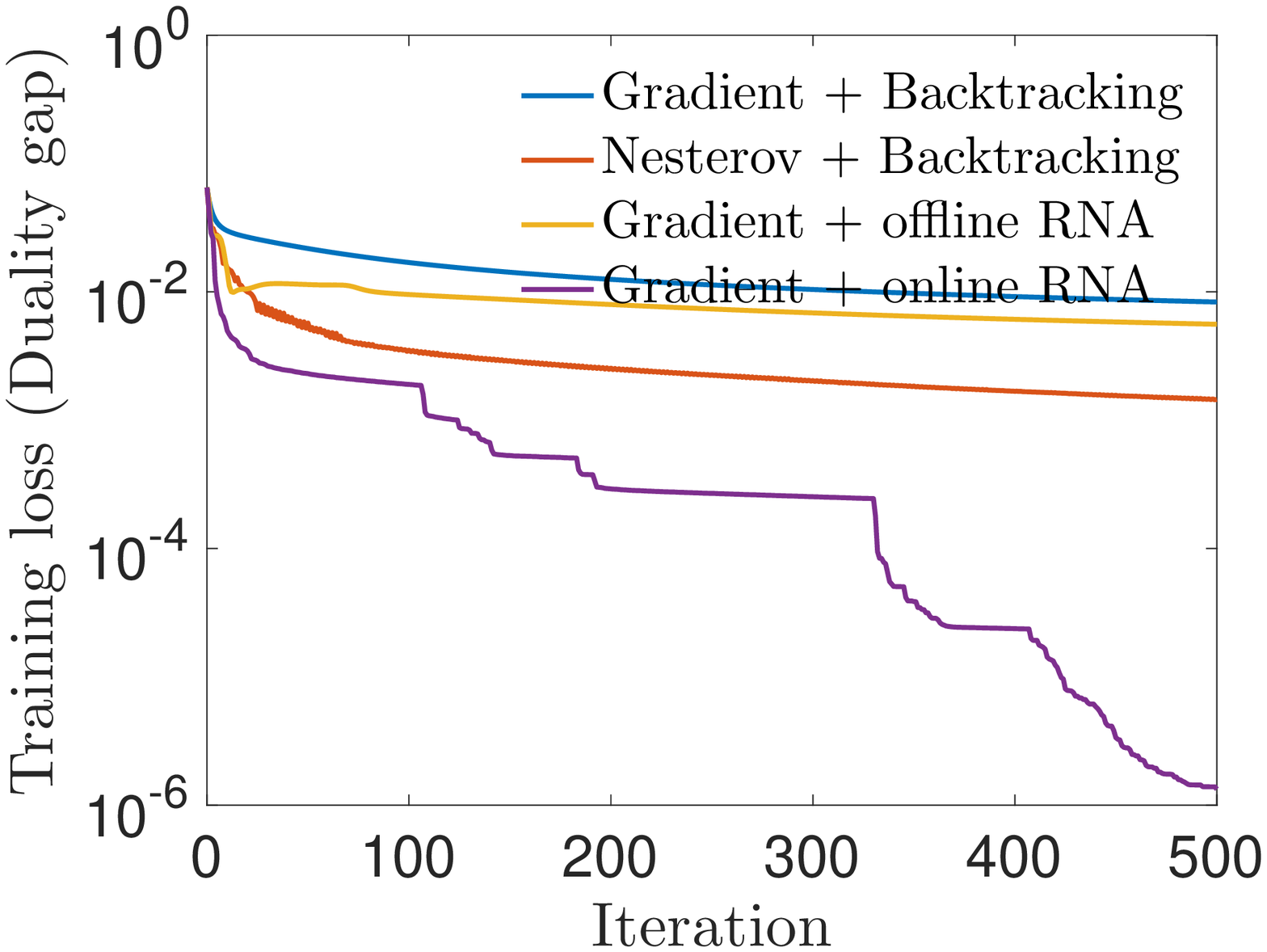}
\includegraphics[width=0.3\textwidth]{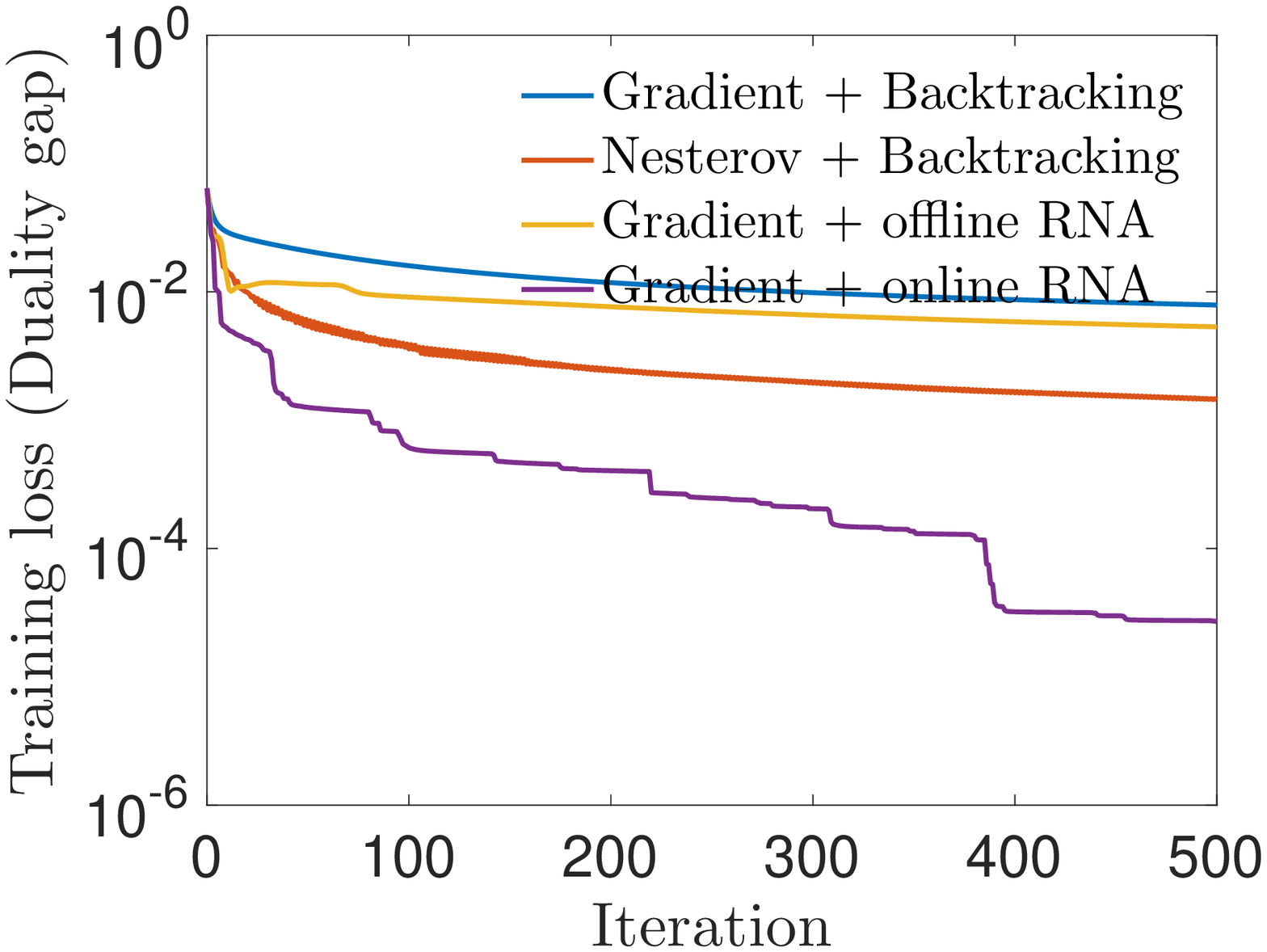}
\includegraphics[width=0.3\textwidth]{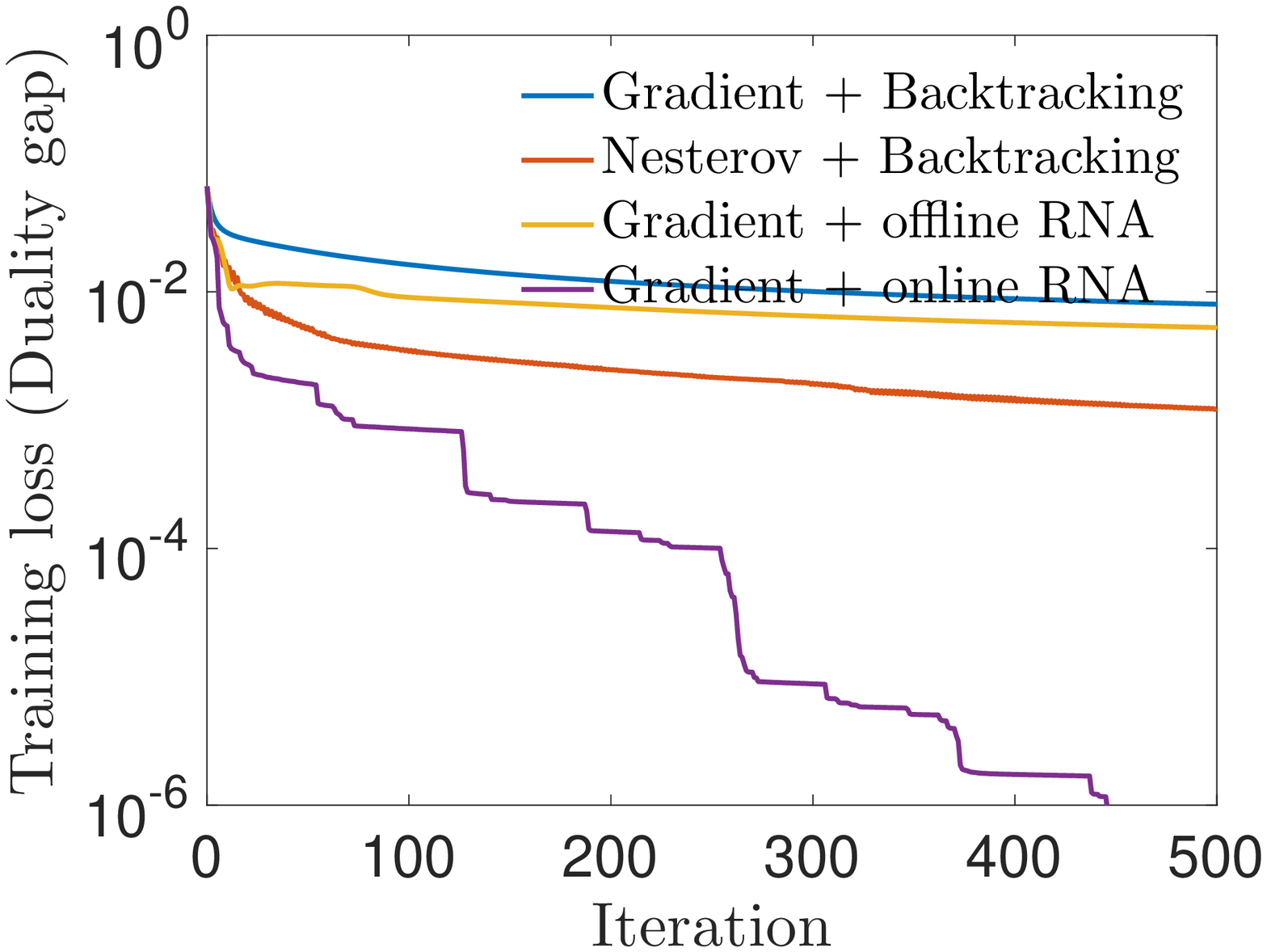}
\end{figure}
\begin{figure}[ht]
\centering
\includegraphics[width=0.3\textwidth]{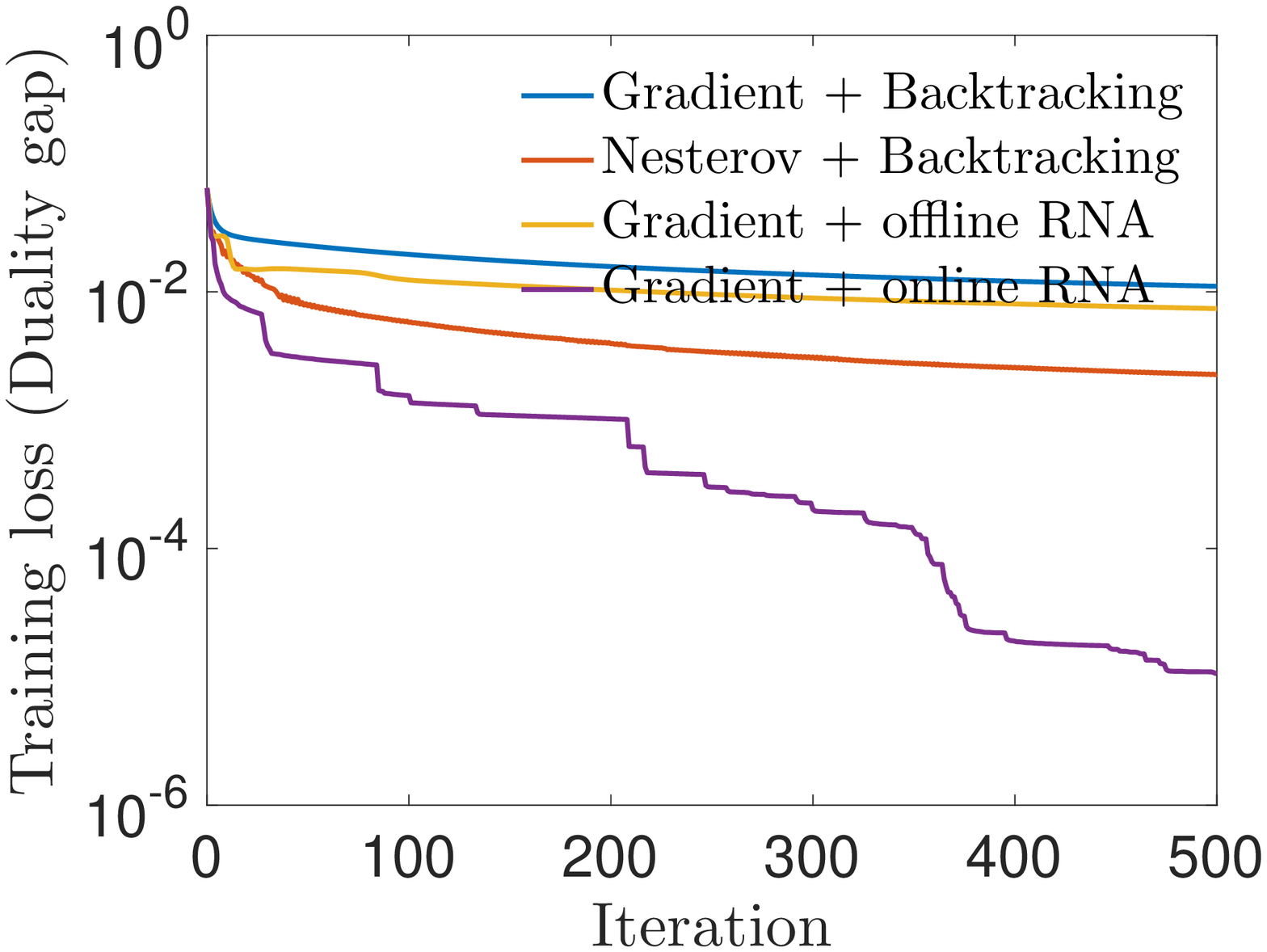}
\includegraphics[width=0.3\textwidth]{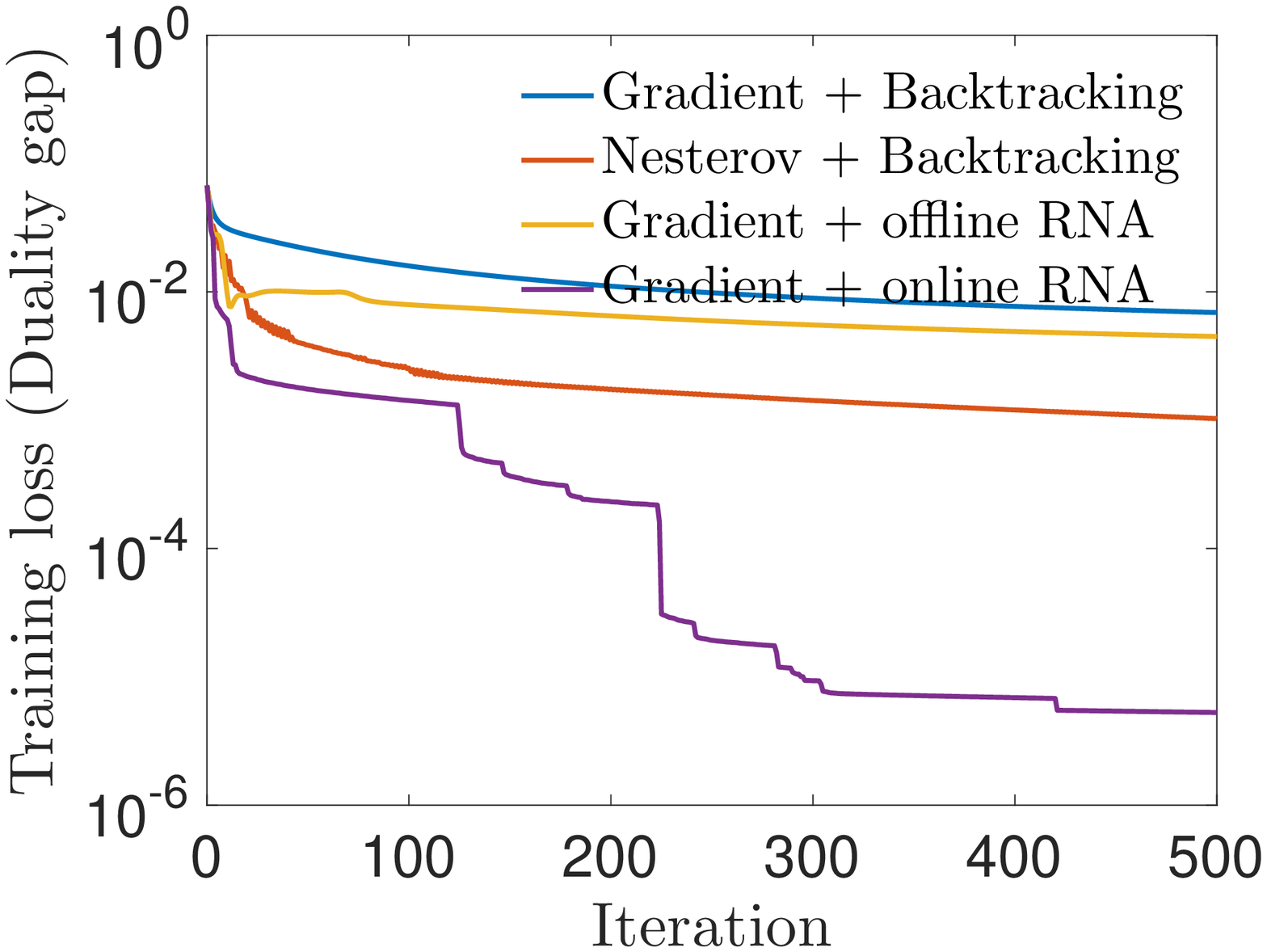}
\includegraphics[width=0.3\textwidth]{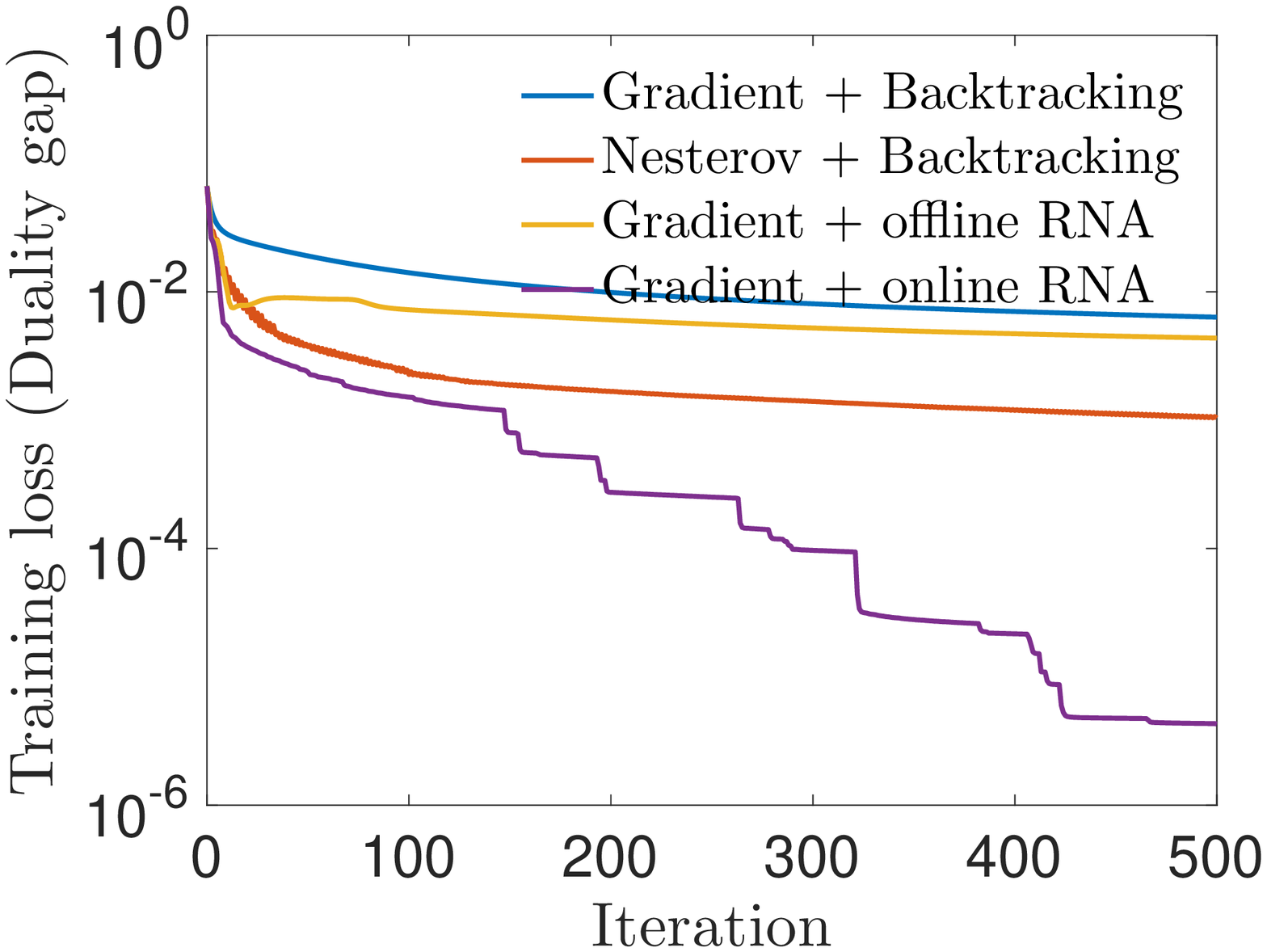}
\end{figure}
\begin{figure}[ht]
\centering
\includegraphics[width=0.3\textwidth]{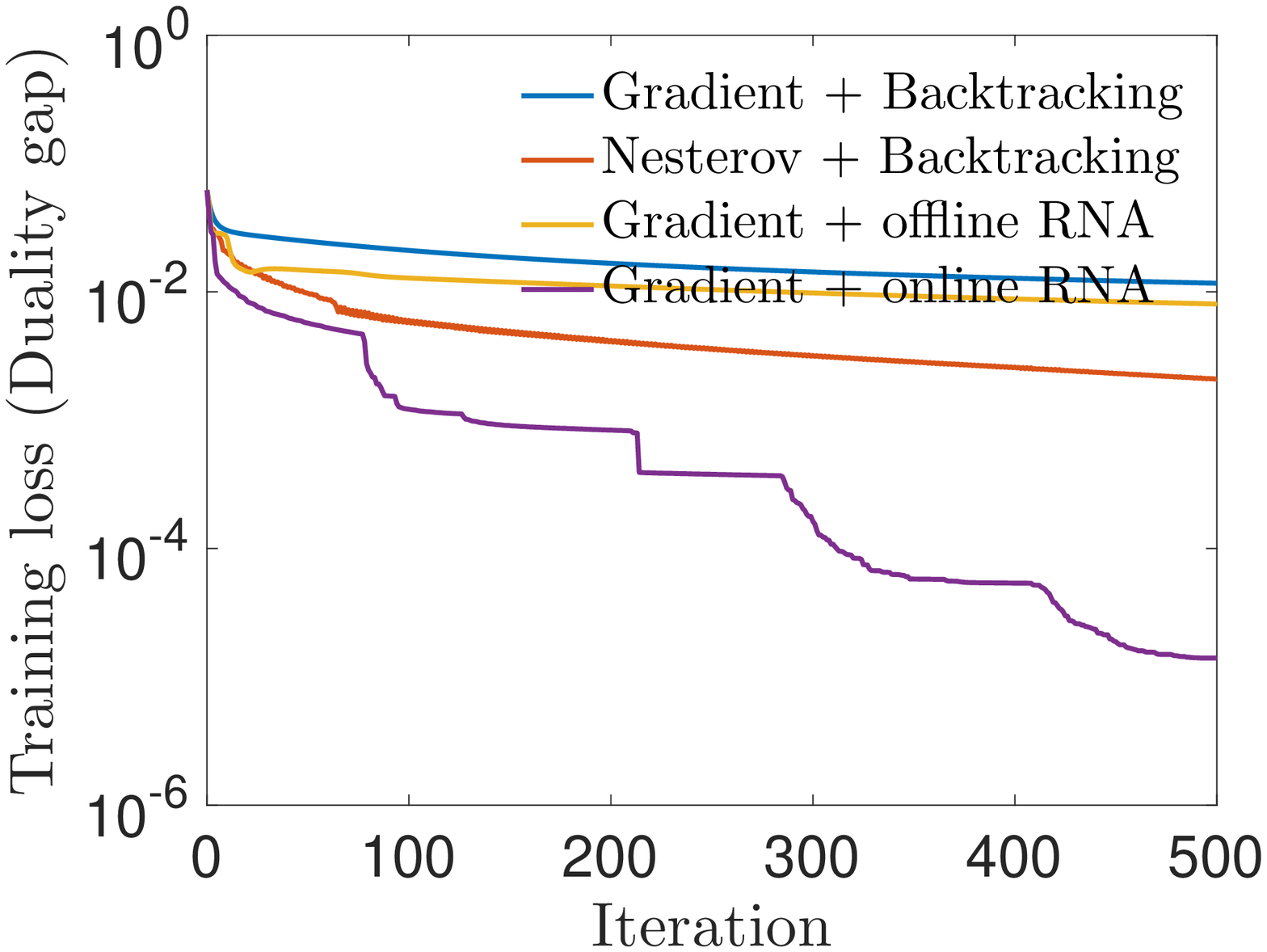}
\includegraphics[width=0.3\textwidth]{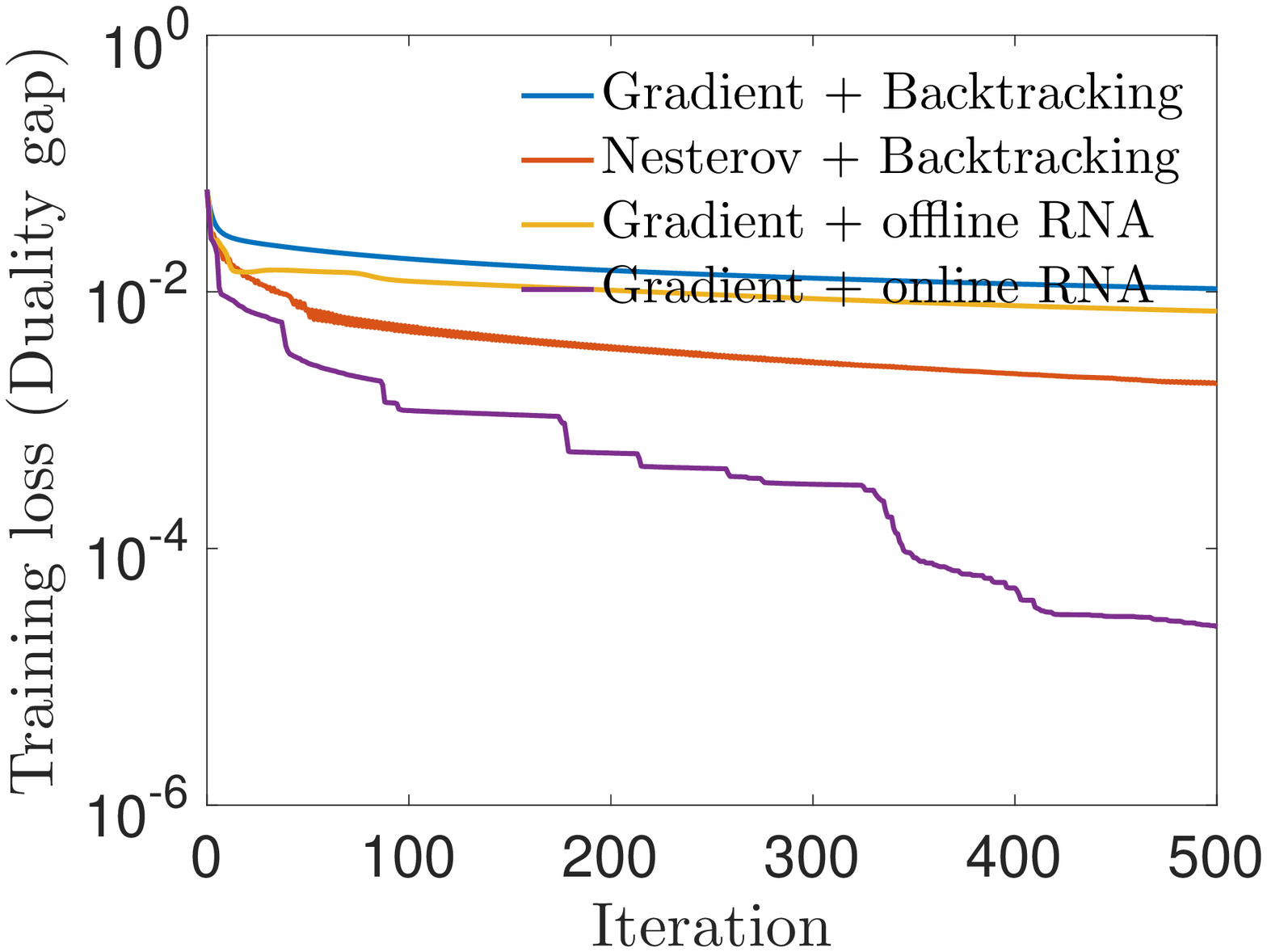}
\includegraphics[width=0.3\textwidth]{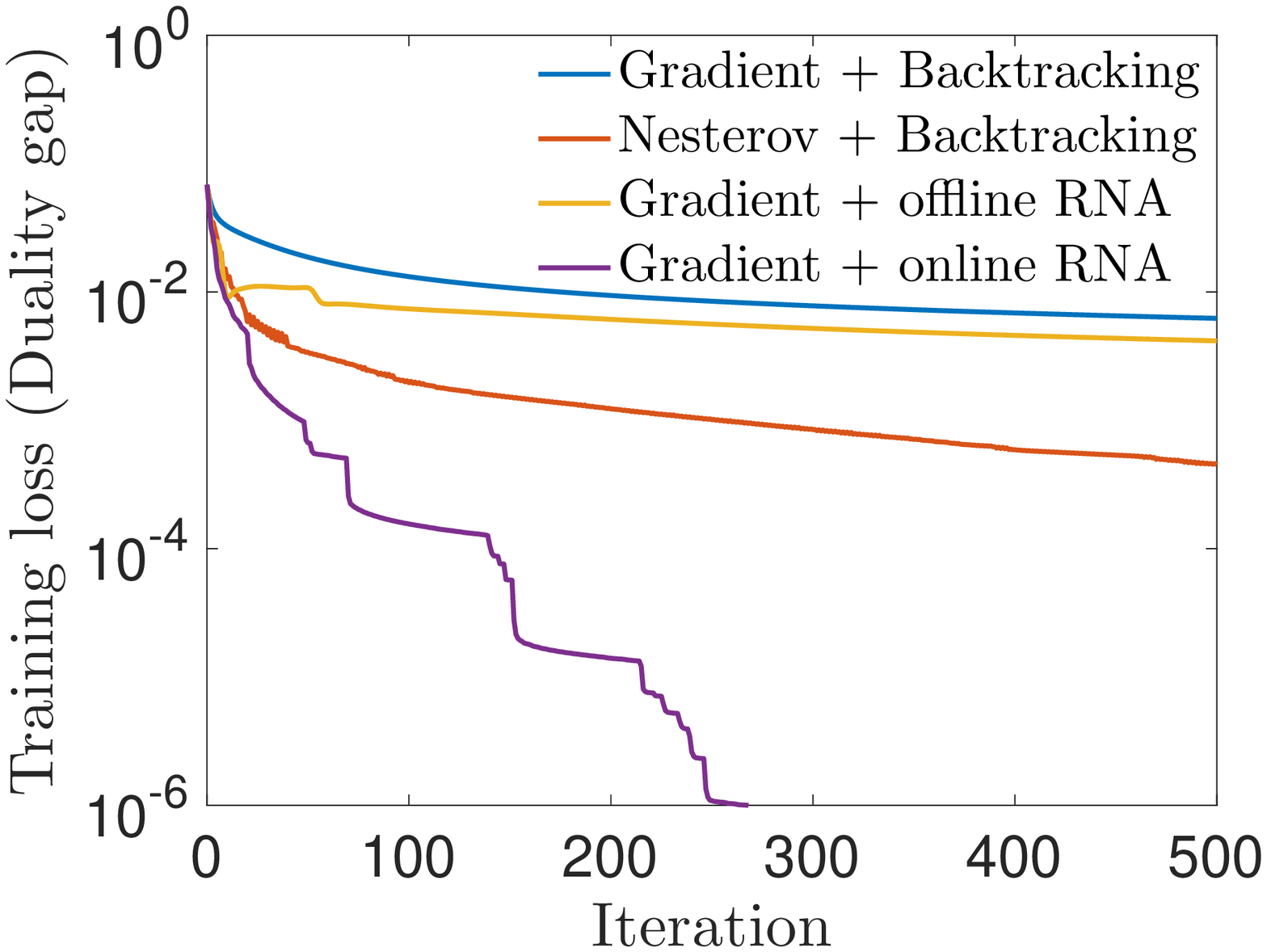}
\end{figure}
\begin{figure}[h!t]
\centering
\includegraphics[width=0.3\textwidth]{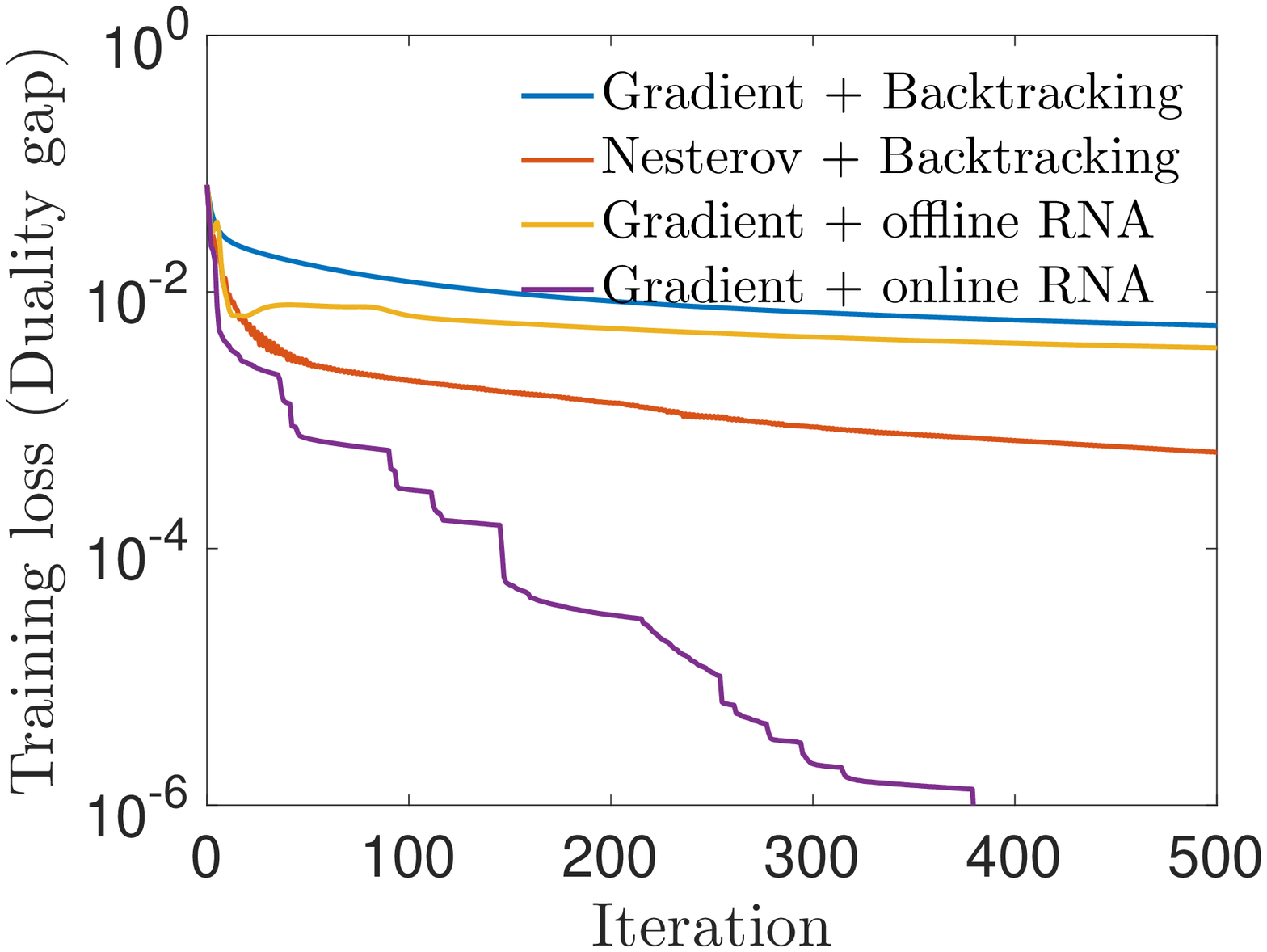}
\caption{From left to right, top to bottom: Loss value when optimizing a logistic regression (one vs all) on MNIST dataset for the numbers 1, 2, 3, 4, 5, 6, 7, 8, 9, 0.}
\end{figure}

\begin{figure}[ht]
\centering
\includegraphics[width=0.3\textwidth]{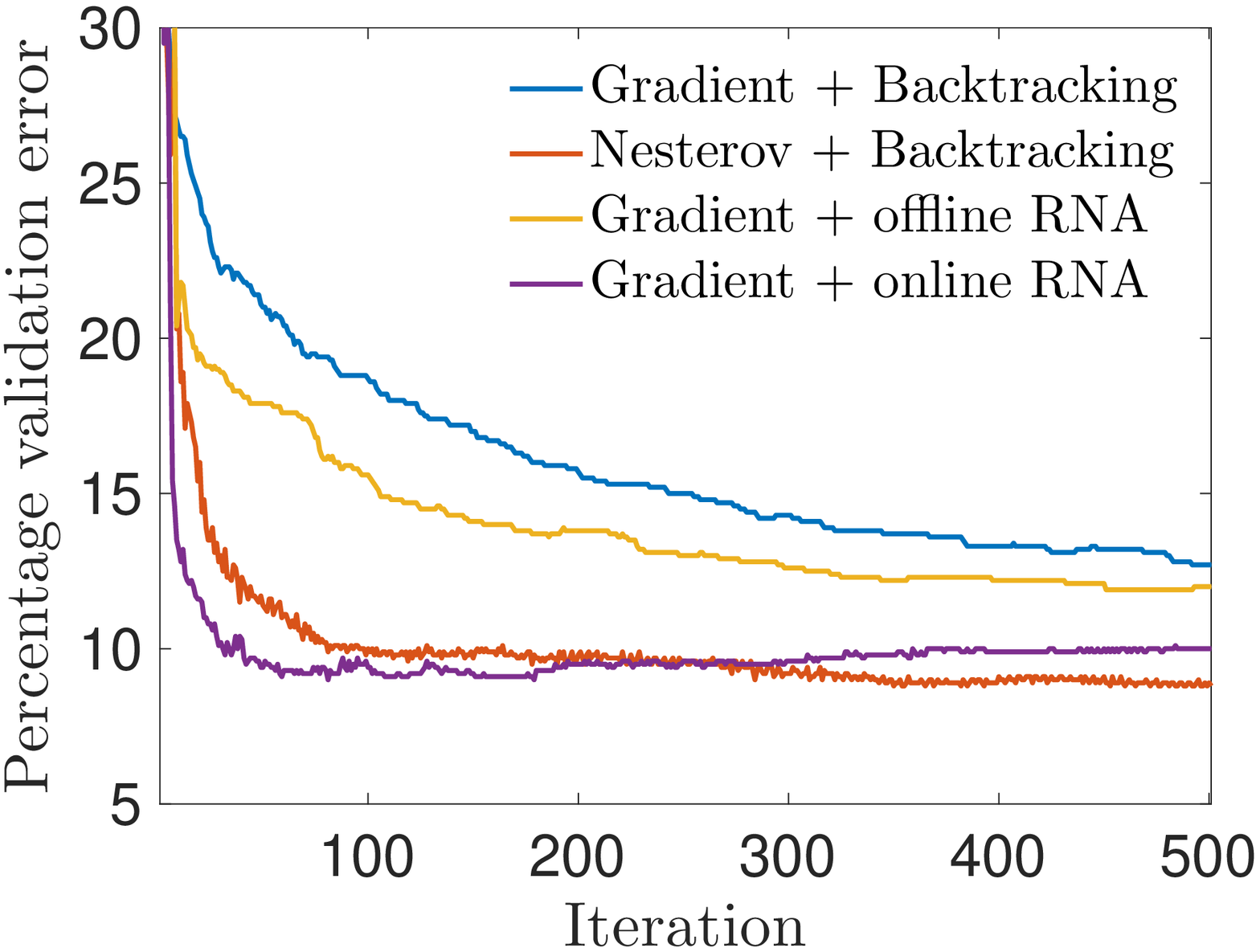}
\includegraphics[width=0.3\textwidth]{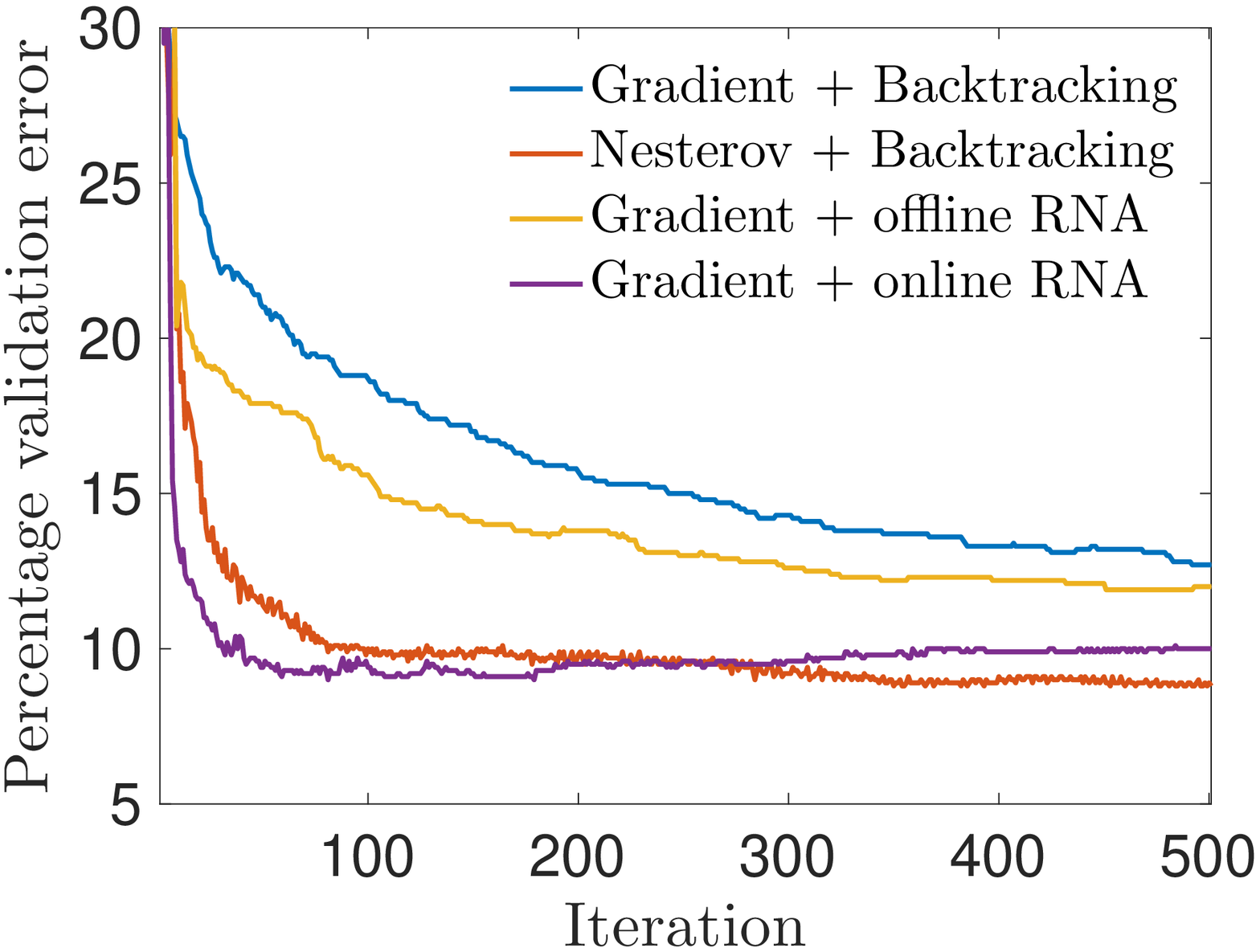}
\includegraphics[width=0.3\textwidth]{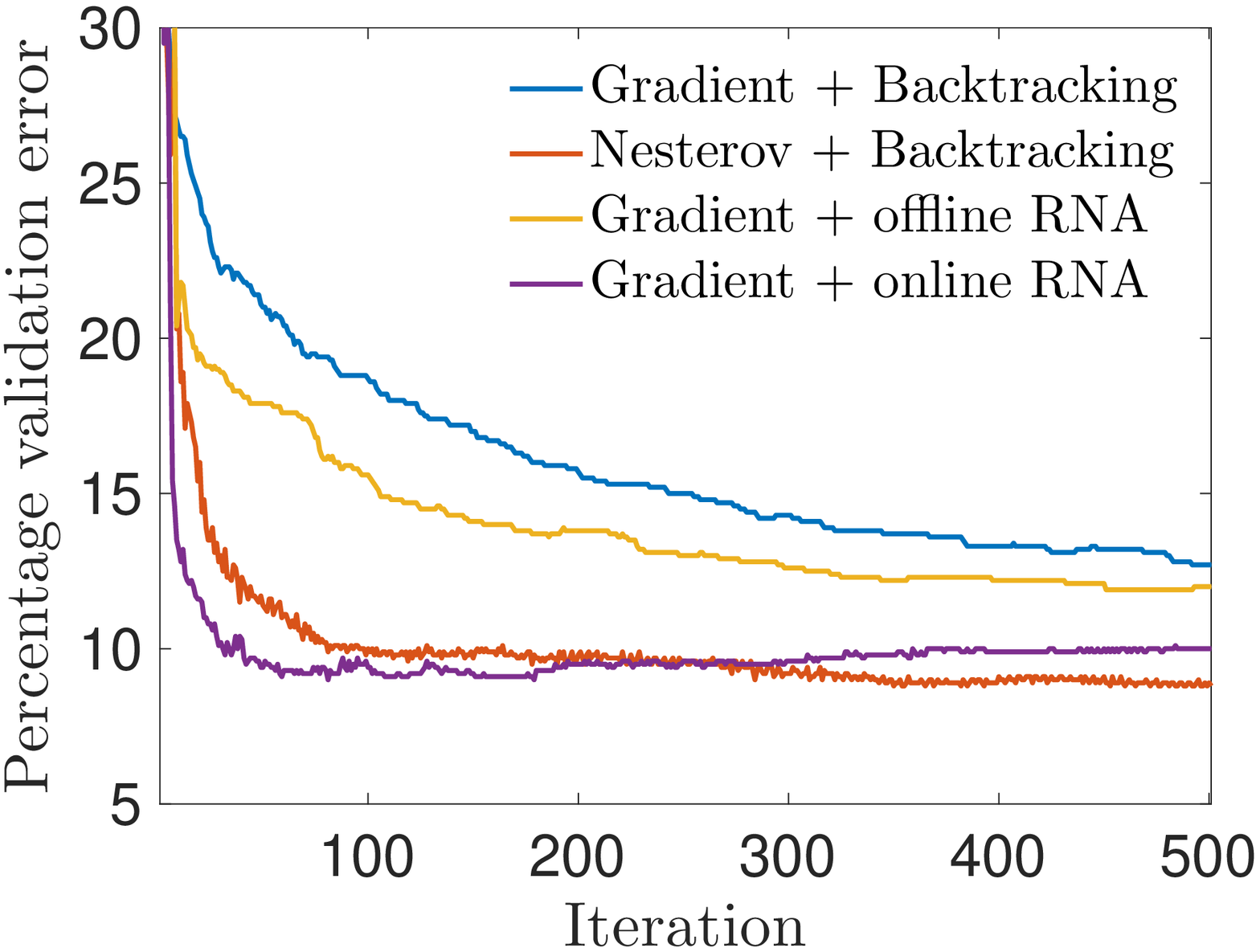}
\end{figure}
\begin{figure}[ht]
\centering
\includegraphics[width=0.3\textwidth]{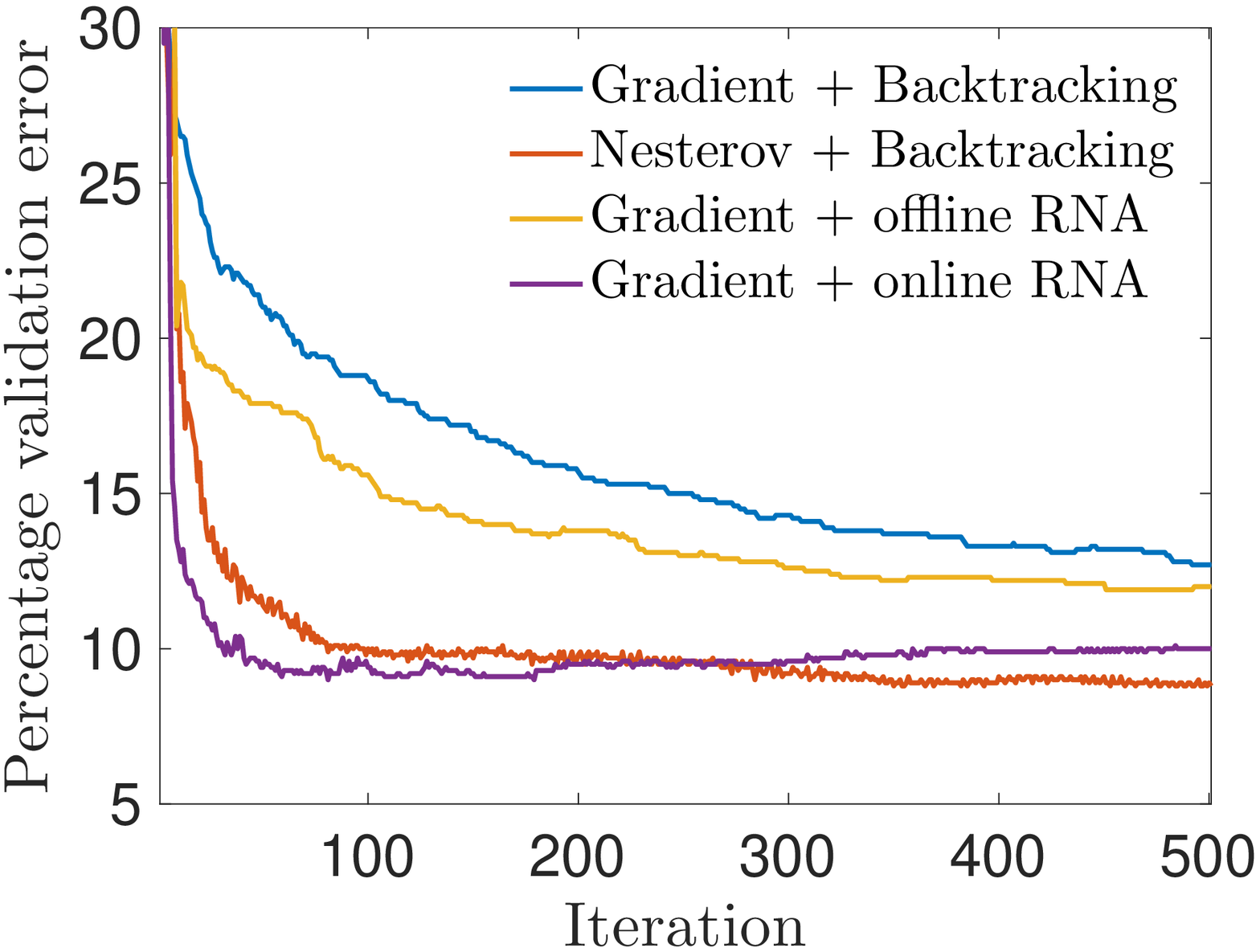}
\includegraphics[width=0.3\textwidth]{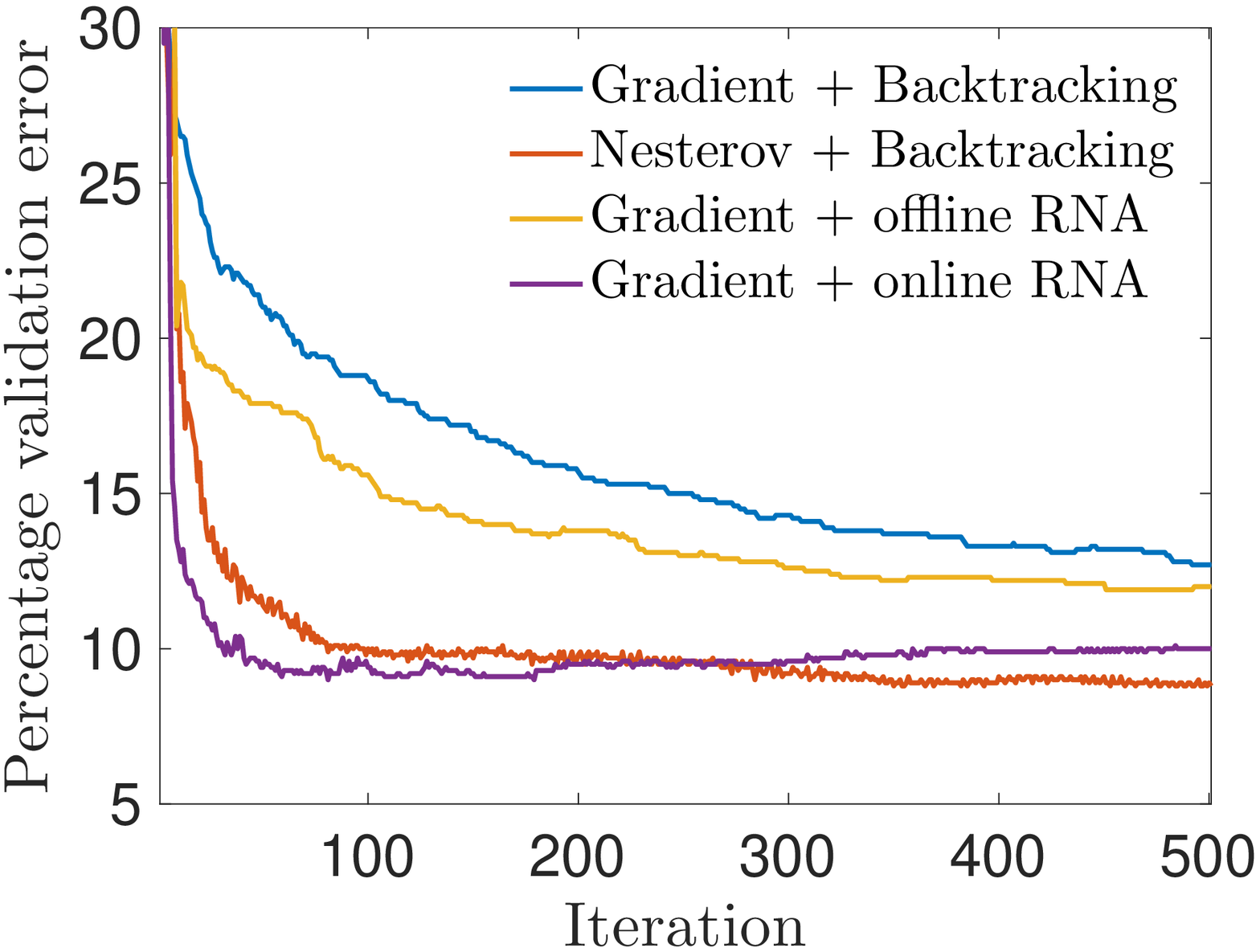}
\includegraphics[width=0.3\textwidth]{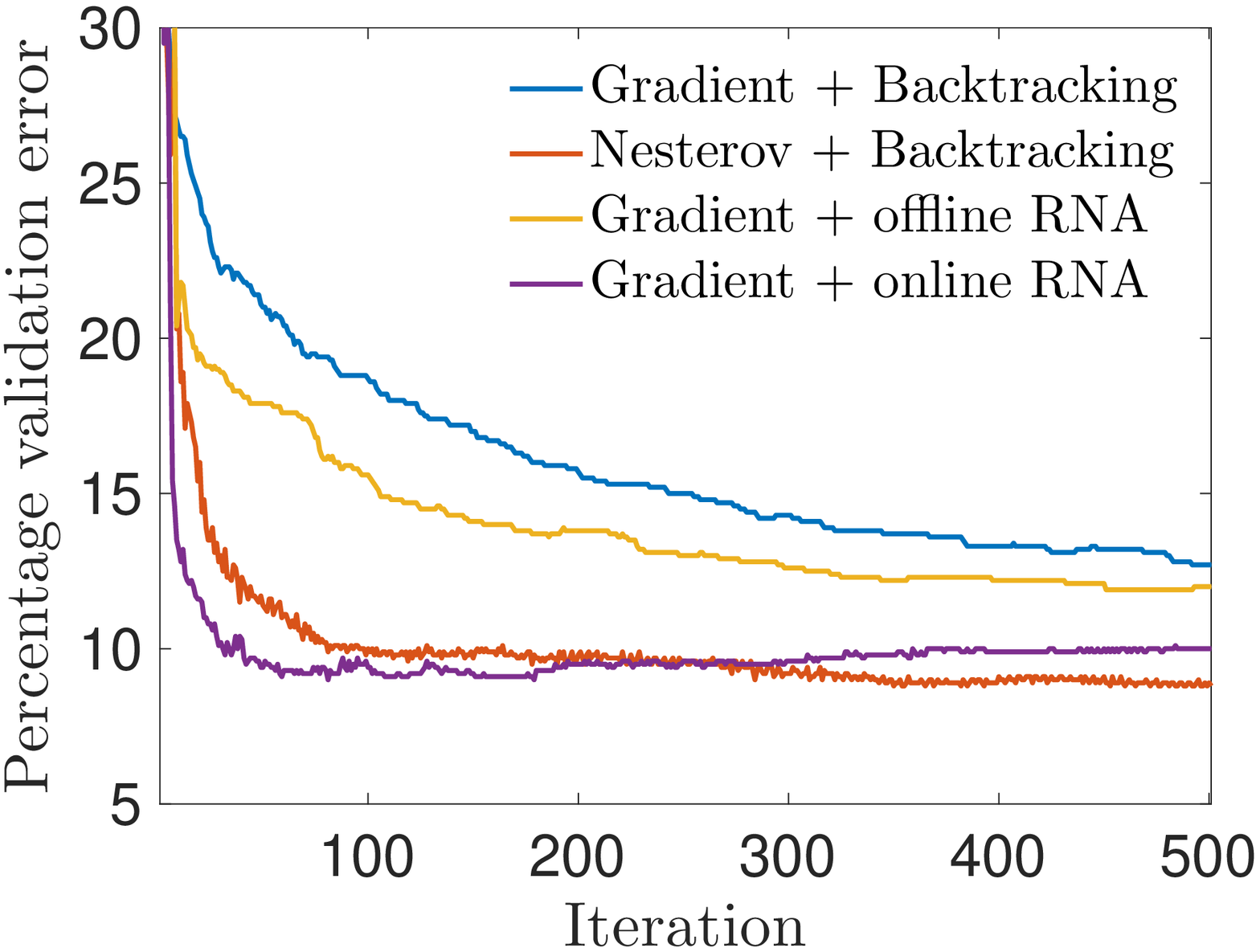}
\end{figure}
\begin{figure}[ht]
\centering
\includegraphics[width=0.3\textwidth]{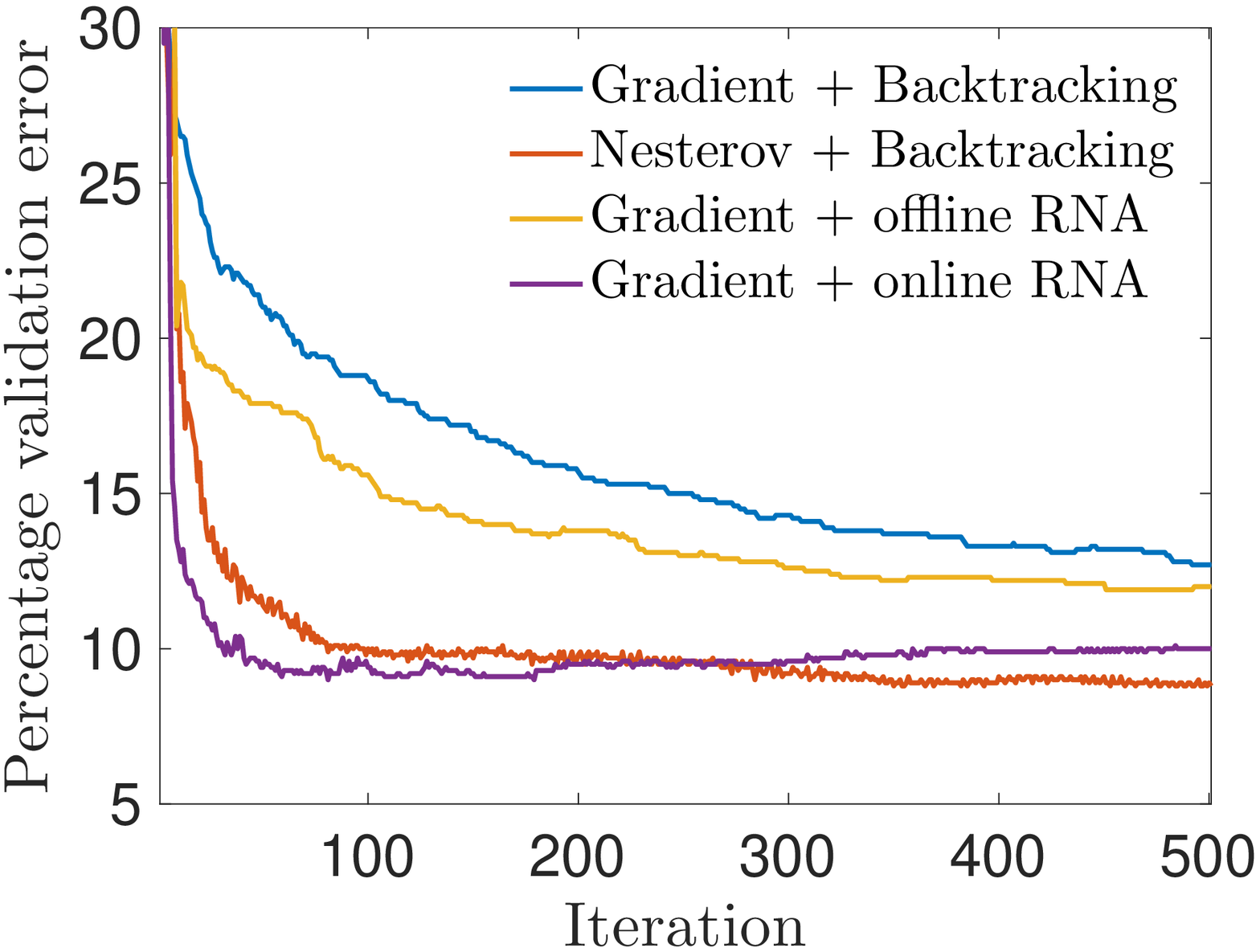}
\includegraphics[width=0.3\textwidth]{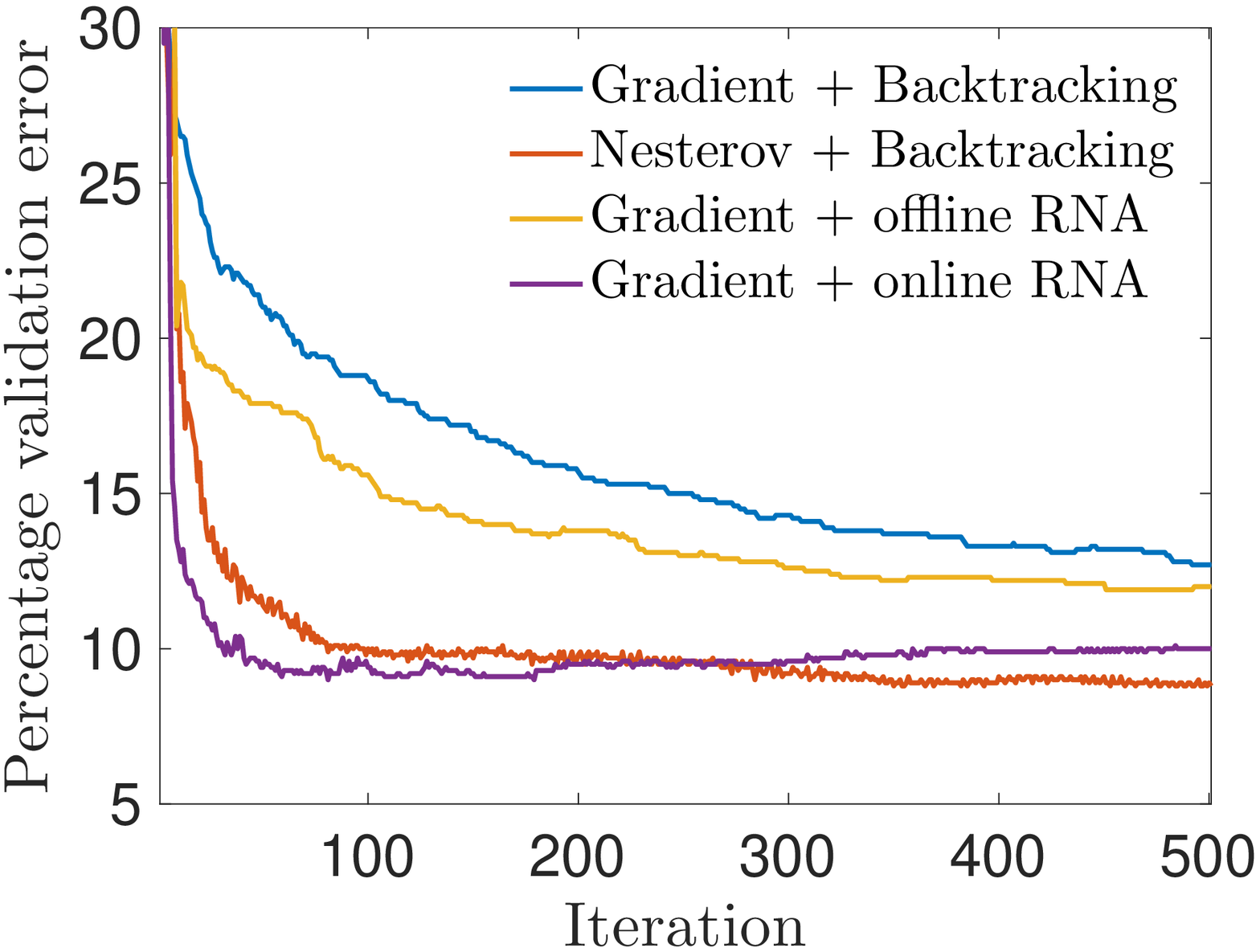}
\includegraphics[width=0.3\textwidth]{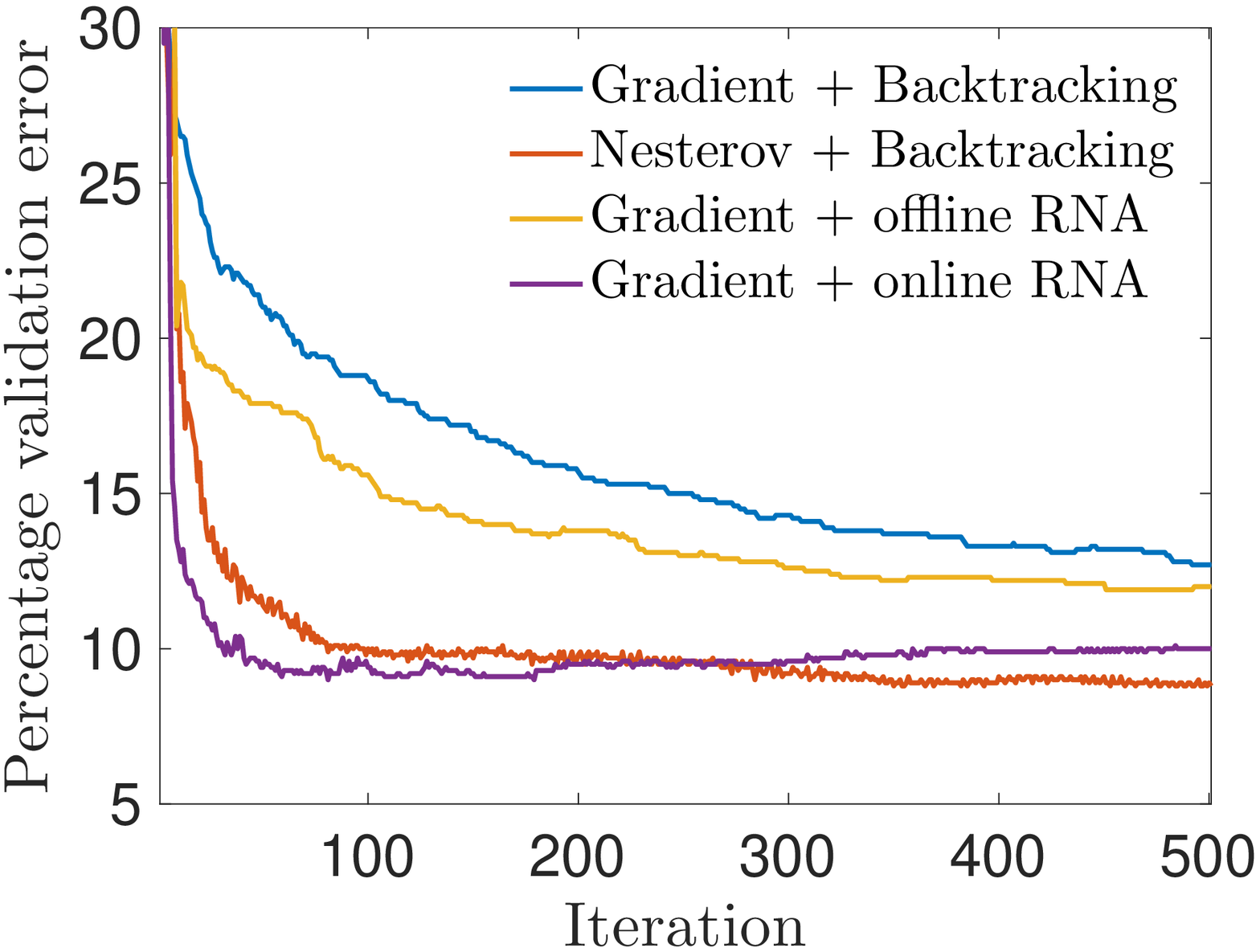}
\end{figure}
\begin{figure}[h!t]
\centering
\includegraphics[width=0.3\textwidth]{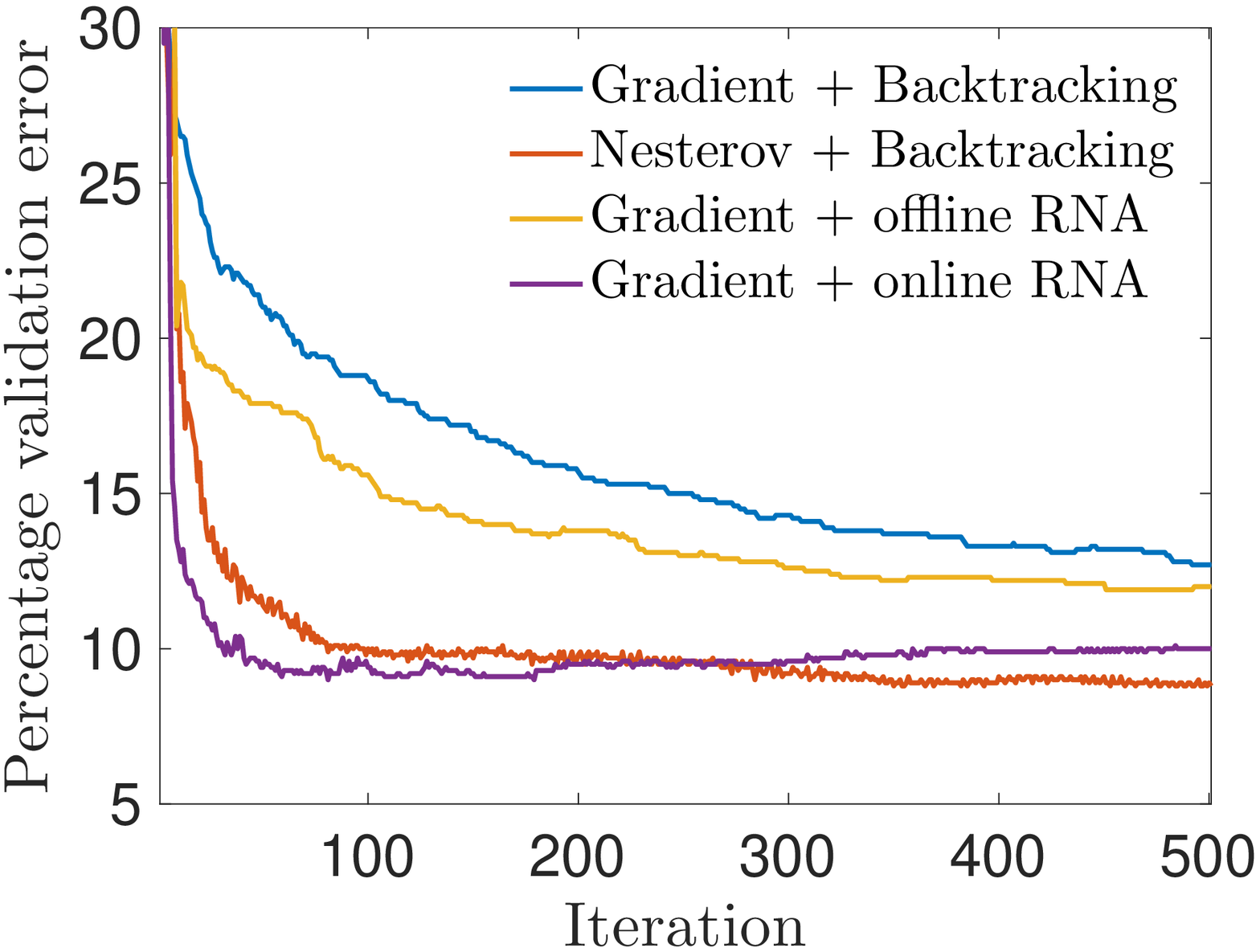}
\caption{From left to right, top to bottom: percentage validation error when optimizing a logistic regression (one vs all) on MNIST dataset for the numbers 1, 2, 3, 4, 5, 6, 7, 8, 9, 0.}
\end{figure}

\end{document}